\documentclass[11pt, reqno, twoside]{amsart}
\usepackage{amsmath,amssymb,amscd,amsthm,amsxtra,amsfonts}
\usepackage{mathtools,mathrsfs,dsfont,xparse}
\usepackage{graphicx,epstopdf}
\usepackage{multirow}
\usepackage{cases}
\usepackage{enumerate}
\usepackage{xcolor}
\usepackage{tikz,pgfplots}
\usetikzlibrary{matrix}
\usepackage{mathtools}
\usepackage[margin=1 in]{geometry}
\mathtoolsset{showonlyrefs}
\usepackage[numbers,sort]{natbib}
\usepackage[all]{xy}

\usepackage{multicol,bbm}
\usepackage[colorlinks=true, pdfstartview=FitV, linkcolor=blue, citecolor=blue, urlcolor=blue]{hyperref}

\allowdisplaybreaks[4]

\theoremstyle{plain}
\newtheorem{theorem}{Theorem}[section]
\newtheorem{proposition}[theorem]{Proposition}
\newtheorem{lemma}[theorem]{Lemma}
\newtheorem{corollary}[theorem]{Corollary}

\theoremstyle{definition}
\newtheorem{definition}[theorem]{Definition}
\newtheorem{remark}[theorem]{Remark}
\numberwithin{equation}{section}
\numberwithin{figure}{section}
\newtheorem{claim}[theorem]{Claim}

\newcommand{\F}{\mathscr{F}}
\newcommand{\G}{\mathbb{G}}

\newcommand{\0}{\mathbf{0}}
\newcommand{\h}{\mathbb{H}}
\newcommand{\M}{\mathcal{M}}

\newcommand{\NN}{\mathcal{N}}

\newcommand{\oo}{\mathscr{O}}
\newcommand{\C}{\mathbb{C}}
\newcommand{\N}{\mathbb{N}}
\newcommand{\R}{\mathbb{R}}

\newcommand{\T}{\mathbb{T}}

\DeclareMathOperator{\re}{Re}

\DeclareDocumentCommand{\abs}{s m}{
  \operatorname{}
  \IfBooleanTF{#1}{#2}{\left|#2\right|}}

\DeclareDocumentCommand{\norm}{s m}{
  \operatorname{}
  \IfBooleanTF{#1}{#2} {\left\| #2\right\|}}

\DeclareDocumentCommand{\inner}{s m}{
  \operatorname{}
  \IfBooleanTF{#1}{#2}{\left \langle#2\right \rangle}}

\DeclareDocumentCommand{\parenthese}{s m}{
  \operatorname{}
  \IfBooleanTF{#1}{#2}{\left(#2\right)}}

\DeclareDocumentCommand{\square}{s m}{
  \operatorname{}
  \IfBooleanTF{#1} {#2}{\left[#2\right]}}

\DeclareDocumentCommand{\bracket}{s m}{
  \operatorname{}
  \IfBooleanTF{#1}{#2}{\left\{#2\right\}}}

\pgfplotsset{compat=1.18}

\begin{document}
\title[On GWP of mass-critical NLS in $\h^3$]{Global well-posedness and scattering for the defocusing mass-critical Schr\"odinger equation in the three-dimensional hyperbolic space}

\author[Wilson and Yu]{Bobby Wilson  and Xueying Yu}

\address{Bobby Wilson
\newline \indent Department of Mathematics, University of Washington\indent 
\newline \indent  C138 Padelford Hall Box 354350, Seattle, WA 98195,\indent }
\email{blwilson@uw.edu}

\address{Xueying  Yu
\newline \indent Department of Mathematics, Oregon State University\indent 
\newline \indent  Kidder Hall 368, Corvallis, OR 97331 \indent 
\newline \indent 
And 
\newline \indent 
Department of Mathematics, University of Washington, 
\newline \indent  C138 Padelford Hall Box 354350, Seattle, WA 98195}
\email{xueying.yu@oregonstate.edu}

\subjclass[2020]{35Q55, 35P25, 35R01, 37K06, 37L50}

\keywords{Nonlinear Schr\"odinger equations, Global well-posedness, Scattering, Hyperbolic spaces}

\begin{abstract}
In this paper, we prove that the initial value problem for the mass-critical defocusing nonlinear Schr\"odinger equation on the three-dimensional hyperbolic space $\mathbb{H}^3$ is globally well-posed and scatters for data with radial symmetry in the critical space $L^2 (\mathbb{H}^3)$. 
\end{abstract}

\maketitle

\setcounter{tocdepth}{1}
\tableofcontents

\parindent = 10pt     
\parskip = 8pt

\section{Introduction}

In this paper, we consider the initial value problem for the defocusing mass-critical nonlinear Schr\"odinger (NLS) posed on the three-dimensional hyperbolic plane $\h^3$ with radial initial datum $\phi$:
\begin{align}\label{NLS}
\begin{cases}
i \partial_t u + \Delta_{\h^3} u = \abs{u}^{\frac{4}{3}} u, & t \in \R , \quad  x \in \h^3,\\
u(0,x) = \phi(x). & 
\end{cases}
\end{align}
The primary objective of this work is to establish the global well-posedness of equation \eqref{NLS} and demonstrate scattering behavior in the critical space $L^2(\mathbb{H}^3)$.

\subsection{Setup}
To provide a more general context for the problem, we consider the initial value problem for the NLS equation defined on a manifold $\mathcal{M}$:
\begin{align}\label{pNLS}
i \partial_t u + \Delta_{\mathcal{M}} u = \abs{u}^{p-1} u, 
\end{align}
where $u : \R_t \times \mathcal{M} \rightarrow \C$ is a complex-valued function of time and space and  $\M$ is a manifold. The NLS equation is characterized as {\it defocusing} due to the positive sign of the nonlinearity.

Equation \eqref{pNLS} possesses two essential conserved quantities: mass and energy, defined as follows:
\begin{align}
M(u)(t) & = \int_{\M} |u(t,x)|^2 \, dx = M(u(0)), \label{Mass}\\
E(u)(t) &  = \int_{\M} \, \frac{1}{2} |\nabla u(t,x)|^2 +  \frac{1}{p+1} |u(t,x)|^{p+1} \, dx = E(u(0)). \label{Energy}
\end{align}
These conservation laws provide control over the $L^2$ and $\dot{H}^1$ norms of the solutions, respectively.

In Euclidean spaces, NLS exhibits a scaling symmetry
\begin{align}
u(t,x) \mapsto \lambda^{\frac{2}{p-1}} u( \lambda^2 t , \lambda x) ,
\end{align}
under which the only invariant, homogeneous, $L_x^2$-based Sobolev norm  is the $\dot{H}^{s_c} (\R^d)$ norm. This symmetry establishes its critical scaling exponent of \eqref{NLS} on $\R^d$, which is given by
\begin{align}
s_c:=\frac{d}{2} - \frac{2}{p-1} . 
\end{align}
Accordingly, the problem NLS can be classified as {\it subcritical}, {\it critical} or {\it supercritical} depending on whether the regularity of the initial data is above, equal, or below the scaling $s_c$ of NLS. We will adopt the language in the scaling context in other manifolds $\M$.

\subsection{History}

In the cases when the equation becomes scale invariant at the level of one of the conserved quantities \eqref{Mass} and \eqref{Energy}, we refer to these situations as the mass-critical NLS  ($s _c =0$,\, $p= 1+ \frac{4}{d}$) and the energy-critical NLS ($s _c =1$, \, $p= 1+ \frac{4}{d-2}$) respectively, and they have received special attention in the past.  It has now become standard that within the critical regime, establishing a uniform {\it a priori} bound for the spacetime $L_{t,x}^{\frac{2(d+2)}{d-2s_c}}$ norm of solutions to the critical NLS implies both global well-posedness and scattering for general data.

In the energy-critical case ($s = s_c =1$), Bourgain \cite{Bour99} first introduced an inductive argument on the size of the energy and a refined Morawetz inequality to prove global existence and scattering in three dimensions for large finite energy data which is assumed to be radial. A different proof of the same result is given in \cite{Gr00}. 
Then, a breakthrough was made by Colliander-Keel-Staffilani-Takaoka-Tao \cite{CKSTT08}. They removed the radial assumption and proved global well-posedness and scattering of the energy-critical problem in three dimensions for general large data. They relied on Bourgain's induction on energy technique to find minimal blow-up solutions that concentrate in both physical and frequency spaces, and proved new interaction Morawetz-type estimates to rule out this kind of minimal blow-up solutions. This milestone was later extended to higher dimensions by Ryckman and Visan \cite{RV07} and Visan \cite{Visan07}, following the groundwork laid by \cite{CKSTT08}.

In \cite{KM06} Kenig and Merle proposed a new methodology, a deep and broad road map to tackle critical problems. In fact, using a contradiction argument they first proved the existence of a critical element such that the global well-posedness and scattering fail. Then, relying on a concentration compactness argument, they showed that this critical element enjoys a compactness property up to the symmetries of this equation. This final step was reduced to a rigidity theorem that precluded the existence of such a critical element.

The mass-critical ($s = s_c =0$) global well-posedness and scattering problem was also first studied in the radial case as in \cite{TVZ07, KTV09}. Then Dodson proved the global well-posedness of the mass-critical problem in any dimension for nonradial data \cite{Dod3d, Dod1d, Dod2d}. A key ingredient in Dodson's work is to prove a long-time Strichartz estimate. This estimate played a crucial role in handling the error term within frequency-localized Morawetz estimates, ultimately enabling the exclusion of minimal blow-up solutions.

In contrast to the energy- and mass-critical problems, for any other $s_c \neq 0,1$, there are no conserved quantities that control the growth in time of the $\dot{H}^{s_c}$ norm of the solutions. In \cite{KM10}, Kenig and Merle showed for the first time that if a solution of the defocusing cubic NLS in three dimensions remains bounded in the critical norm $\dot{H}^{\frac{1}{2}}$ in the maximal time of existence, then the interval of existence is infinite and the solution scatters using concentration compactness and rigidity argument. 
See also \cite{Mur14, Yu18, Dod22} for results of critical problems at the non-conservation law levels.

On non-Euclidean manifolds, using the Kenig and Merle road map and an ad hoc profile decomposition technique, Ionescu-Pausader \cite{IP1, IP2} and Ionescu-Pausader-Staffilani \cite{IPS12} were able to transfer the already available energy-critical global existence results in Euclidean spaces into their corresponding  $\R \times \T^3$,  $\T^3$ and $\h^3$ settings. Their method (which is known as {\it blackbox} trick) has been successfully applied to other general settings, see  \cite{Yue21, YY20, YYZ20, Stru15, PTW14, Zhao19, Zhao21}.
The central idea behind this approach involves breaking down the minimal blow-up solution into a combination of Euclidean-like solutions and scale-1 solutions through profile decomposition. Leveraging the well-established critical global well-posedness theory in Euclidean spaces, they employ it as a blackbox trick to achieve global well-posedness on the given manifold, effectively adapting and transferring the theory to non-Euclidean settings.

\subsection{Motivation}
In Euclidean spaces, where the sectional curvature is constant zero, the global wellposedness and scattering problem of NLS is well understood. However, in curved spaces with negative curvatures, the distinctive geometric properties introduced by the metric geometry pose unique obstacles, such as the lack of a Fourier convolution theorem. As a result, extrapolating results from the Euclidean to the hyperbolic case is often nontrivial. There are only very few results studying the global wellposedness and scattering of NLS in the hyperbolic case.

However, hyperbolic spaces represent the simplest symmetric spaces of noncompact type characterized by a constant negative sectional curvature. In the papers, \cite{IS09, AP09, BD07}, this negative curvature leads to dispersive estimates which are slightly improved compared to those that can be obtained in Euclidean spaces. Such enhanced dispersion has, in fact, facilitated the establishment of global well-posedness and scattering results for subcritical NLS in these spaces, see \cite{BD15,BCD09,BCS08,Ban07,IS09, SY,Ma}. 
In \cite{IPS12}, the authors establish global well-posedness and scattering for energy-critical NLS on $\h^3$.
In this paper, we aim to establish the global well-posedness and scattering theory for the mass-critical NLS on the three-dimensional hyperbolic space, which to the best of the authors' knowledge, is the first mass-critical global well-posedness and scattering result obtained on non-Euclidean manifolds.

\subsection{Main result and discussion}
Now we consider the initial value problem for the mass-critical \eqref{NLS} (that is $p=\frac{7}{3}$ in \eqref{pNLS}) NLS posed on the three-dimensional hyperbolic plane $\h^3$.

Now let us state the main result of this paper. 
     \begin{theorem}\label{thm mainthm} 
     Let $\phi \in L^2 (\h^3)$ and let $\phi$ be rotationally symmetric.
         \begin{enumerate}
             \item Then there exists a unique global solution $u \in C(\mathbb{R}; L^2(\h^3))$ of \eqref{NLS}. In addition, the mapping $\phi \to u$ is a continuous mapping from $L^2(\h^3)$ to $C(\mathbb{R}; L^2(\h^3))$ and $\|u\|_{L^2(\h^3)}$ is conserved.
             \item We have the bound of the global solution
                \begin{align}\label{eq 10/3}
                    \|u\|_{L^{10/3}_{t,x}(\R \times \h^3)} \lesssim_{\|\phi\|_{L^2(\h^3)}} 1
                \end{align}
            which implies $u$ scatters to a linear solution, that is, there exists $u_{\pm} \in L^2(\h^3)$ such that 
                \begin{align*}
                    \lim_{t \to \pm\infty}\|u(t)- e^{it \Delta_{\h^3}}u_{\pm}\|_{L^2(\h^3)}=0.
                \end{align*}
         \end{enumerate}
     \end{theorem}

\begin{remark}
Recall that it is sufficient to prove a uniform {\it a priori} bound for the spacetime $L_{t,x}^{10/3}$ norm of solutions in \eqref{eq 10/3}, as the scattering part follows from a standard argument (see, for example, \cite{Ca, tao2006nonlinear}).
\end{remark}

     One would expect that the stronger dispersion in hyperbolic space would play a role in a well-posedness argument. However, the most important observation for the following argument is that in three dimensions, one can reduce the bilinear Strichartz estimate in hyperbolic space with radial data to a Euclidean bilinear Strichartz estimate using a simple change-of-variables. For hyperbolic space of any dimension greater than 3, this change-of-variables does not produce a nice correspondence between Euclidean and hyperbolic bilinear Strichartz estimates. Since we have an assumption of radial data, we can use improved Euclidean radial bilinear Strichartz estimates of Shao \cite{Shao09}. It is important to note that, due to the presence of frequency localization in bilinear estimates, the most challenging part is to show that this frequency localization is preserved under the change of variables. If this preservation does not hold, the Euclidean bilinear estimate will not be applicable, potentially leading to undesirable error terms. The question of whether one can prove better general-data bilinear Strichartz estimates in hyperbolic space remains.

     For the scattering argument, one would expect a Morawetz estimate to be an essential feature of a contradiction argument. In this case, we use the Morawetz estimate for NLS on hyperbolic space of Ionescu and Staffilani \cite{IS09} combined with the strategy of Dodson \cite{Dod3d} to construct a suitable contradiction argument. The Morawetz estimate of \cite{IS09} was used to establish scattering for the energy-critical NLS on $\h^3$ in \cite{IPS12}. Of course, the difficulty in our setting is that we are assuming a measure of regularity that falls below the scaling regularity of the Morawetz estimate. To overcome this difficulty, we consider the frequency localized version of the Morawetz estimate. This is similar to the strategy employed in \cite{Dod3d}, which makes it a perfect strategy to emulate. Fortunately, the blackbox trick allows for a simplification of the strategy of Dodson.

The structure of the paper is as follows: Section \ref{sec Preliminaries} contains an assortment of important definitions and tools necessary for analysis of Schr\"odinger equations on Hyperbolic space. Then Section \ref{sec local} presents the local wellposedness statement as well as local stability estimates and the Morawetz estimate. In Section \ref{sec mainthm}, the proof of Theorem \ref{thm mainthm} is presented as a reduction to two key propositions: Proposition \ref{prop key} and Proposition \ref{prop MorDod}. Section \ref{sec bilinear} presents the bilinear Strichartz estimate (as well as a nonlinear corollary) followed by estimates that allow for Euclidean approximations in the profile decomposition in Section \ref{sec euclidean}. Section \ref{sec profile} details the profile decomposition machinery. Finally, Section \ref{sec Key} is a presentation of the proof of Proposition \ref{prop key}, and  Section \ref{sec MorDod} discusses the proof of Proposition \ref{prop MorDod}.

\subsection*{Acknowledgement}
Both authors would like to thank Gigliola Staffilani for suggesting this problem, and Sohrab Shahshahani for very insightful conversations.
B. W. is supported by NSF grant DMS 1856124, and NSF CAREER Fellowship, DMS 2142064. This material is based upon work supported by the National Science Foundation under Grant No. DMS-1928930 while B.W. was in residence at the Simons Laufer Mathematical Sciences Institute (formerly MSRI) in Berkeley, California, during the summer of 2023. 
X. Y. is partially supported by NSF DMS-2306429.


\section{Preliminaries}\label{sec Preliminaries}

In this section, we establish notation and provide a basic framework for understanding Schr\"odinger equations on hyperbolic spaces.
\subsection{Notations}
We define
\begin{align}
\norm{f}_{L_t^q L_x^r (J \times \h^3)} : = \square{\int_J \parenthese{\int_{\h^3} \abs{f(t,x)}^r \, dx}^{\frac{q}{r}} dt}^{\frac{1}{q}},
\end{align}
where $J$ is a time interval.

We adopt the usual notation that $A \lesssim  B$ or $B \gtrsim A$ to denote an estimate of the form $A \leq C B$, for some constant $0 < C < \infty$ depending only on the {\it a priori} fixed constants of the problem. Furthermore, we let $A \sim B$ denote the double estimate $A \lesssim B$ and $B \lesssim A$. We also use $a+$ and $a-$ to denote expressions of the form $a + \varepsilon$ and $a - \varepsilon$, for any $0 <  \varepsilon \ll 1$.

\subsection{Hyperbolic geometry}

We consider the Minkowski space $\R^{d+1}$ with the standard Minkowski metric
\begin{align}
- (dx^0)^2 + (dx^1)^2 + (dx^2)^2 + \cdots + (dx^d)^2 
\end{align}
and we define the bilinear form on $\R^{d+1} \times \R^{d+1}$,
\begin{align}
\square{x,y} = x^0 y^0 -x^1 y^1 -x^2 y^2 - \cdots -x^d y^d.
\end{align}
The hyperbolic space $\h^d$ is defined as
\begin{align}
\h^d = \{ x \in \R^{d+1} : \square{x,x} =1 \text{ and } x^0 > 0 \}.
\end{align}
Let $\0 =  \{1, 0_{\R^d} \} = \{(1,0,0, \cdots,0) \}$ denote the origin of $\h^d$. The Minkowski metric on $\R^{d+1}$ induces a Riemannian metric $g$ on $\h^d$, with covariant derivative $D$ and induced measure $d\mu$.

We define $\mathbb{G} : = SO(d,1) = SO_e (d,1)$ as the connected Lie group of $(d+1) \times (d+1)$ matrices that leave the form $[\cdot, \cdot]$ invariant. Clearly, $X \in SO (d,1)$ if and only if
\begin{align}
^{tr} X \cdot I_{d,1} \cdot X = I_{d,1}, \quad \det X=1, \quad X_{00} >0,
\end{align}
where $I_{d,1}$ is the diagonal matrix $\text{diag}[-1, 1, \ldots , 1]$ (since $[x,y] = - ^t x \cdot I_{d,1} \cdot y$). Let $\mathbb{K} = SO(d)$ denote the subgroup of $SO(d,1)$ that fixes the origin $\0$. Clearly,  $SO(d)$ is the compact rotation group acting on the variables $(x^1, \ldots, x^d)$. We define also the commutative subgroup $\mathbb{A}$ of $\mathbb{G}$,
\begin{align}
\mathbb{A} : = \bracket{ a_s = 
\begin{bmatrix}
\cosh s & \sinh s & 0\\
\sinh s & \cosh s & 0\\
0 & 0 & I_{d-1}
\end{bmatrix}
: s \in \R  } ,
\end{align}
and recall the Cartan decomposition
\begin{align}\label{eq Cartan}
\mathbb{G} = \mathbb{K} \mathbb{A}_+ \mathbb{K}, \quad \mathbb{A}_+ : = \{ a_s , s \in [0, \infty) \} .
\end{align}
The semisimple Lie group $\mathbb{G}$ acts transitively on $\h^d$ and hyperbolic space $\h^d$ can be identified with the homogeneous space $\mathbb{G}/ \mathbb{K} = SO (d,1) / SO(d)$. Moreover, for any $h \in SO (d,1)$ the mapping $L_h : \h^d \to \h^d$, $L_h (x) = h \cdot x$, defines an isometry of $\h^d$. Therefore, for any $h \in \mathbb{G}$, we further define the $L^2$ isometries
\begin{align}
\pi_h : L^2 (\h^d) \to L^2(\h^d), \quad \pi_h (f)(x) = f(h^{-1} \cdot x) .
\end{align}
We fix normalized coordinate charts which allow us to pass in a suitable way between functions defined on hyperbolic spaces and functions defined on Euclidean spaces. More precisely, for any $h \in SO (d,1)$ we define the diffeomorphism
\begin{align}\label{eq Diffeo}
\Psi_h : \R^d \to \h^d, \quad \Psi_h (v^1, \ldots , v^d) = h\cdot (\sqrt{1+\abs{v}^2} , v^1, \ldots , v^d) .
\end{align}
Using these diffeomorphisms we define, for any $h \in \mathbb{G}$,
\begin{align}
\widetilde{\pi}_h : C(\R^d) \to C(\h^d), \quad \widetilde{\pi}_h (f) (x) = f(\Psi_h^{-1} (x)) .
\end{align}
We will use the diffeomorphism $\Psi_{\mathcal{I}}$ as a global coordinate chart on $\h^d$, where $\mathcal{I}$ is the identity element of $\mathbb{G}$. We record the integration formula
\begin{align}
\int_{\h^d} f(x) \, d \mu(x) = \int_{\R^d} f(\Psi_{\mathcal{I}} (v)) (1 + \abs{v}^2)^{-\frac{1}{2}} \, dv
\end{align}
for any $f \in C_0(\h^d)$.

An alternative definition for the hyperbolic space is 
\begin{align}
\h^d = \{ x = (s, t) \in \R^{d+1} , (s, t) = ((\sinh r)\omega, \cosh r ), r \geq 0, \omega \in \mathbb{S}^{d-1} \}.
\end{align}
One has
\begin{align}
dt = \sinh r \, dr, \quad ds = \cosh r\omega \, dr + \sinh r \, d \omega
\end{align}
and the metric induced on $\h^d$ is
\begin{align}
dr^2 + \sinh^2 r \, d\omega^2 ,
\end{align}
where $d\omega^2$ is the metric on the sphere $\mathbb{S}^{d-1}$.

Then one can rewrite integrals as 
\begin{align}
\int_{\h^d} f(x) \, dx= \int_0^{\infty} \int_{\mathbb{S}^{d-1}} f(r, \omega) \sinh^{d-1} r \, dr d\omega .
\end{align}
The length of a curve
\begin{align}
\gamma(t) = (\cosh r(t) ,\sinh r(t) \omega(t)),
\end{align}
with $t$ varying from $a$ to $b$, is defined
\begin{align}
L(\gamma) = \int_a^b \sqrt{\abs{\gamma' (t)}^2 + \abs{\sinh r(t)}^2 \abs{\omega' (t)}^2} \, dt .
\end{align}

Recall $\0 = \{(1,0_{\R^d}) \}$ denote the origin of $\h^d$. The distance of a point to $\0$ is
\begin{align}
d((\cosh r, \sinh r\omega) , \0) =r.
\end{align}
More generally, the distance between two arbitrary points is
\begin{align}
d(x, x') = \cosh^{-1} ([x , x']).
\end{align}

The general definition of the Laplace-Beltrami operator is given by
\begin{align}
\Delta_{\h^d} = \partial_r^2 + (d-1)\frac{\cosh r}{\sinh r} \partial_r + \frac{1}{\sinh^2 r} \Delta_{\mathbb{S}^{d-1}} .
\end{align}

\begin{remark}
The form of the Laplace-Beltrami operator implies that there will be no scaling symmetry in $\h^3$ as we usually have in the $\R^d$ setting. 
\end{remark}


\subsection{Fourier transforms}
\subsubsection{Fourier transforms on $\h^d$}
For $\theta \in \mathbb{S}^{d-1}$ and $\lambda$ a real number, the functions of the type
\begin{align}
h_{\lambda,\theta} (x) = [x , \Lambda(\theta)]^{i\lambda -\frac{d-1}{2}},
\end{align}
where $\Lambda(\theta)$ denotes the point of $\R^{d+1}$ given by $(1, \theta)$, are generalized eigenfunctions of the Laplacian-Beltrami operator. Indeed, we have
\begin{align}
-\Delta_{\h^d} h_{\lambda, \theta} = \parenthese{\lambda^2 + \frac{(d-1)^2}{4}} h_{\lambda,\theta}.
\end{align}
The Fourier transform on $\h^d$ is defined as 
\begin{align}
\widehat{f}(\lambda ,\theta) := \int_{\h^d} h_{\lambda , \theta} (x) f(x) \, d x,
\end{align}
and the Fourier inversion formula on $\h^d$ takes the form of 
\begin{align}
f(x) = \int_{-\infty}^{\infty} \int_{\mathbb{S}^{d-1}} \overline{h}_{\lambda ,\theta} (x) \widehat{f}(\lambda ,\theta) \frac{d\theta d\lambda}{\abs{c(\lambda)}^2},
\end{align}
where $c(\lambda)$ is the Harish-Chandra coefficient
\begin{align}
\frac{1}{\abs{c(\lambda)}^2} = \frac{1}{2 (2 \pi)^d} \frac{\abs{\Gamma(i\lambda + \frac{d-1}{2})}^2}{\abs{\Gamma(i\lambda)}^2} .
\end{align}
In particular, when $d=3$, the Harish-Chandra coefficient is simple, that is,
\begin{align}
\abs{c(\lambda)}^2 = \frac{c}{\lambda^2} .
\end{align}

In the radially symmetric case, Fourier transform is given in the following form
\begin{align}
\widetilde{f}(\lambda) & = \int_0^{\infty} f(r) \phi_{\lambda} (r) \sinh^{d-1} r \, dr ,\\
f(r) & = \frac{2^{d-1}}{2 \pi \omega_{d-1}} \int_0^{\infty} \widetilde{f} (\lambda) \phi_{\lambda} (r) \abs{c(\lambda)}^{-2} \, d \lambda .
\end{align}
In particular, in three dimensional radially symmetric case, we have
\begin{align}
\phi_{\lambda} (r) = \frac{c}{\lambda} \frac{\sin (\lambda r)}{\sinh r}
\end{align}
which implies
\begin{align}
\widetilde{f} (\lambda) & = \frac{c}{\lambda} \int_0^{\infty} f(r) \sin (\lambda r)  \sinh r \, dr ,\\
f(r) & = c \int_0^{\infty} \widetilde{f} (\lambda) \frac{\sin (\lambda r)}{\lambda} \frac{1}{\sinh r} \lambda^2 \, d \lambda .
\end{align}
Also Plancherel formula reads
\begin{align}
\int_0^{\infty} \abs{f(r)}^2 \sinh^2 r \, dr = c \int_0^{\infty} \abs{\widetilde{f} (\lambda)}^2 \lambda^2 \, d\lambda .
\end{align}

\subsubsection{Radial Fourier transform on $\R^d$}
We define the Fourier transform on $\R^d$ by
\begin{align}
\widehat{f} (\xi) : = \frac{1}{(2\pi)^{\frac{d}{2}}} \int_{\R^d} e^{-ix \cdot \xi } f(x) \, dx ,
\end{align}
and Fourier inversion 
\begin{align}
f (x) : = \frac{1}{(2\pi)^{\frac{d}{2}}} \int_{\R^d} e^{ix \cdot \xi } \widehat{f}(\xi) \, d\xi .
\end{align}

Before defining the radial Fourier transforms, we recall Bessel functions and their properties. 
The Bessel function of order $n$, $J_n(x)$, is defined by
\begin{align*}
J_{n}(x) = \sum_{j=0}^{\infty} \frac{(-1)^j}{j! \, \Gamma (j+n +1)} \parenthese{\frac{x}{2}}^{2j+n}.
\end{align*}
In particular, $J_{\frac{1}{2}}$ has the following explicit formula
\begin{align}
J_{\frac{1}{2}} (z) = \sqrt{\frac{2}{\pi z}} \sin z .
\end{align}

In the radially symmetric case, the Fourier transform is given in terms of the Bessel function
\begin{align}
\widehat{f} (k) = (2 \pi)^{\frac{d}{2}} \int_0^{\infty} J_{\frac{d-2}{2}} (kr) f(r) r^{\frac{d}{2}} k^{-\frac{d-2}{2}} \, dr ,
\end{align}
and its inversion is given by
\begin{align}
f (r) = (2 \pi)^{\frac{d}{2}} \int_0^{\infty} J_{\frac{d-2}{2}} (kr) \widehat{f}(k) k^{\frac{d}{2}} r^{-\frac{d-2}{2}} \, dk .
\end{align}

\subsection{A change of variables between $\R^d$ and rotationally symmetric manifolds}\label{ssec COV}
In this subsection, we recall a change of variables computation for rotationally symmetric manifolds (see \cite{Pie08, BD07}). 

In the case of rotationally symmetric manifolds $\mathcal{M}$, the metric is given by
\begin{align}
dx^2 = dr^2 + \phi^2 (r) \, d\omega^2
\end{align}
where $d \omega^2$ is the metric on the sphere $\mathbb{S}^{d-1}$, and $\phi$ is a positive function $C([0,\infty))$, such that $\phi(0)=0$, $\phi'(0) = 1$, and $\phi^{(k)} (0)=0$ (for $k\in 2 \mathbb{Z}_+$). For example, $\R^d$ and $\h^d$ are such manifolds, with $\phi(r) = r$ and $\phi(r) = \sinh r$ respectively.
The Laplace-Beltrami operator on $\mathcal{M}$ is 
\begin{align}
\Delta_{\mathcal{M}} = \partial_r^2 + (d-1) \frac{\phi'(r)}{\phi(r)} \partial_r + \frac{1}{\phi^2 (r)} \Delta_{\mathbb{S}^{d-1}} .
\end{align}

Consider the linear Schr\"odinger equation posed on $\mathcal{M}$
\begin{align}
\begin{cases}
i \partial_t u + \Delta_{\mathcal{M}} u = 0 ,\\
u(0,x) = u_0.
\end{cases}
\end{align}
We define an auxiliary function
\begin{align}
k(r) := \parenthese{\frac{\phi(r)}{r}}^{\frac{d-1}{2}}.
\end{align}
Under the following change of variables
\begin{align}
u(t,r, \omega) = \frac{v(t,r,\omega)}{k(r)},
\end{align}
we see that $v$ solves the equation
\begin{align}
i \partial_t v + \Delta_{\R^d} v + \parenthese{\frac{1}{\phi^2(r)} - \frac{1}{r^2}} \Delta_{\mathbb{S}^{d-1}} v - V(r) v = 0 ,
\end{align}
where
\begin{align}
V(r) = \frac{d-1}{2} \frac{\phi''}{\phi} + \frac{(d-1)(d-3)}{4} \parenthese{ \parenthese{\frac{\phi'}{\phi}}^2 - \frac{1}{r^2}}.
\end{align}

In particular, when considering on $\h^3$ with radial data, we simplify to obtain
\begin{align}
u (t,r,w) = u(t,r), \qquad k(r) = \frac{\sinh r}{r}, \qquad V(r) = 1
\end{align}
and we can see that
\begin{align}
v(t,r) = \frac{\sinh r}{r} u(t,r)
\end{align}
solves
\begin{align}
i \partial_t v + \Delta_{\R^3} v - v =0 .
\end{align}
Now let 
\begin{align}
w (t,r) = e^{it} v(t,r) ,
\end{align}
then we see that $w$ solves the linear Schr\"odinger equation
\begin{align}
i \partial_t w + \Delta_{\R^3} w  =0  .
\end{align}

It is clear that the radial, 3-dimensional regime is a very special case of such a change of variables in rotationally symmetric manifolds. We will rely heavily on the simplicity of relating the generic rotationally symmetric case to the Euclidean case throughout the course of this manuscript.

\subsection{Strichartz estimates} 
In this subsection, we recall the Strichartz estimates proved in both Euclidean spaces and hyperbolic spaces.
\subsubsection{In Euclidean spaces}
We say that a couple $(q,r)$ is admissible if $(1/q,1/r)$ belong to the line
\begin{align}
 L_d=\{(\frac{1}{q},\frac{1}{r}) \in [0,\frac{1}{2}] \times ( 0,\frac{1}{2}] \, \big| \, \frac{2}{q} + \frac{d}{r} = \frac{d}{2} \}.
\end{align} 
Then we have the following  
\begin{proposition}[Euclidean Strichartz estimates in \cite{GV92, KT98, Yaji87}] 
Assume $u$ is the solution to the inhomogeneous initial value problem
\begin{align}
\begin{cases}
i \partial_t u + \Delta_{\mathbb{R}^d} u = F, & t \in \R , \quad x \in \R^d,\\
u(0,x) = \phi(x), & 
\end{cases}
\end{align}
For any admissible exponents $(q,r)$ and $(\widetilde{q}, \widetilde{r})$ we have the Strichartz estimates:
\begin{align}
\norm{u}_{L_t^q L_x^r(\R \times \R^d)} \lesssim \norm{\phi}_{L_x^2(\R^d)} + \norm{F}_{L_t^{\widetilde{q}'}  L_x^{\widetilde{r}'} (\R \times \R^d)} ,
\end{align}
where $1/q +1/q' =1$ and $1/r +1/r' =1$.
\end{proposition}

\begin{definition}[Strichartz spaces in $\R^d$]
We define the Banach space
\begin{align}
S_{\R^d}^0 (I) = \bracket{f \in C(I ; L^2(\R^d)) : \norm{f}_{S_{\R^d}^0 (I)} = \sup_{(q,r) \text{ admissible }} \norm{f}_{L_t^q L_x^r (I  \times \R^d)} < \infty} .
\end{align}
\end{definition}

\subsubsection{In hyperbolic spaces}
We have a larger class of admissible pairs. We say that a couple $(q,r)$ is admissible if $(1/q,1/r)$ belong to the triangle
\begin{align}
T_d = \bracket{ (\frac{1}{q},\frac{1}{r}) \in (0,\frac{1}{2}] \times ( 0,\frac{1}{2}) \, \big| \, \frac{2}{q} + \frac{d}{r} \geq \frac{d}{2} } \cup \bracket{ (0,\frac{1}{2})} .
\end{align}
We then have the following theorem:
\begin{proposition}[Hyperbolic Strichartz estimates in \cite{AP09, IS09}]
Assume $u$ is the solution to the inhomogeneous initial value problem
\begin{align}
\begin{cases}
i \partial_t u + \Delta_{\h^d} u = F, & t \in \R , \quad x \in \h^d,\\
u(0,x) = \phi(x). & 
\end{cases}
\end{align}
Then, for any admissible exponents $(q,r)$ and $(\widetilde{q}, \widetilde{r})$ we have the Strichartz estimates:
\begin{align}
\norm{u}_{L_t^q L_x^r (\R \times \h^d)} \lesssim \norm{\phi}_{L_x^2 (\h^d)} + \norm{F}_{L_t^{\widetilde{q}'}  L_x^{\widetilde{r}'} (\R \times \h^d)} ,
\end{align}
where $1/q +1/q' =1$ and $1/r +1/r' =1$.
\end{proposition}
Note that the main inequality we need is the dispersive estimate
\begin{align}\label{eq Dispersive}
\norm{e^{it\Delta_{\h^d}}}_{L^p \to L^{p'}} \lesssim \abs{t}^{-d(\frac{1}{p} - \frac{1}{2})}, \quad p \in [\frac{2d}{(d+2)} , 2], \quad \frac{1}{p} + \frac{1}{p'} = 1 .
\end{align}

\begin{remark}
Strichartz estimates are better in $\h^d$ in the sense that the set $T_d$ of admissible pairs for $\h^d$ is much wider than the corresponding set, $L_d$, for $\R^d$ which is just the lower edge of the triangle. See also Figure \ref{figure} below. 
\end{remark}

\begin{figure}[!htp]
\begin{center}
\begin{tikzpicture}
\draw[->] (0, 0) -- (6, 0) node[right] {\textit{{\large $\frac{1}{q}$}}};
\draw[->] (0, 0) -- (0, 5) node[above] {\textit{{\large $\frac{1}{r}$}}};
\draw[scale=1, domain=0:5, smooth, variable=\x, black] plot ({\x}, {4-0.8*\x});
\draw[scale=1, domain=0:4, smooth, variable=\x, black] plot ({\x}, {4});
\draw[scale=1, domain=0:4, smooth, variable=\y, black]  plot ({4}, {\y});
\draw (4,0)  node [below, fill=white] {$\frac{1}{2}$};
\draw (0,4)  node [left, fill=white] {$\frac{1}{2}$};
\draw (5,0)  node [below, fill=white] {$\frac{d}{4}$};
\draw (0,0)  node [below, fill=white] {$0$};    
\fill[blue,opacity=0.3] (0,4)  --++ (4,0) --++  (0,-3.2) --++  (-4,3.2)-| cycle; 
\end{tikzpicture}    
\end{center}
\caption{Strichartz admissible pair regions for the hyperbolic space $\h^d$.}\label{figure}
\end{figure}
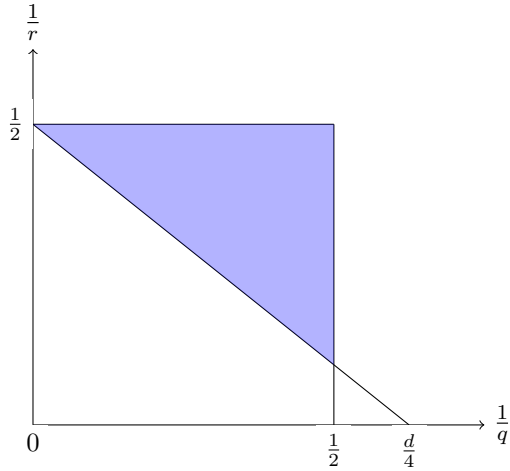

\begin{definition}[Strichartz spaces]\label{defn Strichartz spaces}
We define the Banach space
\begin{align}
S_{\h^d}^0 (I) = \bracket{f \in C(I ; L^2(\h^d)) : \norm{f}_{S_{\h^d}^0 (I)} = \sup_{(q,r) \text{ admissible }} \norm{f}_{L_t^q L_x^r (I  \times \h^d)} < \infty} .
\end{align}
Also we define the dual Banach space $N_{\h^d}^{0} (I)$ of $S_{\h^d}^0 (I)$,
\begin{align}
N_{\h^d}^0 (I) = \bracket{f \in C(I ; L^2(\h^d))  : \norm{f}_{N_{\h^d}^0(I)} := \inf_{(q,r) \text{ admissible } } \norm{f}_{L_{t}^{q'} L_x^{r'} (I \times \h^d)} < \infty} ,
\end{align}
where $1/q +1/q' =1$ and $1/r +1/r' =1$.
\end{definition}

\subsection{Tools needed on $\h^3$} In this subsection we recall some important and classical analysis developed for the hyperbolic spaces.
\subsubsection{Local smoothing estimates in the hyperbolic space}\label{lem H LSE}
\begin{proposition}[Theorem 1.2 in \cite{Kaiz14}: Local Smoothing Estimates in $\h^d$]\label{thm Smoothing}
For any $\varepsilon >0$,
\begin{align}
& \norm{\inner{x}^{-\frac{1}{2} -\varepsilon} \abs{\nabla}^{\frac{1}{2}} e^{it\Delta} f}_{L_{t,x}^2 (\R \times \h^d)} \lesssim \norm{f}_{L_x^2(\h^d)} ,\\
& \norm{\inner{x}^{-\frac{1}{2} -\varepsilon} \nabla \int_0^t e^{i(t-s) \Delta} F(s,x) \, ds}_{L_{t,x}^2 (\R \times \h^d)} \lesssim \norm{\inner{x}^{\frac{1}{2} +\varepsilon}  F}_{L_{t,x}^2 (\R \times \h^d)} .
\end{align}
where the Japanese bracket notation is given by $\inner{x} = (1 + d (x, \0)^2)^{\frac{1}{2}}$.
\end{proposition}
\begin{remark}
In \cite{Kaiz14}, the author considered more general manifolds that there are denoted with $X$. To obtain the theorem above one needs to take  $p(\lambda) = \abs{\lambda}^2$, $p(D) = -\Delta_{X} - \abs{\rho}^2$ and $m=2$. 
\end{remark}

\subsubsection{Heat-flow-based Littlewood-Paley projections and functional inequalities on $\h^3$}

The Little\allowbreak wood-Paley projections on $\h^3$ that we use in this paper are based on the linear heat propagator $e^{s\Delta}$. It turns out, in fact, that in $\h^d$  this is a great substitute for the standard Littlewood-Paley decomposition used in $\R^d$, since in $\h^d$ one cannot localize in frequencies efficiently. We report below several results that first appeared in \cite{LLOS}.
\begin{definition}[Section 2.7.1 in \cite{LLOS}: Heat-flow-based Littlewood-Paley projections]\label{defn LP}
For any $N> 0$, we define
\begin{align}
P_{\leq N} f = e^{N^{-2}\Delta_{\h^3}} f , \quad P_N f = N^{-2}\Delta_{\h^3} e^{N^{-2}\Delta_{\h^3}} f .
\end{align}
By the fundamental theorem of calculus, it is straightforward to verify that
\begin{align}
P_{\leq N} f = 2\int_0^{N} P_{M} f \,\frac{dM}{M} \quad \text{ for } N >0.
\end{align}
In particular, we have
\begin{align}
f = 2\int_0^{\infty} P_{M} f \, \frac{dM}{M} ,
\end{align}
which is the basic identity that relates $f$ with its Littlewood-Paley resolution $\{P_N f\}_{N \in (0,\infty)}$. We also have
\begin{align}
P_{\geq N} f = 2\int_N^{\infty} P_{M} f \, \frac{dM}{M} .
\end{align}
\end{definition}
\begin{remark}
Intuitively, $P_N f$ may be interpreted as a projection of $f$ to frequencies comparable to $N$. $P_{\leq N}$ and $P_{\geq N}$ can be viewed as projections into low and high frequencies, respectively.
\end{remark}

\begin{proposition}[Theorem 3.2 in \cite{hebey2000nonlinear}]
 Let $(\mathcal{M}, \mathbf{g})$ be a smooth, complete Riemannian $d$-manifold with Ricci curvature bounded from below. Assume that
    \begin{align*}
        \inf _{x \in \mathcal{M}} \operatorname{Vol}_{\mathbf{g}} \left(B_x(1)\right)>0
    \end{align*}
where $\operatorname{Vol}_{\mathbf{g}} \left(B_x(1)\right)$ stands for the volume of $B_x(1)$ with respect to $\mathbf{g}$. Then for any $q \in [1, d)$, 
    \begin{align*}
        W^{1, p}(\mathcal{M})  \hookrightarrow  L^q(\mathcal{M})
    \end{align*} 
    where $1 / q=1 / p-1 / d$.
\end{proposition}

\begin{lemma}[Sobolev embedding]\label{lem Sobolev}
\begin{align}
W^{s,p} (\h^d) \hookrightarrow L^q (\h^d), \quad \text{if } 1 < p \leq q < \infty \text{ and } \frac{1}{p}-\frac{1}{q} =\frac{s}{d}.
\end{align}
\end{lemma}

\begin{lemma}[Refined Fatou]\label{lem Fatou}
Suppose $\{ f_n \} \subset L^p (\h^d)$ with $\limsup_{n \to \infty} \norm{f_n}_{L^p} < \infty$. If $f_n \to f$ almost everywhere, then
\begin{align}
\int_{\h^d} \abs{\abs{f_n}^p - \abs{f_n - f}^p - \abs{f}^p} \, dx \to 0 .
\end{align}
In particular, 
\begin{align}
\norm{f_n}_{L^p}^p - \norm{f_n - f}_{L^p}^p \to \norm{f}_{L^p}^p .
\end{align}
\end{lemma}

\section{Local Theory, Stability and Morawetz Estimates}\label{sec local}
In this section, we include a local well-posedness argument, a stability theory, and Morawetz estimates. 
\begin{definition}[The partial ordering on trajectories]\label{def order}
   Let
\begin{align}
\mathscr{P} = \{ (I, u) : I \subset \R \text{ is an open interval and } u \in C (I ; L^2(\h^3))\}
\end{align}
with the natural partial order
\begin{align}
(I, u) \leq (I' , u') \text{ if and only if } I \subset I' \text{ and } u'(t) = u(t) \text{ for any } t \in I.
\end{align} 
\end{definition}

We denote the solution norm 
\begin{align}
\norm{f}_{Z(I)} := \norm{f}_{L_{t,x}^{\frac{10}{3}} (I \times \h^3)}.
\end{align}

\begin{proposition}[Local well-posedness]\label{prop LWP}
Assume $\phi \in L^2 (\h^3)$. Then there is a unique maximal solution $(I, u) = (I(\phi) , u(\phi)) \in \mathscr{P}, 0 \in I $, of the initial-value problem \eqref{NLS}
on $I \times \h^3$. The mass defined in \eqref{Mass} is constant on $I$, and $\norm{u}_{S_{\h^3}^0 (J)} < \infty$ for any compact interval $J \subset I$. In addition,
\begin{align}
& \norm{u}_{Z(I_+)} = \infty \quad \text{if} \quad I_+ : = I \cap [0, \infty) \text{ is bounded}, \\
& \norm{u}_{Z(I_-)} = \infty \quad \text{if} \quad I_- : = I \cap (- \infty, 0] \text{ is bounded}.
\end{align}
\end{proposition}

\begin{proposition}[Stability]\label{prop Stab}
Assume $I$ is an open interval, $\rho \in \{0, 1\}$, and $v \in C (I ; L^2 (\h^3))$ satisfies the approximate Schr\"odinger equation
\begin{align}\label{eq PerNLS}
i \partial_t v + \Delta_{\h^3} v = \rho \abs{v}^{\frac{4}{3}} v + e
\end{align}
on $I \times \h^3$. Assume in addition that
\begin{align}\label{eq Stab1}
\norm{v}_{L_{t,x}^{\frac{10}{3}} (I \times \h^3)} + \norm{v}_{L_t^{\infty}L_x^2 (I \times \h^3)} \leq M, 
\end{align}
for some $M \in [0 , \infty)$. Assume $t_0 \in I$ and $u(t_0) \in L^2 (\h^3)$ is such that the smallness condition
\begin{align}
\norm{u(t_0) - v (t_0)}_{L^2 (\h^3)} + \norm{e}_{N_{\h^3}^0 (I)} \leq \varepsilon
\end{align}
holds for some $0 < \varepsilon < \varepsilon_1$, where $\varepsilon_1 \leq 1$ is a small constant $\varepsilon_1 = \varepsilon_1 (M) > 0$.
Then there exists a solution $u \in C (I ; L^2 (\h^3))$ of the Schr\"odinger equation
\begin{align}
i \partial_t u + \Delta_{\h^3} u = \rho \abs{u}^{\frac{4}{3}} u 
\end{align}
on $I \times \h^3$, and 
\begin{align}
\norm{u}_{S_{\h^3}^0 (I \times \h^3)} + \norm{v}_{S_{\h^3}^0 (I \times \h^3)}  \leq C(M) ,\\
\norm{u -v}_{S_{\h^3}^0 (I \times \h^3)}  \leq C(M) \varepsilon .
\end{align}
\end{proposition}

\begin{proposition}[Morawetz estimates in \cite{IS09}]
Assume that $I \subset \R$ is an open interval, and $u \in C(I; L^2 (\h^3))$ is a solution of \eqref{NLS}. Then for any $t_1 ,t_2 \in I$
\begin{align}
\norm{u}_{L_{t,x}^{\frac{10}{3}}([t_1,t_2] \times \h^3)}^{\frac{10}{3}}  \lesssim \norm{u}_{L_t^{\infty} L_x^2 ([t_1,t_2] \times \h^3)} \norm{u}_{L_t^{\infty} H_x^1 ([t_1,t_2]\times \h^3)} .
\end{align}
\end{proposition}

\begin{proposition}[Modified Morawetz estimates, Proposition 4.1 in \cite{SY}]\label{prop Morawetz}
If  $u$ solves
\begin{align}\label{eq modNLS}
i \partial_t u + \Delta_{\h^3} u = \abs{u}^{\frac{4}{3}} u + \NN ,
\end{align}
then the modified Morawetz estimate becomes
\begin{align}
\begin{aligned}\label{modmor}
\norm{u}_{L_{t,x}^{\frac{10}{3}}([t_1,t_2] \times \h^2)}^{\frac{10}{3}} & \lesssim \norm{u}_{L_t^{\infty} L_x^2 ([t_1,t_2] \times \h^3)} \norm{u}_{L_t^{\infty} H_x^1 ([t_1,t_2] \times \h^3)}  \\
& \quad + \norm{\NN\bar{u}}_{L_{t,x}^1 ([t_1,t_2] \times \h^3)} + \norm{\NN \nabla \bar{ u}}_{L_{t,x}^1 ([t_1,t_2] \times \h^3)} .
\end{aligned}
\end{align}
\end{proposition}


\section{The Proof of the Main Theorem}\label{sec mainthm}

In this section, we present the proof of Theorem \ref{thm mainthm} by reducing the statement to Proposition \ref{prop key}. The remainder of the paper is devoted to proving Proposition \ref{prop key} and other important results used throughout this section. 

First, for any $M \in [0, \infty)$, define
    \begin{align*}
        S(M):= \sup\left\{ \|u\|_{Z(I)}~:~ \|u\|_{L^2(\h^3)} \leq M \right\}
    \end{align*}
where the supremum is taken over all solutions $u \in C(I; L^2(\h^3))$ to \eqref{NLS} defined on the interval $I$. We further define 
    \begin{align}\label{eq mmax}
        M_{\max}:= \sup\{ M~:~ S(M)< \infty\} .
    \end{align}
Stability at the trivial solution implies that $M_{\max}>0$. Proposition \ref{prop LWP} then implies that if $\|u\|_{L^2(\h^3)} < M_{\max}$ then $u$ exists globally and scatters. Now, if $M_{\max}= \infty$, then Theorem \ref{thm mainthm} is proven. Therefore, we assume, by contradiction, that $M_{\max}< \infty$.

If $M_{\max}< \infty$, then we can construct a sequence of functions, $u_k \in C((-T_k, T^k); L^2(\h^3))$ such that the hypotheses of the following proposition hold:

\begin{proposition}[Key proposition]\label{prop key}
For $k= 1,2 , \cdots$, let $(-T_k, T^k)\subset \mathbb{R}$ be a sequence of intervals and let $u_k \in C( ( -T_k, T^k) ; L^2 (\h^3))$,  be a sequence of  solutions of the nonlinear equation 
 \begin{align}
i \partial_t u + \Delta_{\h^3} u = \abs{u}^{\frac{4}{3}} u, 
\end{align}
 such that $M(u_k) \to M_{\max}$. Let $t_k \in (-T_k, T^k) $ be a sequence of times with
\begin{align}\label{eq Contradiction}
\lim_{k \to \infty} \norm{u_k}_{Z((-T_k, t_k))} = \lim_{k \to \infty} \norm{u_k}_{Z((t_k , T^k))} = + \infty .
\end{align}
Then there exists $w_0 \in  L^2 (\h^3)$ and a sequence of isometries $h_k \in \G$ such that, up to passing to a subsequence, $u_k (t_k, h_k^{-1} \cdot x) \to w_0 \in L^2$ strongly.
\end{proposition}
(The proof of Proposition \ref{prop key} is presented in Section \ref{sec Key}.)

Let $u \in C((-T_*, T^*); L^2(\h^3))$ be the maximal solution to \eqref{NLS} with initial data $w_0$. If $\|u\|_{Z((0, T^*))}< \infty$, then stability and the fact that $u$ is maximal implies that $T^*=\infty$ and
    \begin{align*}
        \sup_k\|u_k\|_{Z((t_k, \infty))} \leq C_{\|u\|_{Z((0, \infty))}} ,
    \end{align*}
which presents a contradiction. Therefore, $\|u\|_{Z((0, T^*))}= \infty$. Local wellposedness (Proposition \ref{prop LWP}) now implies that $u$ can be extended to a global solution, $u \in C(\R, L^2(\h^3))$, such that $\|u\|_{L^2(\h^3)}=M_{\max}$ and 
    \begin{align*}
        \|u\|_{Z((0, \infty))}=\|u\|_{Z((-\infty, 0))}=\infty.
    \end{align*}

Proposition \ref{prop key} (with $u_k=u$ for all $k$) and Arzel\`a-Ascoli theorem now implies that  $u$ can be assumed to be almost periodic modulo $\mathbb{G}$. Here, the notion of almost periodicity modulo $\mathbb{G}$ is defined in the following sense:
    \begin{definition}[Almost Periodic Modulo $\mathbb{G}$]\label{defn AP solution}
         A solution $u$ to \eqref{NLS} with lifespan $I$ is said to be almost periodic modulo $\mathbb{G}$ if there exists a  function $h: I \rightarrow \mathbb{G}$ and a function $C: \mathbb{R}^{+} \rightarrow \mathbb{R}^{+}$ such that
        \begin{align*}
            \int_{d(x, \0) \geq C(\eta)}|(\pi_{h(t)}u)(t, x)|^2 \, d \mu(x)+\|(P_{\geq C(\eta)}u)(t, \cdot)\|^2_{L^2(\h^3)} \leq \eta .
        \end{align*}
    \end{definition}

In the three-dimensional Euclidean case \cite{Dod3d}, Dodson's notion of almost periodicity requires three additional functions $\xi$, $x$, and $N$ so  that
    \begin{align*}
        \int_{|x-x(t)| \geq C(\eta)/N(t)}|u(t, x)|^2 \, dx+ \int_{|\xi-\xi(t)| \geq C(\eta)N(t)}|\hat{u}(t, \xi)|^2 \, d\xi < \eta .
    \end{align*}
    We refer to the function $N(t)$ as the frequency scale function, $x(t)$ is the spatial center function, $\xi(t)$ is the frequency center function, and $C(\eta)$ is the compactness modulus function. In particular, one can prove the following properties (Lemma \ref{lem AP constancy} and Lemma \ref{lem aprioriT}) of the frequency scale function:

\begin{lemma}\label{lem AP constancy}
Let $u$ be a minimal mass blow-up solution to \eqref{NLS} on $I$ that is almost periodic modulo $\mathbb{G}$. Then there exists $\delta(u)$ such that for all $t_0 \in I$,
\begin{align}
[t_0 - \delta N(t_0)^{-2}, t_0 + \delta N(t_0)^{-2}] \subset I
\end{align}
and
\begin{align}
N(t) \sim N(t_0). 
\end{align}
\end{lemma}

\begin{definition}\label{Defn LocCon}
Divide $[0,\infty)$ into consecutive intervals $J_k$ such that $\norm{u}_{L_{t,x}^{\frac{10}{3}} (J_k \times \h^3)} =1$. We call these $J_k$'s the intervals of local constancy and 
\begin{align}
N(t) \equiv N_k \geq 1 \quad \text{for each} \quad t \in J_k.
\end{align}
If $J \subset [0,\infty)$ is a union of consecutive intervals of local constancy, then
\begin{align}
\sum_{J_k} N(J_k) \sim \int_J N(t)^3 \, dt = : K .
\end{align}
For convenience let $J_k(t)$ denote the intervals $J_k$ to which $t$ belongs.
\end{definition}

\begin{lemma}\label{lem aprioriT}
If $u(t,x)$ is a minimal mass blow-up solution on an interval $J$, then
\begin{align}
\int_J N(t)^2 \, dt \leq \norm{u}_{L_{t,x}^{\frac{10}{3}} (J \times \h^3)}^{\frac{10}{3}} \lesssim 1 + \int_J N(t)^2 \, dt .
\end{align}
Combining with conservation of mass and Strichartz inequality, we have 
\begin{align}\label{eq local}
\norm{u}_{L_t^2 L_x^6} \lesssim 1 + \int_J N(t)^2 \, dt \lesssim 1 + \frac{1}{N_{\min}} \int_J N(t)^3 \,dt \lesssim 1 + \frac{K}{N_{\min}} ,
\end{align}
where $N_{\min} = \int_{t\in J} N(t)$. 
\end{lemma}

We make the observation that, since we have eliminated rescaling symmetries from our almost periodic solutions, $N(t)=1$ for all $t\in J$. Therefore, Lemma \ref{lem aprioriT} implies that 
    \begin{align}
        \norm{u}_{L_{t,x}^{\frac{10}{3}} (J \times \h^3)}^{\frac{10}{3}} \sim |J|.
    \end{align}
Using a Morawetz estimate we can prove the following proposition for solutions similar to $u$:

\begin{proposition}\label{prop MorDod}
Let $u \in C(\R; L^2(\h^3))$ be an almost periodic solution to \eqref{NLS}. Then for every $\eta>0$ there exists $T_0(\eta)$  and $C=C(\|u\|_{L^2(\h^3)})$ such that $T\geq T_0$ implies
    \begin{align*}
        C^{-1}T \leq\norm{P_{\leq T}u}_{L_{t,x}^{\frac{10}{3}}([-T,T] \times \h^3)}^{\frac{10}{3}} \leq \eta C T .
    \end{align*}
\end{proposition}
(Proposition \ref{prop MorDod} is proven in Section \ref{sec MorDod})

For $\eta>0$ small enough, this presents a contradiction. Therefore, $M_{\max}=\infty$, which proves Theorem \ref{thm mainthm}.


\section{Bilinear Strichartz Estimates and Improved Strichartz Inequalities}\label{sec bilinear}
In this section, we establish a bilinear Strichartz estimate for linear solutions in hyperbolic spaces. Then we subsequently derive an improved Strichartz inequality, which will be used in Section \ref{sec profile}.

\subsection{Bilinear Strichartz estimates}

Note that in this subsection, we use the convention that upper case letters denote functions in $\h^3$ while lower case letters denote functions in $\R^3$. Also $\widehat{f}$ denotes Fourier transforms in $\R^3$ and $\widetilde{F}$ denotes Fourier transforms in $\h^3$. 

\begin{lemma}[Bilinear estimates on $\R^d$, Corollary~6.5 in \cite{Shao09}]\label{lem Shao}
Suppose $\widehat{f}_N$ is a radial function and compactly supported on $\{ \xi \in \R^d : \abs{\xi} \sim N\}$ and $\widehat{g}_L$ is a radial function and compactly supported on $\{ \xi \in \R^d : \abs{\xi} \sim L\}$ with $N \leq L /4$.

Then
\begin{align}
\norm{e^{it \Delta_{\R^d}} f_N \, e^{it \Delta_{\R^d}} g_L }_{L_{t,x}^q (\R \times \R^d)} \lesssim C(N,L) \norm{f_N}_{L_x^2 (\R^d)} \norm{g_L}_{L_x^2 (\R^d)} ,
\end{align}
where $C(N,L)$ varies under different constrains of $q$.
\begin{align}
C(N,L) =
\begin{cases}
N^{\frac{2d+1}{2} - \frac{d+2}{q}} L^{-\frac{1}{2}} & \text{if} \quad \frac{d+1}{d} < q \leq 2 ,\\ 
N^{\frac{4d-1}{4} - \frac{2d+1}{2q}} L^{- \frac{3}{2q}+\frac{1}{4}}  & \text{if} \quad 2 \leq q \leq \frac{2(2d+1)}{2d-1}  ,\\ 
N^{\frac{d}{2}} L^{\frac{d}{2} - \frac{d+2}{q}} & \text{if} \quad q \geq \frac{2(2d+1)}{2d-1} .
\end{cases}
\end{align}
\end{lemma}

\begin{proposition}[Bilinear estimates on $\h^3$]\label{prop Bilinear}
Suppose $\widetilde{F}_N$ is a radial function and compactly supported on $\{ \lambda \in [0, \infty): \lambda \sim N\}$ and $\widetilde{G}_L$ is a radial function and compactly supported on $\{ \lambda \in [0, \infty) : \lambda \sim L\}$ with $N \leq L /4$. 

Then
\begin{align}
\norm{e^{it \Delta_{\h^3}} F_N \, e^{it \Delta_{\h^3}} G_L }_{L_{t,x}^q (\R \times \h^3)} \lesssim C(N,L) \norm{F_N}_{L_x^2 (\h^3)} \norm{G_L}_{L_x^2 (\h^3)} ,
\end{align}
where $C(N,L)$ is the same coefficient in Lemma \ref{lem Shao} (with $d=3$) under different constrains of $q$.
\begin{align}
C(N,L) =
\begin{cases}
N^{\frac{7}{2} - \frac{5}{q}} L^{-\frac{1}{2}} & \text{if} \quad \frac{4}{3} < q \leq 2 ,\\ 
N^{\frac{11}{4} - \frac{7}{2q}} L^{- \frac{3}{2q}+\frac{1}{4}} & \text{if} \quad 2 \leq q \leq \frac{14}{5} ,\\ 
N^{\frac{3}{2}} L^{\frac{3}{2} - \frac{5}{q}} & \text{if} \quad q \geq \frac{14}{5} .
\end{cases}
\end{align}
\end{proposition}

\begin{proof}[Proof of Proposition \ref{prop Bilinear}]
We first recall the radial Fourier transform and  its inversion on $\R^3$:
\begin{align}
\widehat{f} (k) & = \int_0^{\infty} J_{\frac{1}{2}} (kr) f(r) r^{\frac{3}{2}} k^{-\frac{1}{2}} \, dr ,\\
f(r) & = \int_0^{\infty} J_{\frac{1}{2}} (kr) \widehat{f}(k) k^{\frac{3}{2}} r^{-\frac{1}{2}} \, dk ,
\end{align}
and the radial Fourier transform and  its inversion on $\h^3$:
\begin{align}
\widetilde{F} (\lambda) & = \frac{c}{\lambda} \int_0^{\infty} \sin (\lambda r) F(r) \sinh r \, dr ,\\
F(r) & = c \int_0^{\infty} \widetilde{F} (\lambda) \frac{\sin (\lambda r)}{\lambda} \frac{1}{\sinh r} \lambda^2 \, d \lambda ,
\end{align}
with Plancherel theorem
\begin{align}
\int_0^{\infty} \abs{F(r)}^2 \sinh^2 r \, dr = c \int_0^{\infty} \abs{\widetilde{F} (\lambda)}^2 \lambda^2 \, d\lambda .
\end{align}

Recall from Section \ref{ssec COV}. 
We can write the radial Laplacian in $\R^3$ and $\h^3$ of the following form
\begin{align}
\Delta_{\R^3} & = \partial_r^2 + \frac{2}{r} \partial_r ,\\
\Delta_{\h^3} & = \partial_r^2 + 2 \coth r \partial_r . 
\end{align}
Hence if $u \in L^2(\h^3)$ solves 
\begin{align}
i \partial_t u + \Delta_{\h^3} u = 0 ,
\end{align}
then 
\begin{align}
w : = e^{it} \frac{\sinh r}{r} u \in L^2 (\R^3)
\end{align}
solves
\begin{align}
i \partial_t w + \Delta_{\R^3} w = 0 .
\end{align}
Also under this change of variables, we have the invariance of the $L^2$ norms, that is,
\begin{align}
\norm{u}_{L^2 (\h^3)}  = \norm{w}_{L^2 (\R^3)} .
\end{align}

We wish to use this change of variables to convert linear solutions on $\h^3$ to those on $\R^3$, then use the known bilinear estimates on $\R^3$ to derive a bilinear estimate on $\h^3$ using the following detour
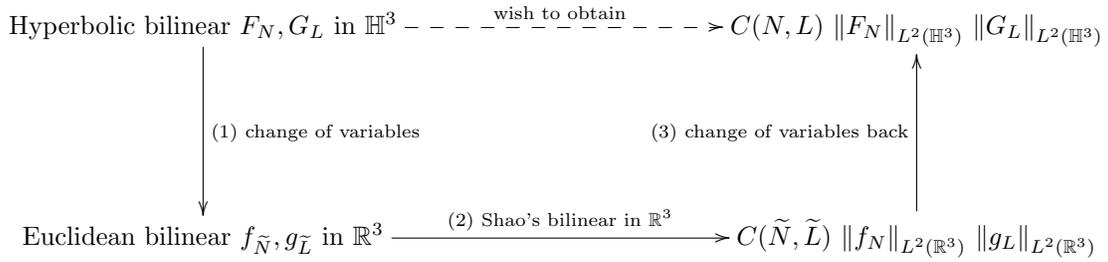
\begin{align}
\\
\xymatrixcolsep{10pc}
\xymatrixrowsep{5pc}
\xymatrix{
\text{Hyperbolic bilinear } F_N, G_L \text{ in } \h^3 \ar@{-->}[r]^{\text{wish to obtain}} \ar[d]^{\text{(1) change of variables}} & C(N,L) \norm{F_N}_{L^2 (\h^3)} \norm{G_L}_{L^2 (\h^3)}  \\
\text{Euclidean bilinear } f_{N'}, g_{L'} \text{ in } \R^3 \ar@{->}[r]^{\text{(2) Shao's bilinear in $\R^3$}} & C(N' , L') \norm{f_N}_{L^2(\R^3)} \norm{g_L}_{L^2(\R^3)} \ar[u]^{\text{(3) change of variables back}}
}
\\
\end{align}

More precisely, assuming $N \leq L$, we wish to prove
\begin{align}\label{eq Bilinear}
& \quad \norm{e^{it \Delta_{\h^3}} F_N e^{it \Delta_{\h^3}} G_L }_{L_{t,x}^q (\R \times \h^3)} \stackrel{(1)}{\sim} \norm{e^{it \Delta_{\R^3}} f_{N'} e^{it \Delta_{\R^3}} g_{L'} }_{L_{t,x}^q (\R \times \R^3)} \\
& \stackrel{(2)}{\lesssim} C(N', L') \norm{f_{N'}}_{L^2 (\R^3)}\norm{g_{L'}}_{L^2 (\R^3)} \stackrel{(3)}{\sim}  C(N, L) \norm{F_N}_{L^2 (\h^3)}\norm{G_L}_{L^2 (\h^3)} .
\end{align}
where the inequality (2) is some suitable bilinear estimate on $\R^3$. Moreover, we need to verify that under the change of variables, the frequency location preserves in (1) and (3), that is, $N \sim N'$ and $L \sim L'$.

Recall the linear solutions on $\R^3$ and $\h^3$ 
\begin{align}
e^{it \Delta_{\R^3}} f(r) & = \int_0^{\infty} J_{\frac{1}{2}} (kr) k^{\frac{3}{2}} r^{-\frac{1}{2}} e^{-itk^2} \widehat{f} (k) \, dk , \\
e^{it \Delta_{\h^3}} F(r) & = \int_0^{\infty} e^{-it (\lambda^2 + \rho^2)} \widetilde{F} (\lambda) \frac{\sin (\lambda r)}{\lambda \sinh r} \lambda^2 \, d \lambda ,
\end{align}
where $J_{\alpha}$ is Bessel functions of the first kind of order $\alpha$. Recall that $J_{\frac{1}{2}}$ has the following explicit formula
\begin{align}\label{eq Bessel}
J_{\frac{1}{2}} (z) = \sqrt{\frac{2}{\pi z}} \sin z .
\end{align}

Under the change of variables, we have
\begin{align}
\Phi (t,r) & = e^{it} \frac{\sinh r}{r} e^{it \Delta_{\h^3}} F_N (r) \\
& = e^{it} \int_0^{\infty} e^{-it (\lambda^2 + \rho^2)} \widetilde{F}_N (\lambda) \frac{\sinh r}{r} \frac{\sin (\lambda r)}{\lambda \sinh r} \lambda^2 \, d \lambda \\
& = e^{it} \int_0^{\infty} e^{-it (\lambda^2 + \rho^2)} \widetilde{F}_N (\lambda) \frac{\sin (\lambda r)}{\lambda r} \lambda^2 \, d\lambda ,
\end{align}
here we note that $\Phi $ is a linear solution in $\R^3$.

Then
\begin{align}
\widehat{\Phi} (t,k) & = \int_0^{\infty} \Phi(t,r) J_{\frac{1}{2}} (kr) r^{\frac{3}{2}} k^{-\frac{1}{2}} \, dr\\
& = e^{it} \int_0^{\infty} \int_0^{\infty} e^{-it (\lambda^2 + \rho^2)} \widetilde{F}_N (\lambda) \frac{\sin (\lambda r)}{\lambda r} \lambda^2 J_{\frac{1}{2}} (kr) r^{\frac{3}{2}} k^{-\frac{1}{2}} \, d\lambda dr \\
& =  e^{it} \int_0^{\infty} \int_0^{\infty} e^{-it (\lambda^2 + \rho^2)} \widetilde{F}_N (\lambda) \sin (\lambda r) J_{\frac{1}{2}} (kr) \lambda  r^{\frac{1}{2}} k^{-\frac{1}{2}} \, d\lambda dr  \\
&  =  e^{it} \int_0^{\infty}  \parenthese{ \int_0^{\infty} J_{\frac{1}{2}} (kr)  \sin (\lambda r)  r^{\frac{1}{2}} \, dr }  e^{-it (\lambda^2 + \rho^2)} \widetilde{F}_N (\lambda) \lambda   k^{-\frac{1}{2}} \, d\lambda .
\end{align}

Thanks to the exact expression of $J_{\frac{1}{2}}$ in \eqref{eq Bessel}, we write
\begin{align}
\int_0^{\infty} J_{\frac{1}{2}} (kr)  \sin (\lambda r)  r^{\frac{1}{2}} \, dr  =c  k^{-\frac{1}{2}}  \int_0^{\infty} \sin (kr) \sin (\lambda r) \, dr . 
\end{align}
We wish to show $\int_0^{\infty} J_{\frac{1}{2}} (kr)  \sin (\lambda r)  r^{\frac{1}{2}} \, dr $ is a delta like function in $k, \lambda$. In fact, let
    \begin{align}
        d\mu^{\lambda}(k):= c  k^{-\frac{1}{2}}  \int_0^{\infty} \sin (kr) \sin (\lambda r) \, dr ,
    \end{align}
then, since the integrand is even,
    \begin{align}
        d\mu^{\lambda}(k)&= \frac{c}{2}  k^{-\frac{1}{2}}  \int_{\mathbb{R}} \sin (kr) \sin (\lambda r) \, dr\\
        &= \frac{c}{4}  k^{-\frac{1}{2}}  \int_{\mathbb{R}} \cos((\lambda-k)r) - \cos((\lambda+k)r) \, dr\\
        &=\frac{c}{4}  k^{-\frac{1}{2}}  \re\left(\int_{\mathbb{R}} e^{i(\lambda-k)r} - e^{i(\lambda+k)r} \, dr \right)\\
        &=\frac{2\pi c}{4}  k^{-\frac{1}{2}}  (\delta_{\lambda}(k) - \delta_{-\lambda}(k) ) .
    \end{align}
As $k, \lambda >0$, we have
\begin{align}\label{eq delta}
d\mu^{\lambda}(k) = c k^{-\frac{1}{2}} \delta_{\lambda} (k) .
\end{align}
Then
\begin{align}
\widehat{\Phi} (t,k) = c e^{it} e^{-it (k^2 + \rho^2)} \widetilde{F}_N(k) .
\end{align}
which implies that in \eqref{eq Bilinear}, $N = N'$ and  $L = L'$, that is, under that change of variables, the linear solutions in $\R^3$ and $\h^3$ share the same frequency location range.

\begin{remark}
Let us stress that such frequency stability is the key feature in applying Euclidean bilinear estimates. In fact, in higher dimensions, due to the presence of the nontrivial potential in the change of variables trick (see Section \ref{ssec COV}), the identity in \eqref{eq delta} does not hold, but holds in a weak sense, that is, it will be a Dirac delta function plus an error term. However, the error term is not small enough to generalize this argument to higher dimensions. Thanks to \eqref{eq delta}, we can proceed with the bilinear estimate as follows. 
\end{remark}

Now the initial data $f_N$ of $\Phi$ can be found in the following form
\begin{align}
f_N & = e^{-it\Delta_{\R^3}} \Phi \\
\widehat{f_N} (k) & = e^{it k^2} \widehat{\Phi} (t,k) = e^{it k^2} c e^{it} e^{-it (k^2 + \rho^2)} \widetilde{F}_N(k)  = c e^{it} e^{-it \rho^2} \widetilde{F}_N(k)
\end{align}
and
\begin{align}
\norm{f_N}_{L^2(\R^3)} =c \norm{\widetilde{F}_N}_{L_{k}^2 (\R^3)} = c \norm{F_N}_{L^2 (\h^3)} ,
\end{align}
where the last equality is due to Plancherel theorem.

In fact, the constant $c$ above should be $c=1$, since
\begin{align}
\norm{f_N}_{L^2 (\R^3)} & = \norm{e^{-it\Delta_{\R^3}} \Phi }_{L^2(\R^3)} = \norm{\Phi }_{L^2(\R^3)} = \norm{e^{it} \frac{\sinh r}{r} e^{it \Delta_{\h^3}} F_N (r)}_{L^2(\R^3)} \\
& = \norm{e^{it \Delta_{\h^3}} F_N (r)}_{L^2(\h^3)} = \norm{ F_N (r)}_{L^2(\h^3)} .
\end{align}

Now we proved that $\Phi (t,r) = e^{it} \frac{\sinh r}{r} e^{it \Delta_{\h^3}} F_N (r)$ solves a linear Schr\"odinger equation on $\R^3$ with initial data $f_N$ localized around frequency $N$ and invariance of its $L^2$ norm under change of variables. That is,
\begin{align}
e^{it\Delta_{\R^3}} f_N =  e^{it} \frac{\sinh r}{r} e^{it \Delta_{\h^3}} F_N \implies e^{-it} \frac{r}{\sinh r} e^{it\Delta_{\R^3}} f_N =   e^{it \Delta_{\h^3}} F_N ,
\end{align}
then
\begin{align}
e^{it \Delta_{\h^3}} F_N  \, e^{it \Delta_{\h^3}} G_L  =  e^{-2it} \frac{r^2}{\sinh^2 r}  e^{it\Delta_{\R^3}} f_N \, e^{it\Delta_{\R^3}} g_L .
\end{align}

For $p \geq 1$, we claim
\begin{align}
\norm{e^{it \Delta_{\h^3}} F_N \, e^{it \Delta_{\h^3}} G_L}_{L^p (\h^3)} \leq \norm{e^{it\Delta_{\R^3}} f_N \, e^{it\Delta_{\R^3}} g_L}_{L^p(\R^3)} . 
\end{align}
In fact,
\begin{align}
\norm{e^{it \Delta_{\h^3}} F_N \, e^{it \Delta_{\h^3}} G_L}_{L^p (\h^3)}^p  & = \int \abs{e^{it \Delta_{\h^3}} F_N \, e^{it \Delta_{\h^3}} G_L}^p \sinh^2 r \, dr \\
& = \int \abs{e^{-2it} \frac{r^2}{\sinh^2 r}  e^{it\Delta_{\R^3}} f_N \, e^{it\Delta_{\R^3}} g_L}^p \sinh^2 r \, dr \\
& = \int \abs{ e^{it\Delta_{\R^3}} f_N \, e^{it\Delta_{\R^3}} g_L}^p \frac{r^{2p-2}}{\sinh^{2p -2} r} r^2 \, dr \\
& = \int \abs{ e^{it\Delta_{\R^3}} f_N \, e^{it\Delta_{\R^3}} g_L}^p \frac{r^{2p-2}}{\sinh^{2p -2} r} r^2 \, dr \\
& \leq \int \abs{ e^{it\Delta_{\R^3}} f_N \, e^{it\Delta_{\R^3}} g_L}^p  r^2 \, dr = \norm{e^{it\Delta_{\R^3}} f_N \, e^{it\Delta_{\R^3}} g_L}_{L^p(\R^3)}^p .
\end{align}

Hence we are able to apply the result in Lemma \ref{lem Shao} and write
\begin{align}
\norm{e^{it \Delta_{\h^3}} F_N \, e^{it \Delta_{\h^3}} G_L}_{L_{t,x}^p (\R \times \h^3)} &\leq \norm{e^{it\Delta_{\R^3}} f_N \, e^{it\Delta_{\R^3}} g_L}_{L_{t,x}^p(\R \times \R^3)} \\
& \lesssim C(N,L) \norm{f_N}_{L^2 (\R^3)} \norm{g_L}_{L^2 (\R^3)} \\
& = C(N,L) \norm{F_N}_{L^2 (\h^3)} \norm{G_L}_{L^2 (\h^3)} .
\end{align}
We finish the proof of Proposition \ref{prop Bilinear}.
\end{proof}

We also have the following nonlinear version of Proposition \ref{prop Bilinear} that follows as a corollary.

\begin{lemma}[Equivalent to Lemma 3.4 in \cite{CKSTT08} (see also Lemma 2.5 in \cite{Visan07})]\label{lem nonbilinear}
For any space-time slab $I \times \h^3$, any $t_0 \in I$, and any $\delta>0$, we have for $N \ll L$
\begin{align}\label{eq nonbilinear}
\|P_N u P_L v\|_{L_{t, x}^q(I \times \h^3)} \lesssim C(N,L)  \norm{P_N u}_{S_*^0} \norm{P_L v}_{S_*^0} ,
\end{align}
where the $S_*^0$ norm is defined as follows
\begin{align}
\norm{P_N u}_{S_*^0} & : = \norm{P_N u(t_0)}_{L_x^2} + \norm{(i \partial_t+\Delta_{\h^3} )P_N u}_{L_t^1 L_x^2 (I \times \h^3)} = \norm{P_N u(t_0)}_{L_x^2} + \norm{P_N F(u)}_{L_t^1 L_x^2 (I \times \h^3)} ,\\
\norm{P_L v}_{S_*^0} & : = \norm{P_L v(t_0)}_{L_x^2} + \norm{(i \partial_t+\Delta_{\h^3} )P_L v}_{L_t^1 L_x^2 (I \times \h^3)} = \norm{P_L v(t_0)}_{L_x^2} + \norm{P_L F(v)}_{L_t^1 L_x^2 (I \times \h^3)} ,
\end{align}
and $C(N,L)$ is the coefficient in Lemma \ref{lem Shao} for different range of $q$.

In particular, when $q=2$, \eqref{eq nonbilinear} agrees with the result in Lemma 3.4 in \cite{CKSTT08}.
\end{lemma}

\begin{proof}[Proof of Lemma \ref{lem nonbilinear}]
The proof is adapted from Lemma 3.4 in \cite{CKSTT08}.
Using the Duhamel formula, we write
\begin{align}
u & = e^{i(t-t_0)\Delta_{\h^3}} u(t_0)  -i \int_{t_0}^t e^{i(t-s)\Delta_{\h^3}} (i \partial_s + \Delta_{\h^3}) u (s) \, ds ,\\
v & = e^{i(t-t_0)\Delta_{\h^3}} v(t_0)  -i \int_{t_0}^t e^{i(t-s)\Delta_{\h^3}} (i \partial_s + \Delta_{\h^3}) v (s) \, ds .
\end{align}
We obtain
\begin{align}
&\norm{P_N u P_L v}_{L_{t,x}^q}  \\
& \quad \lesssim \norm{ e^{i(t-t_0)\Delta_{\h^3}} P_N u(t_0) e^{i(t-t_0)\Delta_{\h^3}} P_L v(t_0)}_{L_{t,x}^q} \\
& \quad\quad + \norm{e^{i(t-t_0)\Delta_{\h^3}} P_N u(t_0) \parenthese{ \int_{t_0}^t e^{i(t-s)\Delta_{\h^3}} (i \partial_s + \Delta_{\h^3}) P_L v(s) \, ds }}_{L_{t,x}^q} \\
& \quad \quad + \norm{e^{i(t-t_0)\Delta_{\h^3}} P_L v(t_0) \parenthese{\int_{t_0}^t e^{i(t-s)\Delta_{\h^3}} (i \partial_s + \Delta_{\h^3}) P_N u(s) \, ds }}_{L_{t,x}^q}  \\
& \quad  \quad+ \norm{\parenthese{\int_{t_0}^t e^{i(t-s)\Delta_{\h^3}}(i \partial_s + \Delta_{\h^3}) P_N u(s) \, ds} \parenthese{\int_{t_0}^t e^{i(t-s')\Delta_{\h^3}}(i \partial_{s'} + \Delta_{\h^3}) P_L v(s') \, ds'}}_{L_{t,x}^q} \\
& \quad = : I_1 + I_2 + I_3 + I_4 .
\end{align}

Term $I_1$ is treated in Proposition \ref{prop Bilinear}. For Term $I_2$, by the Minkowski inequality (since $q >1$), we write
\begin{align}
I_2 \lesssim \int_{\R} \norm{ e^{i(t-t_0)\Delta_{\h^3}} P_N u(t_0) e^{i(t-s)\Delta_{\h^3}} (i\partial_s + \Delta_{\h^3}) P_L v(s)}_{L_{t,x}^q} \, ds
\end{align}
then this case follows from the homogenous estimate in  Proposition \ref{prop Bilinear}. Term $I_3$ can be treated similarly.

For Term $I_4$, by the Minkowski inequality, we have
\begin{align}
I_4 \lesssim \int_{\R} \int_{\R} \norm{\parenthese{e^{i(t-s)\Delta_{\h^3}} (i\partial_s + \Delta_{\h^3}) P_N u(s) } \parenthese{e^{i(t-s')\Delta_{\h^3}} (i\partial_{s'} + \Delta_{\h^3}) P_L v(s') }}_{L_{t,x}^q} \, ds ds' ,
\end{align}
then the proof follows by inserting in the integrand the homogeneous estimate in Proposition \ref{prop Bilinear}. This completes the proof of Lemma \ref{lem nonbilinear}.
\end{proof}

\subsection{Improved Strichartz inequalities}
In this subsection, we prove a useful inequality that will be utilized in profile decomposition (see Section \ref{ssec profile}). 
\begin{proposition}[Improved Strichartz inequalities]\label{prop Improved}
\begin{align}
\norm{e^{it\Delta_{\h^3}}f}_{L_{t,x}^{\frac{10}{3}} (\R \times \h^3)}^{\frac{10}{3}} \lesssim \sup_{N} \norm{N^{-\frac{3}{2}}e^{it\Delta_{\h^3}} P_N f}_{L_{t,x}^{\infty} (\R \times \h^3)}^{\frac{1}{3}} \norm{f}_{L_x^2 (\h^3)}^{3} .
\end{align}
More generally, for $\frac{4}{3} < q < \frac{5}{3}$
\begin{align}
\norm{e^{it\Delta_{\h^3}}f}_{L_{t,x}^{\frac{10}{3}} (\R \times \h^3)}^{\frac{10}{3}} \lesssim \sup_{N} \norm{N^{-\frac{3}{2}}e^{it\Delta} P_N f}_{L_{t,x}^{\infty} (\R \times \h^3)}^{\frac{10}{3} -2q} \norm{f}_{L_x^2 (\h^3)}^{2q} .
\end{align}
\end{proposition}

\begin{proof}[Proof of Proposition \ref{prop Improved}]
Take $q = \frac{3}{2} \in (\frac{4}{3},2]$ in Lemma \ref{prop Bilinear}
\begin{align}
\norm{e^{it \Delta_{\h^3}} f_N \, e^{it \Delta_{\h^3}} g_L }_{L_{t,x}^{\frac{3}{2}} (\R \times \h^3)} \lesssim N^{\frac{1}{6}} L^{-\frac{1}{2}} \norm{f_N}_{L_x^2 (\h^3)} \norm{g_L}_{L_x^2 (\h^3)} .
\end{align}
Then breaking $f$ into different frequencies, using triangle inequality, interpolation, and Proposition \eqref{prop Bilinear}, we write
\begin{align}
&\quad \norm{e^{it\Delta_{\h^3}}f}_{L_{t,x}^{\frac{10}{3}} (\R \times \h^3)}^{\frac{10}{3}} \lesssim  \norm{\sum_{N \leq L} [e^{it\Delta_{\h^3}}  P_N f] [e^{it\Delta_{\h^3}} P_L f]}_{L_{t,x}^{\frac{5}{3}} (\R \times \h^3)}^{\frac{5}{3}} \\
& \lesssim  \sum_{N \leq L}  \norm{[e^{it\Delta_{\h^3}}  P_N f] [e^{it\Delta_{\h^3}} P_L f]}_{L_{t,x}^{\frac{5}{3}} (\R \times \h^3)}^{\frac{5}{3}} \\
& \lesssim \sum_{N \leq L} \norm{e^{it\Delta_{\h^3}}  P_N f \, e^{it\Delta_{\h^3}} P_L f}_{L_{t,x}^{\infty} (\R \times \h^3)}^{\frac{1}{6}} \norm{e^{it\Delta_{\h^3}}  P_N f \, e^{it\Delta_{\h^3}} P_L f}_{L_{t,x}^{\frac{3}{2}} (\R \times \h^3)}^{\frac{3}{2}} \\
& \lesssim \sum_{N \leq L} \norm{N^{-\frac{3}{2}}e^{it\Delta_{\h^3}}  P_N f }_{L_{t,x}^{\infty} }^{\frac{1}{6}}  \norm{L^{-\frac{3}{2}}e^{it\Delta_{\h^3}} P_L f}_{L_{t,x}^{\infty} }^{\frac{1}{6}} (N^{\frac{3}{2}} L^{\frac{3}{2}})^{\frac{1}{6}} (N^{\frac{1}{6}} L^{-\frac{1}{2}})^{\frac{3}{2}} \norm{P_N f}_{L_x^2}^{\frac{3}{2}} \norm{P_L f}_{L_x^2}^{\frac{3}{2}} \\
& \lesssim \sup_{N} \norm{N^{-\frac{3}{2}}e^{it\Delta_{\h^3}} P_N f}_{L_{t,x}^{\infty} (\R \times \h^3)}^{\frac{1}{3}} \sum_{N \leq L} N^{\frac{1}{2}} L^{-\frac{1}{2}} \norm{P_N f}_{L_x^2(\h^3)}^{\frac{3}{2}} \norm{P_L f}_{L_x^2(\h^3)}^{\frac{3}{2}} \\
& \lesssim \sup_{N} \norm{N^{-\frac{3}{2}}e^{it\Delta_{\h^3}} P_N f}_{L_{t,x}^{\infty} (\R \times \h^3)}^{\frac{1}{3}} \norm{f}_{L_x^2 (\h^3)}^{3} .
\end{align}
Note that the second inequality is not always true, one needs to use almost orthogonality (see, for example, Lemma A.9 in \cite{KVClay}).

As one notices in the computation above $\frac{3}{2}$ norm is not essential, any $q$ norm ($ \frac{4}{3} < q < \frac{5}{3}$) would work. In fact,
\begin{align}
&\quad \norm{e^{it\Delta_{\h^3}}f}_{L_{t,x}^{\frac{10}{3}} (\R \times \h^3)}^{\frac{10}{3}} \lesssim  \norm{\sum_{N \leq L} [e^{it\Delta_{\h^3}}  P_N f] [e^{it\Delta_{\h^3}} P_L f]}_{L_{t,x}^{\frac{5}{3}} (\R \times \h^3)}^{\frac{5}{3}} \\
& \lesssim  \sum_{N \leq L}  \norm{[e^{it\Delta_{\h^3}}  P_N f] [e^{it\Delta_{\h^3}} P_L f]}_{L_{t,x}^{\frac{5}{3}} (\R \times \h^3)}^{\frac{5}{3}} \\
& \lesssim \sum_{N \leq L} \norm{e^{it\Delta_{\h^3}}  P_N f \, e^{it\Delta_{\h^3}} P_L f}_{L_{t,x}^{\infty} (\R \times \h^3)}^{\frac{5}{3} -q} \norm{e^{it\Delta_{\h^3}}  P_N f \, e^{it\Delta_{\h^3}} P_L f}_{L_{t,x}^{q} (\R \times \h^3)}^{q} \\
& \lesssim \sum_{N \leq L} \norm{N^{-\frac{3}{2}}e^{it\Delta_{\h^3}}  P_N f }_{L_{t,x}^{\infty} }^{\frac{5}{3}-q}  \norm{L^{-\frac{3}{2}}e^{it\Delta_{\h^3}} P_L f}_{L_{t,x}^{\infty} }^{\frac{5}{3}-q} (N^{\frac{3}{2}} L^{\frac{3}{2}})^{\frac{5}{3}-q} (N^{\frac{7}{2} -\frac{5}{q}} L^{-\frac{1}{2}})^{q} \norm{P_N f}_{L_x^2}^{q} \norm{P_L f}_{L_x^2}^{q} \\
& \lesssim \sup_{N} \norm{N^{-\frac{3}{2}}e^{it\Delta_{\h^3}} P_N f}_{L_{t,x}^{\infty} (\R \times \h^3)}^{\frac{10}{3} -2q} \sum_{N \leq L} N^{-\frac{5}{2} + 2q} L^{\frac{5}{2}-2q} \norm{P_N f}_{L_x^2(\h^3)}^{q} \norm{P_L f}_{L_x^2(\h^3)}^{q} \\
& \lesssim \sup_{N} \norm{N^{-\frac{3}{2}}e^{it\Delta_{\h^3}} P_N f}_{L_{t,x}^{\infty} (\R \times \h^3)}^{\frac{10}{3} -2q} \norm{f}_{L_x^2 (\h^3)}^{2q} ,
\end{align}
where we needed $\frac{5}{2} - 2q < 0$ to sum the dyadic coefficient, which implies $q > \frac{5}{4}$. Also, we need $ \frac{4}{3} < q < 2$ to be able to use  Shao's bilinear estimates in Lemma \ref{lem Shao}, and $q < \frac{5}{3}$ to make sense of the interpolation.

In the last inequality above, we sum for $a>0$, $p \geq 1$,
\begin{align}
\sum_{N \leq L} (\frac{N}{L})^{a} \norm{P_N f}_{L_x^2}^p \norm{P_L f}_{L_x^2}^p & \lesssim (\sum_{N \leq L} (\frac{N}{L})^{a} \norm{P_N f}_{L_x^2}^{2p} )^{\frac{1}{2}} (\sum_{N \leq L} (\frac{N}{L})^{a} \norm{P_L f}_{L_x^2}^{2p})^{\frac{1}{2}} \lesssim \norm{f}_{L_x^2}^{2p} ,
\end{align}
where the first sum above is controlled by
\begin{align}
\sum_{N \leq L} (\frac{N}{L})^{a} \norm{P_N f}_{L_x^2}^{2p} & = \sum_{ N } N^{a} \norm{P_N f}_{L_x^2}^{2p} \sum_{L: L \geq N}  \frac{1}{L^{a}}   \\
& \lesssim  \sum_{ N } N^{a} \norm{P_N f}_{L_x^2}^{2p} \frac{1}{N^{a}} = \sum_{N} \norm{P_N f}_{L_x^2}^{2p} \leq  (\sum_N  \norm{P_N f}_{L_x^2}^2 )^{p} = \norm{f}_{L_x^2}^{2p} ,
\end{align}
and the second sum above is bounded by
\begin{align}
\sum_{N \leq L} (\frac{N}{L})^{a} \norm{P_L f}_{L_x^2}^2 & = \sum_{L} \frac{1}{L^{a}} \norm{P_L f}_{L_x^2}^{2p} \sum_{N : N \leq L} N^{a} \\
& \lesssim \sum_{L} \frac{1}{L^{a}} \norm{P_L f}_{L_x^2}^{2p} L^{a} = \sum_{L} \norm{P_L f}_{L_x^2}^{2p} \leq \norm{f}_{L_x^2}^{2p} .
\end{align}
Now we complete the proof of Proposition \ref{prop Improved}.
\end{proof}


\section{Euclidean Approximations}\label{sec euclidean}
This section marks the beginning of the profile decomposition argument. For profiles whose support localizes in hyperbolic space, we would like to use established well-posedness results in Euclidean space along with the fact that hyperbolic space is locally Euclidean to guarantee the global well-posedness of profiles defined by what will be known as ``Euclidean profiles". 

We fix a spherically symmetric function $\eta \in C_0^{\infty} (\R^3)$ supported in the disk of radius $2$ and equal to $1$ in the disk of radius $1$. Given $\phi \in L^2(\R^3)$ and a real number $N \geq 1$ we define
\begin{align}\label{eq f_N}
\begin{aligned}
 \phi_N & \in C_0^{\infty} (\R^3),  & \phi_N (x)  & = \eta(\frac{x}{N^{\frac{1}{2}}}) \cdot (e^{\Delta_{\R^3} / N} \phi) (x), \\
T_N \phi  & \in C_0^{\infty} (\h^3),  &  T_N \phi (y) & = N^{3/2}\phi_N (N\Psi_{\mathcal{I}}^{-1} (y)) ,
\end{aligned}
\end{align}
where $\Psi_{\mathcal{I}}$ is defined in \eqref{eq Diffeo}. Thus $\phi_N$ is a regularized, compactly supported modification  of the profile $\phi$, $N^{3/2}\phi_N (N \cdot)$ is an $L^2$-invariant rescaling of $ \phi_N$, and $T_N \phi$ is the function obtained by transferring $\phi_N$ to a neighborhood of $\0$ in $\h^3$.

Note here the scaling of $N^{3/2}\phi_N (N \cdot)$ is due to the scaling symmetry of the equation $i \partial_t u + \Delta u = \abs{u}^{p-1}u$
\begin{align}
u(t,x) \to \lambda^{\frac{2}{p-1}} u(\lambda^2 t ,\lambda x) .
\end{align}

\begin{theorem}[Mass-critical global well-posedness on $\R^3$ in Tao-Visan-Zhang \cite{TVZ07} and Dodson \cite{Dod3d}]\label{thm Dodson}
Assume $\phi \in L^2(\R^3)$. Then there is a unique global solution $v \in C(\R ; L^2 (\R^3))$ of the initial-value problem
\begin{align}
\begin{cases}
i \partial_t v + \Delta_{\R^3} v = \abs{v}^{\frac{4}{3}} v, \\
v(0,x) = \phi(x), 
\end{cases}
\end{align}
and 
\begin{align}
\norm{v}_{S_{\R^3}^0 (\R \times \R^3)} \leq \widetilde{C} (\norm{\phi}_{L^2 (\R^3)}).
\end{align}
This solution scatters in the sense that  there exists $v^{\pm \infty} \in L^2 (\R^3)$ such that
\begin{align}\label{eq Scattering}
\lim_{ t \to \pm \infty} \norm{v(t) - e^{it\Delta_{\R^3}} v^{\pm \infty}}_{L^2 (\R^3)} =0 .
\end{align} 
If $\phi \in H^5 (\R^3)$, then $v \in C(\R ; H^5 (\R^3))$ and 
\begin{align}
\sup_{t \in \R} \norm{v(t)}_{H^5(\R^3)} \lesssim_{\norm{\phi}_{H^5 (\R^3)}} 1.
\end{align}
\end{theorem}

\begin{lemma}\label{lem ProfileAppx}
Assume $\phi \in L^2 (\R^3)$, $T_0 \in (0, \infty)$ and $\rho \in \{ 0, 1\}$ are given, and define $T_N\phi$ as in \eqref{eq f_N}. 
\begin{enumerate}
\item
There is $N_0 = N_0 (\phi , T_0)$ sufficiently large such that for any $N \geq N_0$ there is a unique solution $U_N \in C ( (-T_0 N^{-2} , T_0 N^{-2}) ; L^2(\h^3) )$ of the initial-value problem
\begin{align}\label{eq NLSN}
\begin{cases}
i \partial_t U_N + \Delta_{\h^3} U_N = \rho \abs{U_N}^{\frac{4}{3}} U_N\\
U_N (0) = T_N\phi .
\end{cases}
\end{align}
Moreover, for any $N \geq N_0$,
\begin{align}
\norm{U_N}_{S_{\h^3}^0((-T_0 N^{-2} , T_0 N^{-2})) } \lesssim_{\norm{\phi}_{L^2 (\R^3)}} 1 .
\end{align}

\item
Assume $\varepsilon_1 \in ( 0 ,1]$ is sufficiently small (depending only on $\norm{\phi}_{L^2 (\R^3)}$), and let $\phi' \in H^5(\R^3)$ satisfy $\norm{\phi - \phi'}_{L^2 (\R^3)} \leq \varepsilon_1$. Let $v' \in C (\R ; H^5(\R^3))$ denote the solution of the initial-value problem
\begin{align}
\begin{cases}
i \partial_t v' + \Delta_{\R^3} v' = \rho \abs{v'}^{\frac{4}{3}} v'\\
v' (0) = \phi' .
\end{cases}
\end{align}
For $R, N \geq 1$ we define
\begin{align}
v'_R (t,x) & = \eta (\frac{x}{R}) v' (t,x)  ,  &  (t,x) & \in (-T_0 , T_0) \times \R^3 \\
v'_{R,N} (t,x) & = N^{\frac{3}{2}}  v'_R (N^2 t, Nx)  ,  &  (t,x) & \in (-T_0 N^{-2} , T_0 N^{-2}) \times \R^3 \\
V_{R,N} (t,y) & =  v'_{R,N} (t, \Psi_{\mathcal{I}}^{-1} (y))  ,  &  (t,y) & \in (-T_0 N^{-2} , T_0 N^{-2}) \times \h^3 
\end{align}
Then there is $R_0 \geq 1$ (depending on $T_0$, $\phi'$ and $\varepsilon_1$) such that, for any $R \geq R_0$,
\begin{align}
\limsup_{N \to \infty} \norm{U_N - V_{R,N}}_{S_{\h^3}^0( (-T_0 N^{-2} , T_0 N^{-2})) } \lesssim_{\norm{\phi}_{L^2 (\R^3)}} \varepsilon_1 .
\end{align}
\end{enumerate} 
\end{lemma}

\begin{proof}[Proof of Lemma \ref{lem ProfileAppx}]
Part (1) follows from part (2) in a straightforward manner, so we proceed with the proof of part (2):

Let us note first that the implicit constants may depend on $M_{\R^3} (\phi)$. By Theorem \ref{thm Dodson}, we have,
\begin{align}
\norm{v}_{S_{\R^3}^0 (\R )} \lesssim 1 , \quad \sup_{t \in \R} \norm{v(t)}_{H^5(\R^3)} \lesssim_{\norm{\phi}_{H^5 (\R^3)}} 1 .
\end{align}
We will prove that for any $R_0$ sufficiently large there is $N_0$ such that $V_{R_0, N}$ is an almost-solution of \eqref{eq NLSN}, for any $N \geq N_0$. 

Let 
\begin{align}
e_{R} (t,x) & : = [(i \partial_t + \Delta_{\R^3})v_R'  - \rho \abs{v_R' }^{\frac{4}{3}} v_R' ] (t,x) \\
& = \rho \cdot [\eta(\frac{x}{R}) - \eta(\frac{x}{R})^{\frac{7}{3}}] \abs{v'}^{\frac{4}{3}} v' (t,x) + R^{-2} v' (t,x) \Delta_{\R^3} \eta (\frac{x}{R}) + 2 R^{-1} \sum_{j=1}^2 \partial_j v' (t,x) \partial_j \eta (\frac{x}{R}) .
\end{align}
Then 
\begin{align}
\norm{e_{R}}_{L_t^1 L_x^2 ( (-T_0, T_0) \times \R^3)} \lesssim \norm{\chi_{\{R \leq\abs{x}\leq 2R\}} \parenthese{ v'(t,x) +  \sum_{j=1}^2 \partial_j v' (t,x)} }_{L_t^1 L_x^2 ( (-T_0, T_0) \times \R^3)} \to 0 
\end{align}
as $R \to \infty$. 

Letting 
\begin{align}
e_{R,N} (t,x) & : = [(i \partial_t + \Delta_{\R^3})v_{R,N}'  - \rho \abs{v_{R,N}' }^{\frac{4}{3}
} v_{R,N}' ] (t,x)  = N^{\frac{7}{2}} e_R (N^2 t, Nx) ,
\end{align}
there exists $R_0 \geq 1$ such that, for any $R \geq R_0$ and $N \geq 1$, 
\begin{align}\label{eq ERN}
\norm{e_{R,N}}_{L_t^1 L_x^2 ( (-T_0 N^{-2} , T_0 N^{-2}) \times \R^3)} = \norm{e_{R}}_{L_t^1 L_x^2 ( (-T_0  , T_0 ) \times \R^3)}  \leq \varepsilon_1 .
\end{align}
Turning to $V_{R,N} (t,y) = v_{R,N}' (t, \Psi_{\mathcal{I}}^{-1} (y)) $, we write
\begin{align}
E_{R,N} (t,y)  : & = [(i \partial_t + \Delta_{\mathbf{g}})V_{R,N}  - \rho \abs{V_{R,N} }^{\frac{4}{3}} V_{R,N} ] (t,y)  \\
& = e_{R,N} (t, \Psi_{\mathcal{I}}^{-1}(y)) + \Delta_{\mathbf{g}} V_{R,N} (t,y) - (\Delta_{\R^3} v_{R,N}' ) (t, \Psi_{\mathcal{I}}^{-1}(y)) .
\end{align}

We write $\partial_j$, $j=1,2,3$ as the standard vector fields on $\R^3$ and $\widetilde{\partial_j} := (\Psi_{\mathcal{I}}) * (\partial_j)$ as the induced vector fields on $\h^3$. 
\begin{align}
\mathbf{g}_{ij} (y) : = \mathbf{g}_y (\widetilde{\partial_i} , \widetilde{\partial_j} ) = \delta_{ij} - \frac{v_i v_j}{1 + \abs{v}^2} , \quad y = \Psi_{\mathcal{I}} (v) .
\end{align}
Using the standard formula for the Laplace-Beltrami operator 
\begin{align}
\Delta_{\mathbf{g}} f = \abs{\mathbf{g}}^{-\frac{1}{2}} \widetilde{\partial_i} (\abs{\mathbf{g}}^{\frac{1}{2}} \mathbf{g}^{ij} \widetilde{\partial_j} f) ,
\end{align}
then we have
\begin{align}
\abs{ \Delta_{\mathbf{g}} f(y) - \Delta (f \circ \Psi_{\mathcal{I}}) (\Psi_{\mathcal{I}}^{-1} (y)) } \lesssim \sum_{k=1}^3 \abs{\Psi_{\mathcal{I}}^{-1} (y)}^{k-1} \abs{\widetilde{\nabla}^k f(y)} ,
\end{align}
for any $C^3$ function $f : \h^3 \to \C$ supported in the ball of radius 1 around $\0$, where 
\begin{align}
\abs{\widetilde{\nabla}^k f(y)} : = \sum_{k_1 + k_2 + k_3 =k} \abs{\widetilde{\partial}_1^{k_1} \widetilde{\partial}_2^{k_2} \widetilde{\partial}_3^{k_3} h(y)} .
\end{align}

Then
\begin{align}
\abs{E_{R,N} (t,y)} & \lesssim \abs{e_{R,N} (t, \Psi_{\mathcal{I}}^{-1}(y))} +  \sum_{k=1}^2 \sum_{k_1 + k_2 + k_3 =k} \abs{ \Psi_{\mathcal{I}}^{-1}(y)}^{k-1} \abs{\partial_1^{k_1} \partial_2^{k_2} \partial_3^{k_3} v_{R,N}'(t, \Psi_{\mathcal{I}}^{-1}(y))} \\
& \lesssim \abs{e_{R,N} (t, \Psi_{\mathcal{I}}^{-1}(y))} + (\frac{R}{N}) N^2 N^{\frac{3}{2}} \sum_{k_1 + k_2 + k_3 \in \{ 1,2\} } \abs{\partial_1^{k_1} \partial_2^{k_2} \partial_3^{k_3} v_{R}'(t, N\Psi_{\mathcal{I}}^{-1}(y))} \\
& = \abs{e_{R,N} (t, \Psi_{\mathcal{I}}^{-1}(y))} + R N^{\frac{5}{2}} \sum_{k_1 + k_2 + k_3 \in \{ 1,2\} } \abs{\partial_1^{k_1} \partial_2^{k_2} \partial_3^{k_3} v_{R}'(t, N\Psi_{\mathcal{I}}^{-1}(y))} .
\end{align}
Therefore, combining with \eqref{eq ERN}, we have that for any $R_0$ sufficiently large there is $N_0$ such that for any $N \geq N_0$
\begin{align}
\norm{E_{R,N}}_{L_t^1 L_x^2 ( (-T_0 N^{-2} , T_0 N^{-2}) \times \h^3)} \leq \varepsilon_1 + c R N^{\frac{5}{2}}  N^{-\frac{3}{2}} (T_0 N^{-2})  \leq 2 \varepsilon_1 .
\end{align}

Check the smallness condition \eqref{eq Stab1} in Proposition \ref{prop Stab}. 
\begin{align}
& \quad \norm{V_{R_0 , N}}_{L_{t,x}^{\frac{10}{3}} ( (-T_0N^{-2} , T_0 N^{-2}) \times \h^3)} + \sup_{t \in (-T_0N^{-2} , T_0 N^{-2}) } \norm{V_{R_0 , N} (t)}_{L^2 (\h^3)} \\
& \lesssim \norm{v_{R_0 , N}'}_{L_{t,x}^{\frac{10}{3}} ( (-T_0N^{-2} , T_0 N^{-2}) \times \R^3)} + \sup_{t \in (-T_0N^{-2} , T_0 N^{-2}) } \norm{v_{R_0 , N}' (t)}_{L^2 (\R^3)} \\
& = \norm{v_{R_0}'}_{L_{t,x}^{\frac{10}{3}} ( (-T_0 , T_0 ) \times \R^3)} + \sup_{t \in (-T_0, T_0 ) } \norm{v_{R_0}' (t)}_{L^2 (\R^3)} \\
& \lesssim 1 .
\end{align}

Finally, 
\begin{align}
\norm{T_N\phi - V_{R_0 , N}}_{L^2 (\h^3)} & \lesssim \norm{N^{3/2}\phi_N(N \cdot) - v_{R_0, N}' (0)}_{L^2 (\R^3)} = \norm{\phi_N - v_{R_0}' (0)}_{L^2 (\R^3)} \\
& \leq \norm{\phi_N - \phi}_{L^2 (\R^3)} + \norm{\phi - \phi'}_{L^2 (\R^3)} + \norm{\phi' - v_{R_0}' (0)}_{L^2 (\R^3)} \leq 3 \varepsilon_1 .
\end{align}

The proof of Lemma \ref{lem ProfileAppx} is then finished.
\end{proof}

As a consequence, we have
\begin{corollary}\label{cor ProfileAppx}
Assume $ \psi \in L^2 (\R^3)$, $\varepsilon > 0$, $J \subset \R$ is an interval, and
\begin{align}
\norm{e^{it\Delta_{\R^3}} \psi}_{L_t^p L_x^q (J \times \R^3)} \leq \varepsilon ,
\end{align}
where $(p,q)$ is admissible in $\R^3$, $q > 2$. 
For $N \geq 1$, we define as before,
\begin{align}
 \psi_N (x)  & = \eta(\frac{x}{N^{\frac{1}{2}}}) \cdot (e^{\Delta_{\R^3} / N} \psi) (x),   &  T_N\psi  (y) & =N^{3/2}\phi_N (N\Psi_{\mathcal{I}}^{-1} (y)) .
\end{align}
Then there exists $N_1 = N_1 (\psi , \varepsilon)$ such that, for any $N \geq N_1$,
\begin{align}
\norm{e^{it\Delta_{\h^3}} T_N\psi}_{L_t^p L_x^q (N^{-2}J \times \h^3)} \lesssim_q \varepsilon .
\end{align}
\end{corollary}

\begin{proof}[Proof of Corollary \ref{cor ProfileAppx}]
As before, the implicit constants may depend on $M_{\R^3} (\psi)$.

We assume that $\psi \in C_0^{\infty} (\R^3)$. 
\begin{align}
\norm{e^{it \Delta_{\h^3}} T_N\psi}_{L_x^q (\h^3)} & \lesssim \abs{t}^{\frac{3}{q} -\frac{3}{2}} \norm{ T_N\psi}_{L_x^{q'} (\h^3)} \lesssim \abs{t}^{\frac{3}{q} -\frac{3}{2}} \norm{\psi_N}_{L_x^{q'} (\R^3)}  \\
& \lesssim \abs{t}^{\frac{3}{q}-\frac{3}{2}} N^{\frac{3}{2}- \frac{3}{q'}} \norm{\psi_N}_{L_x^2 (\R^3)} \lesssim_{\psi} \abs{t}^{\frac{3}{q}-\frac{3}{2}} N^{\frac{3}{q}-\frac{3}{2}} ,
\end{align}
where $\frac{1}{q} + \frac{1}{q'} =1$. 

Thus for $T_1>0$,
\begin{align}
\norm{e^{it\Delta_{\h^3}}  T_N\psi}_{L_t^p L_x^q ([\R \setminus (-T_1N^{-1} , T_1 N^{-2})] \times \h^3)} \lesssim_{\psi} \abs{T_1N^{-2}}^{\frac{3}{q}-\frac{3}{2}+ \frac{1}{p}} N^{\frac{3}{q}-\frac{3}{2}}  \lesssim T_1^{-\frac{1}{p}} 
\end{align}
Therefore we can fix $T_1 = T_1 (\psi , \varepsilon)$ such that for any $N \geq 1$,
\begin{align}
\norm{e^{it\Delta_{\h^3}}  T_N\psi}_{L_t^p L_x^q ([\R \setminus (-T_1N^{-1} , T_1 N^{-2})] \times \h^3)} \lesssim_q \varepsilon .
\end{align}
The desired bound on the remaining interval $N^{-2} J \cap (-T_1N^{-1} , T_1 N^{-2})$ follows from {\it Part (2)} in Lemma \ref{lem ProfileAppx} with $\rho =0$.

Then
\begin{align}
\norm{e^{it\Delta_{\h^3}} T_N\psi}_{L_t^p L_x^q (N^{-2}J \times \h^3)} & \lesssim \norm{e^{it\Delta_{\h^3}} T_N\psi}_{L_t^p L_x^q ([N^{-2}J \cap (\R \setminus (-T_1N^{-1} , T_1 N^{-2}))] \times \h^3)} \\
& \quad + \norm{e^{it\Delta_{\h^3}} T_N\psi}_{L_t^p L_x^q ([N^{-2} J \cap (-T_1N^{-1} , T_1 N^{-2})] \times \h^3)} \\
& \lesssim \varepsilon. 
\end{align}
Now we have finished the proof of Corollary \ref{cor ProfileAppx}. 
\end{proof}


\section{Profile Decomposition in Hyperbolic Spaces}\label{sec profile}

In this section, we present a profile decomposition of the linear solutions and nonlinear solutions.  We note that, in this section and Section \ref{sec Key}, we will (regrettably) recycle the $\widetilde{f}$ notation. In this section, $\widetilde{f}$ will not denote the Fourier transform on $\h^3$, but a different operation to be specified in Definition \ref{defn Profiles}.

\begin{definition}\label{defn Pi+T}
Given $(f , t_0 , h_0) \in L^2 (\h^3) \times \R \times \G$ we define
\begin{align}\label{eq Pi}
\Pi_{t_0 ,h_0} f(x) = (e^{-it_0\Delta_{\h^3}} f )(h_0^{-1} x) = (\pi_{h_0} e^{-it_0\Delta_{\h^3}} f )(x) .
\end{align}
Given $\phi \in L^2 (\R^3)$ and $N \geq 1$, we recall the definition
\begin{align}
\phi_N(y) : =  \eta(\frac{y}{N^{\frac{1}{2}}}) \cdot  e^{\Delta_{\R^3}/N}\phi(y)
\end{align}
and 
\begin{align}\label{eq T}
T_N  \phi (x) : = N^{3/2} \phi_N (N \Psi_{\mathcal{I}}^{-1} (x)) , 
\end{align}
and observe that $T_N : L^2 (\R^3) \to L^2 (\h^3)$ is bounded linear operator with $\norm{T_N \phi}_{L^2 (\h^3)} \lesssim \norm{\phi}_{L^2 (\R^3)}$. We also define $T^*_N : L^2 (\h^3) \to L^2 (\R^3)$ by 
\begin{align*}
    T^*_N  f (y) : = e^{\Delta_{\R^3}/N}\left[\eta(\frac{y}{N^{\frac{1}{2}}})\cdot N^{-3/2} f ( \Psi_{\mathcal{I}} (N^{-1}y))\right] .
\end{align*}
\end{definition}

\subsection{A Tool for the profile decomposition argument}\label{ssec profile}

\begin{proposition}[Inverse Strichartz Inequality]\label{prop InvStrichartz}
Let $\{f_k\} \subset L_x^2 (\h^3)$. Suppose that
\begin{align}
\lim_{k \to \infty} \norm{f_k}_{L_x^2 (\h^3)} = B \quad \text{and} \quad \lim_{k \to \infty} \norm{e^{it\Delta_{\h^3}} f_k}_{L_{t,x}^{\frac{10}{3}} (\R \times \h^3)} = \varepsilon .
\end{align}
Then exists a subsequence in $k$, $\{ N_k\} \subset (0, \infty)$, and $\{ t_k, h_k\} \in \R \times \mathbb{G}$ so that along the subsequence, we have the following:
\begin{enumerate}
\item If $\lim_k N_k = \infty$, then there exists $\phi \in L_x^2 (\R^3)$ such that 
    \begin{enumerate}
\item
    $T^{*}_{N_k}(\Pi_{-t_k, h^{-1}_k} f_k) (x) \rightharpoonup  \phi(x) \quad \text{weakly in } L_x^2 (\R^3)$,
\item 
$\lim_{k \to \infty} \norm{f_k}_{L_x^2(\h^3)}^2 - \norm{f_k - \phi_k}_{L_x^2(\h^3)}^2 = \norm{\phi}_{L_x^2(\R^3)}^2 \gtrsim B^2 (\frac{\varepsilon}{B})^{20} = \varepsilon^2 (\frac{\varepsilon}{B})^{18}$,
\item 
$\limsup_{k \to \infty} \norm{e^{it\Delta_{\h^3} } (f_k - \phi_k)}_{L_{t,x}^{\frac{10}{3}}(\h^3)} \leq \varepsilon [1 - c (\frac{\varepsilon}{B})^{30}]^{\frac{3}{10}}$,
\end{enumerate}
where $c$ and $\beta$ are constants and 
\begin{align}
\phi_k (x) : = \Pi_{t_k, h_k}(T_{N_k}\phi)(x).
\end{align}
\item If $\lim_k N_k <\infty$, then there exists $\phi \in L_x^2 (\h^3)$ such that 
    \begin{enumerate}
\item
    $\Pi_{-t_k, h^{-1}_k} f_k (x) \rightharpoonup  \phi(x) \quad \text{weakly in } L_x^2 (\h^3)$,
\item 
$\lim_{k \to \infty} \norm{f_k}_{L_x^2(\h^3)}^2 - \norm{f_k - \phi_k}_{L_x^2(\h^3)}^2 = \norm{\phi}_{L_x^2(\h^3)}^2 \gtrsim B^2 (\frac{\varepsilon}{B})^{20} = \varepsilon^2 (\frac{\varepsilon}{B})^{18}$.
\item 
$\limsup_{k \to \infty} \norm{e^{it\Delta_{\h^3} } (f_k - \phi_k)}_{L_{t,x}^{\frac{10}{3}}(\h^3)} \leq \varepsilon [1 - c (\frac{\varepsilon}{B})^{30}]^{\frac{3}{10}}$,
\end{enumerate}
where $c$ and $\beta$ are constants and 
\begin{align}
\phi_k (x) : = \Pi_{t_k, h_k}\phi(x).
\end{align}
\end{enumerate}

\end{proposition}

\begin{proof}[Proof of Proposition \ref{prop InvStrichartz}]
Passing to a subsequence, we may assume that
\begin{align}
\lim_{k \to \infty} \norm{f_k}_{L_x^2 (\h^3)} \leq  2B \quad \text{and} \quad \lim_{k \to \infty} \norm{e^{it\Delta_{\h^3}} f_k}_{L_{t,x}^{\frac{10}{3}} (\R \times \h^3)} \geq \frac{1}{2}\varepsilon .
\end{align}

By Proposition \ref{prop Improved}, there exists $\{N_k, t_k, x_k\}_{k=1}^{\infty} \subset \R_+ \times \R \times \h^3 $ such that
\begin{align}
\varepsilon^{10 } \lesssim \liminf_{k \to \infty} |N_k^{-\frac{3}{2}} e^{i t_k \Delta_{\h^3}} P_{N_k} f_k(x_k)| B^{9} ,\\
\varepsilon^{10} B^{-9} \lesssim \liminf_{k \to \infty} |N_k^{-\frac{3}{2}} e^{i t_k \Delta_{\h^3}} P_{N_k} f_k(x_k) | .
\end{align}
By taking subsequences we can assume that $\{N_k\}$ is either $N_k \to \infty$ or $N_k \to N$ for some $N \in \mathbb{R}_+$. 

If $N_k \to \infty$, let $h_k\cdot \0 := x_k$ and
\begin{align}
g_k : = T^{*}_{N_k}(\Pi_{-t_k, h^{-1}_k} f_k) \in L^2(\R^3).
\end{align}
We see that
\begin{align}
\norm{g_k}_{L_x^2 (\R^3)} \lesssim \norm{f_k}_{L_x^2 (\h^3)}  \lesssim B .
\end{align}
Then choose $\phi$ such that $g_k \rightharpoonup \phi$ weakly in $L_x^2(\R^3)$ (Alaoglu’s theorem). If $h \in L_x^2(\R^3)$, then
\begin{align}
\abs{\inner{h, \phi}_{L^2_x(\R^3)}} & = \lim_{k \to \infty} \abs{\inner{ h,  T^{*}_{N_k}(\Pi_{-t_k, h^{-1}_k} f_k )  }_{L^2_x(\R^3)}} \\
&\gtrsim \lim_{k \to \infty} \abs{\inner{ \Pi_{0, h_k} T_{N_k}h,  \Pi_{-t_k, \mathcal{I}} f_k   }_{L^2_x(\h^3)}} .
\end{align}
If we let $h := e^{\Delta_{\mathbb{R}^3}} \delta_0$, then $\Pi_{0, h_k} T_{N_k}h -N_k^{-\frac{3}{2}}P_{N_k}\delta_{x_k} \to 0$ in $L^2(\h^3)$. Therefore, 
\begin{align}
\lim_{k \to \infty} \abs{\inner{ \Pi_{0, h_k} T_{N_k}h,  \Pi_{-t_k, \mathcal{I}} f_k   }_{L^2_x(\h^3)}} & = \lim_{k \to \infty}  \abs{\inner{ N_k^{-\frac{3}{2}}P_{N_k}\delta_{x_k}, e^{it_k \Delta_{\h^3}} f_k}_{L^2_x(\h^3)}} \\
& = \lim_{k \to \infty} N_k^{-\frac{3}{2}} \abs{[e^{it_k \Delta_{\h^3}}P_{N_k} f_k](x_k)} \gtrsim B (\frac{\varepsilon}{B})^{10}.
\end{align}
Therefore, $\|\phi\|_{L^2_x(\R^3)} \gtrsim B (\frac{\varepsilon}{B})^{10}$.
A similar computation yields the following estimate for $e^{it\Delta_{\R^3}} \phi$:
    \begin{align*}
        B (\frac{\varepsilon}{B})^{10}\lesssim \norm{ e^{it\Delta_{\R^3}} \phi}_{L^{\frac{10}{3}}_{t,x}(\R \times \R^3)} .
    \end{align*}
By local smoothing estimate and the Rellich–Kondrashov Theorem, 
\begin{align}
e^{it\Delta_{\h^3}}  f_k (x) - e^{it\Delta_{\h^3}} \phi_k (x) \to 0\quad \text{a.e. } (t,x) \in \R \times \h^3 .
\end{align}
and by Refined Fatou Lemma  (Lemma \ref{lem Fatou}), we have
\begin{align}
\norm{e^{it\Delta_{\h^3} } f_k }_{L_{t,x}^{\frac{10}{3}} (\R \times \h^3)}^{\frac{10}{3}} - \norm{e^{it\Delta_{\h^3} } (f_k - \phi_k)}_{L_{t,x}^{\frac{10}{3}} (\R \times \h^3)}^{\frac{10}{3}} - \norm{e^{it\Delta_{\h^3} } \phi_k }_{L_{t,x}^{\frac{10}{3}} (\R \times \h^3)}^{\frac{10}{3}} \to 0.
\end{align}
This implies  
\begin{align*}
    \limsup_{k \to \infty}\norm{e^{it\Delta_{\h^3} } (f_k - \phi_k)}_{L_{t,x}^{\frac{10}{3}} (\R \times \h^3)}^{\frac{10}{3}}&\leq   \limsup_{k \to \infty}\norm{e^{it\Delta_{\h^3} } f_k }_{L_{t,x}^{\frac{10}{3}} (\R \times \h^3)}^{\frac{10}{3}}- \norm{e^{it\Delta_{\R^3} } \phi }_{L_{t,x}^{\frac{10}{3}} (\R \times \R^3)}^{\frac{10}{3}} \\
    &\leq \varepsilon^{\frac{10}{3}}(1- c (\tfrac{\varepsilon}{B})^{30}).
\end{align*}

Now if $N_k \to N$, let 
\begin{align}
g_k : = \Pi_{-t_k, h^{-1}_k} f_k \in L^2(\h^3).
\end{align}
We see that
\begin{align}
\norm{g_k}_{L_x^2 (\h^3)} = \norm{f_k}_{L_x^2 (\h^3)}  \lesssim B .
\end{align}
Then choose $\phi$ such that $g_k \rightharpoonup \phi$ weakly in $L_x^2(\h^3)$. If $h \in L_x^2(\h^3)$, then
\begin{align}
\abs{\inner{h, \phi}_{L^2_x(\h^3)}} & = \lim_{k \to \infty} \abs{\inner{ h,  \Pi_{-t_k, h^{-1}_k} f_k  }_{L^2_x(\R^3)}} \gtrsim \lim_{k \to \infty} \abs{\inner{ \Pi_{0, h_k} h,  \Pi_{-t_k, \mathcal{I}} f_k   }_{L^2_x(\h^3)}} .
\end{align}
If we let $h := N^{-\frac{3}{2}}P_{N}\delta_{0}$, then 
    \begin{align*}
        \Pi_{0, h_k} h -N_k^{-\frac{3}{2}}P_{N_k}\delta_{x_k} =N^{-\frac{3}{2}}P_{N}\delta_{x_k} -N_k^{-\frac{3}{2}}P_{N_k}\delta_{x_k}\to 0
    \end{align*}
    in $L^2(\h^3)$. Therefore, 
\begin{align}
\lim_{k \to \infty} \abs{\inner{ \Pi_{0, h_k} h,  \Pi_{-t_k, \mathcal{I}} f_k   }_{L^2_x(\h^3)}} & = \lim_{k \to \infty}  \abs{\inner{ N_k^{-\frac{3}{2}}P_{N_k}\delta_{x_k}, e^{it_k \Delta_{\h^3}} f_k}_{L^2_x(\h^3)}} \\
& = \lim_{k \to \infty} N_k^{-\frac{3}{2}} \abs{[e^{it_k \Delta_{\h^3}}P_{N_k} f_k](x_k)} \gtrsim B (\frac{\varepsilon}{B})^{10} .
\end{align}

A similar computation yields the following estimate for $e^{it\Delta_{\R^3}} \phi$:
    \begin{align*}
        B (\frac{\varepsilon}{B})^{10}\lesssim \norm{ e^{it\Delta_{\R^3}} \phi}_{L^{\frac{10}{3}}_{t,x}(\R \times \h^3)} .
    \end{align*}
By local smoothing estimate (Theorem \ref{thm Smoothing}) and the Rellich–Kondrashov Theorem, 
\begin{align}
e^{it\Delta_{\h^3}}  f_k (x) - e^{it\Delta_{\h^3}} \phi_k (x) \to 0\quad \text{a.e. } (t,x) \in \R \times \h^3
\end{align}
and by Refined Fatou Lemma  (Lemma \ref{lem Fatou}), we have
\begin{align}
\norm{e^{it\Delta_{\h^3} } f_k }_{L_{t,x}^{\frac{10}{3}} (\R \times \h^3)}^{\frac{10}{3}} - \norm{e^{it\Delta_{\h^3} } (f_k - \phi_k)}_{L_{t,x}^{\frac{10}{3}} (\R \times \h^3)}^{\frac{10}{3}} - \norm{e^{it\Delta_{\h^3} } \phi_k }_{L_{t,x}^{\frac{10}{3}} (\R \times \h^3)}^{\frac{10}{3}} \to 0.
\end{align}
This implies  
\begin{align*}
    \limsup_{k \to \infty}\norm{e^{it\Delta_{\h^3} } (f_k - \phi_k)}_{L_{t,x}^{\frac{10}{3}} (\R \times \h^3)}^{\frac{10}{3}}&\leq   \limsup_{k \to \infty}\norm{e^{it\Delta_{\h^3} } f_k }_{L_{t,x}^{\frac{10}{3}} (\R \times \h^3)}^{\frac{10}{3}}- \norm{e^{it\Delta_{\R^3} } \phi }_{L_{t,x}^{\frac{10}{3}} (\R \times \h^3)}^{\frac{10}{3}} \\
    &\leq \varepsilon^{\frac{10}{3}}(1- c (\tfrac{\varepsilon}{B})^{30}) .
\end{align*}
Now we finish the proof of Proposition \ref{prop InvStrichartz}.
\end{proof}

\subsection{Frames}

\begin{definition}\label{defn Profiles}
\begin{enumerate}
\item
We define a {\it frame} to be a sequence $\oo_k = (N_k , t_k ,h_k) \in [1, \infty) \times \R \times \G$, $ k \in \N^+$, where $ N_k \geq 1$ is a scale, $t_k \in \R$ is a time, and $h_k \in \G$ is a translation element. We also assume that either $N_k =1$ for all $k$ (in which case we call $\{\oo_k \}_{k \geq 1}$  a hyperbolic frame) or that $N_k \nearrow \infty$ (in which case we call $\{\oo_k \}_{k \geq 1}$  a Euclidean frame). Let $\F_e$ denote the set of Euclidean frames,
\begin{align}
\F_e = \{ \oo = \{ (N_k , t_k ,h_k)\}_{k \geq 1} : N_k \in [1 , \infty), N_k  \nearrow \infty , t_k \in \R , h_k \in \G  \} .
\end{align}
and let  $\F_h$ denote the set of hyperbolic frames,
\begin{align}
\F_h = \{ \widetilde{\oo} = \{ (1 , t_k ,h_k)\}_{k \geq 1} :  t_k \in \R , h_k \in \G  \} .
\end{align}

\item
We say that two frames $\{ (N_k , t_k ,h_k)\}_{k \geq 1} $ and $\{ (N_k' , t_k' ,h_k')\}_{k \geq 1} $ are {\it orthogonal} if 
\begin{align}\label{eq Orth}
\lim_{k \to \infty} \square{\abs{\ln (\frac{N_k}{N_k'}) }  + N_k^2 \abs{t_k -t_k'} + N_k d(h_k \cdot \0, h_k' \cdot \0))} = + \infty.
\end{align}
Two frames that are not orthogonal are called {\it equivalent}.

\item
Given $\phi \in L^2 (\R^3)$ and a Euclidean frame $\oo = \{ \oo_k\}_{k \geq 1}= \{ (N_k , t_k ,h_k)\}_{k \geq 1} \in \F_e$, we define {\it the Euclidean profile associated with} $(\phi , \oo)$ as the sequence $\widetilde{\phi}_{\oo_k}$, where
\begin{align}\label{eq 5.5}
\widetilde{\phi}_{\oo_k} : = \Pi_{t_k ,h_k} (T_{N_k} \phi),
\end{align}
The operators $\Pi$ and $T$ are defined in \eqref{eq Pi} and \eqref{eq T}.

\item
Given $\psi \in L^2 (\h^3)$ and a hyperbolic frame $\widetilde{\oo} = \{ \widetilde{\oo}_k\}_{k \geq 1}= \{ (1 , t_k ,h_k)\}_{k \geq 1} \in \F_h$ we define {\it the hyperbolic profile associated with} $(\psi , \widetilde{\oo} )$ as the sequence $\widetilde{\psi}_{\oo_k}$, where
\begin{align}\label{eq 5.6}
\widetilde{\psi}_{\oo_k} : = \Pi_{t_k ,h_k} \psi,
\end{align}

\end{enumerate}

\end{definition}

\begin{definition}
We say a sequence $(f_k)_k$ bounded in $L^2(\h^3)$ is {\it absent} from a frame $\oo = \{ (N_k , t_k ,h_k)\}_{k}$ if its localization to $\oo$ converges weakly to $0$, that is, if for any profile $\widetilde{\phi}_{\oo_k}$ associated to $\oo$, we have
\begin{align}\label{eq 5.7}
\lim_{k \to \infty} \inner{f_k , \widetilde{\phi}_{\oo_k}}_{L^2 \times L^2 (\h^3)} =0 . 
\end{align}
\end{definition}

\begin{remark}
\begin{enumerate}
\item
If $\oo = \{ (1 , t_k ,h_k)\}_{k}$ is a hyperbolic frame, this is equivalent to saying that
\begin{align}
\Pi_{-t_k , h_k^{-1}} f_k \rightharpoonup 0
\end{align}
as $k \to \infty$ in $L^2 (\h^3)$.

\item
If $\oo = \{ (N_k , t_k ,h_k)\}_{k}$ is a Euclidean frame, this is equivalent to saying that for all $R >0$
\begin{align}
\mathbf{g}_k^R (v) = \eta (\frac{v}{R}) N_k^{-\frac{3}{2}} (\Pi_{-t_k , h_k^{-1}}f_k) (\Psi_{\mathcal{I}} (\frac{v}{N_k})) \rightharpoonup 0
\end{align}
as $k \to \infty$ in $L^2 (\R^3)$.

\item
If $\oo$ is a Euclidean frame and $\oo'$ is a hyperbolic frame, then the two frames are orthogonal.
\end{enumerate}
\end{remark}

\begin{lemma}\label{lem EquiOrth}
\begin{enumerate}
\item
Assume $\{ \oo_k \}_{k \geq 1} = \{ (N_k , t_k ,h_k)\}_{k \geq 1}$ and $\{ \oo_k' \}_{k \geq 1} = \{ (N_k' , t_k' ,h_k')\}_{k \geq 1}$ are two equivalent Euclidean frames (or hyperbolic frames), and $\phi \in L^2 (\R^3)$ (or $\phi \in L^2 (\h^3)$). Then there is $\phi' \in L^2 (\R^3)$ (or $\phi' \in L^2 (\h^3)$) such that, up to a subsequence,
\begin{align}\label{eq 5.8}
\lim_{k \to \infty} \norm{\widetilde{\phi}_{\oo_k} - \widetilde{\phi}_{\oo'_k}'  }_{L^2 (\h^3)} =0,
\end{align}
where $\widetilde{\phi}_{\oo_k}$, $\widetilde{\phi}_{\oo'_k}'$ are defined as in Definition \ref{defn Profiles}.

\item
Assume $\{ \oo_k \}_{k \geq 1} = \{ (N_k , t_k ,h_k)\}_{k \geq 1}$ and $\{ \oo_k' \}_{k \geq 1} = \{ (N_k' , t_k' ,h_k')\}_{k \geq 1}$ are two orthogonal frames (either Euclidean or hyperbolic) and $\widetilde{\phi}_{\oo_k}$, $\widetilde{\psi}_{\oo'_k}$ are associated profiles. Then
\begin{align}\label{eq 5.9}
\lim_{k \to \infty} \abs{\int_{\h^3} \widetilde{\phi}_{\oo_k} \overline{\widetilde{\psi}_{\oo'_k}} \, d \mu} =0 .
\end{align}

\item
If $\widetilde{\phi}_{\oo_k}$, $\widetilde{\psi}_{\oo_k}$ are two Euclidean profiles associated to the same frame, then 
\begin{align}
\lim_{k \to \infty} \inner{ \widetilde{\phi}_{\oo_k} ,  \widetilde{\psi}_{\oo_k}}_{L^2 \times L^2 (\h^3)} = \lim_{k \to \infty} \abs{\int_{\h^3} \widetilde{\phi}_{\oo_k} \overline{\widetilde{\psi}_{\oo_k}} \, d \mu} = \int_{\R^3} \phi(x) \cdot \overline{\psi(x)} \, dx  =  \inner{ \phi,  \psi}_{L^2 \times L^2 (\R^3)}.
\end{align}
\end{enumerate}
\end{lemma}

\begin{proof}[Proof of Lemma \ref{lem EquiOrth}]
\begin{enumerate}
\item
We will prove the claim in the following subcases.

\textbf{Case 1:} In the case of two hyperbolic frames,  $\{ \oo_k\}_{k\geq 1}$ and $\{ \oo_k'\}_{k\geq 1}$, by passing to a subsequence we may assume $\lim_{k \to \infty} - t_k' + t_k = \overline{t}$ and $\lim_{k \to \infty} h_k'^{-1} h_k = \overline{h}$, and define
\begin{align}
\phi' : = \Pi_{\overline{t}, \overline{h}} \phi.
\end{align}

\textbf{Case 2:} In the case that $\{ \oo_k\}_{k\geq 1}$ and  $\{ \oo_k'\}_{k\geq 1}$ are equivalent Euclidean frames, we decompose $h_k'^{-1} h_k$ using the Cartan decomposition \eqref{eq Cartan}
\begin{align}\label{eq 5.10}
h_k'^{-1} h_k = m_k a_{s_k} n_k, \quad m_k , n_k \in \mathbb{K}, \quad s_k \in [0, \infty) .
\end{align}
Therefore, using the compactness of the subgroup $\mathbb{K}$ and the definition \eqref{eq Orth}, after passing to a subsequence, we may assume that
\begin{align}\label{eq 5.11}
\lim_{k \to \infty} \frac{N_k}{N_k'} = \overline{N}, \quad \lim_{k \to \infty} N_k^2 (t_k - t_k') = \overline{t} , \quad \lim_{k \to \infty}  m_k = m, \quad \lim_{k \to \infty}  n_k = n, \quad \lim_{k \to \infty}  N_k s_k = \overline{s} .
\end{align}
We observe that for any $N \geq 1$, $\psi \in L^2 (\R^3)$, $t\in \R$, $g \in \mathbb{G}$, and $q \in \mathbb{K}$
\begin{align}
\Pi_{t, g q}(T_{N} \psi)=\Pi_{t, g}(T_{N} \psi_{q}), \quad \text { where } \psi_{q}(x)=\psi (q^{-1} \cdot x ) .
\end{align}
Therefore, in \eqref{eq 5.10} we may assume that
\begin{align}
m_{k}=n_{k}=\mathcal{I}, \quad h_{k}'^{-1} h_{k}=a_{s_{k}} .
\end{align}
With $ \overline{x}=(\overline{s}, 0,0) $, we define
\begin{align}
\phi'(x):=\overline{N}^{\frac{3}{2}}(e^{-i \overline{t} \Delta} \phi)(\overline{N} x-\overline{x}), \quad \phi' \in L^2 (\mathbb{R}^{3}),
\end{align}
and define $ \widetilde{\phi}_{\mathscr{O}_{k}'}' $ as in \eqref{eq 5.5}. The identity \eqref{eq 5.8} is equivalent to
\begin{align}\label{eq 5.12}
\lim _{k \to \infty}\left\|T_{N_{k}'} \phi'-\pi_{h_{k}'^{-1} h_{k}} e^{i(t_{k}'-t_{k}) \Delta_{\h^3}}(T_{N_{k}} \phi)\right\|_{L^2 (\mathbb{H}^{3})}=0 .
\end{align}
To prove \eqref{eq 5.12} we may assume that $\phi' \in C_{0}^{\infty}(\mathbb{R}^{3}), \phi \in H^{5}(\mathbb{R}^{3})$, and apply Lemma \ref{lem ProfileAppx} (2) with $ \rho=0 $. Let $ v(t,x)=(e^{i t \Delta_{\R^3}} \phi)(x)$ and, for $ R \geq 1 $,
\begin{align}
v_{R}(t,x)=\eta(\frac{x}{R}) v(t,x), \quad v_{R, N_{k}}(t,x)=N_{k}^{\frac{3}{2}} v_{R}(N_{k}^{2} t , N_{k} x), \quad V_{R, N_{k}}(t, y)=v_{R, N_{k}}(t, \Psi_{\mathcal{I}}^{-1}(y)) .
\end{align}
It follows from Lemma \ref{lem ProfileAppx} (2) that for any $ \varepsilon>0 $ sufficiently small there is $ R_{0} $ sufficiently large such that, for any $ R \geq R_{0} $,
\begin{align}\label{eq 5.13}
\limsup _{k \to \infty}\left\|e^{i(t_{k}'-t_{k}) \Delta_{\h^3}}(T_{N_{k}} \phi)-V_{R, N_{k}}(t_{k}'-t_{k})\right\|_{L^2(\mathbb{H}^{3})} \leq \varepsilon .
\end{align}
Therefore, to prove \eqref{eq 5.12} it suffices to show that, for $ R $ large enough,
\begin{align}\label{eq 5.135}
\limsup_{k \to \infty}\left\|\pi_{h_{k}^{-1}h_{k}'} (T_{N_{k}'} \phi')-V_{R, N_{k}}(t_{k}'-t_{k})\right\|_{L^2(\mathbb{H}^{3})} \lesssim \varepsilon,
\end{align}
which, after examining the definitions and recalling that $ \phi' \in C_{0}^{\infty}(\mathbb{R}^{3}) $, is equivalent to
\begin{align}
\limsup _{k \to \infty}\left\|N_{k}'^{ \frac{3}{2}} \phi'(N_{k}' \Psi_{\mathcal{I}}^{-1}(h_{k}'^{-1} h_{k} \cdot y))-N_{k}^{\frac{3}{2}} v_{R}( N_{k}^{2}(t_{k}'-t_{k}), N_{k} \Psi_{\mathcal{I}}^{-1}(y))\right\|_{L_{y}^{2}(\mathbb{H}^{3})} \lesssim \varepsilon .
\end{align}
After changing variables $ y=\Psi_{\mathcal{I}}(x) $ this is equivalent to
\begin{align}
\limsup _{k \to \infty}\left\|N_{k}'^{ \frac{3}{2}} \phi'(N_{k}' \Psi_{\mathcal{I}}^{-1}(h_{k}'^{-1} h_{k} \cdot \Psi_{\mathcal{I}}(x)))-N_{k}^{\frac{3}{2}} v_{R}( N_{k}^{2}(t_{k}'-t_{k}), N_{k} x)\right\|_{L_{x}^{2}(\mathbb{R}^{3})} \lesssim \varepsilon .
\end{align}
Since, by definition, $ \phi'(z)=\overline{N}^{\frac{3}{2}} v(-\overline{t} , \overline{N} z-\overline{x})$, this follows provided that
\begin{align}
\lim _{k \to \infty} N_{k} \Psi_{\mathcal{I}}^{-1}(h_{k}'^{-1} h_{k} \cdot \Psi_{\mathcal{I}}(x / N_{k}))-x=\overline{x} \quad \text { for any } x \in \mathbb{R}^{3} .
\end{align}
This last claim follows from explicit computations using \eqref{eq 5.11} and the definition \eqref{eq Diffeo}. Finally, for arbitrarily small $\varepsilon>0$, combining \eqref{eq 5.13} and \eqref{eq 5.135}
\begin{align}
   &\limsup_{k \to \infty}\left\|T_{N_{k}'} \phi'-\pi_{h_{k}'^{-1} h_{k}} e^{i(t_{k}'-t_{k}) \Delta_{\h^3}}(T_{N_{k}} \phi)\right\|_{L^2 (\mathbb{H}^{3})}\\
   &\leq \limsup _{k \to \infty}\left\|e^{i(t_{k}'-t_{k}) \Delta_{\h^3}}(T_{N_{k}} \phi)-V_{R, N_{k}}(t_{k}'-t_{k})\right\|_{L^2(\mathbb{H}^{3})}\\
   &\hspace{3cm}+\limsup _{k \to \infty}\left\|\pi_{h_{k}^{-1}h_{k}'} (T_{N_{k}'} \phi')-V_{R, N_{k}}(t_{k}'-t_{k})\right\|_{L^2(\mathbb{H}^{3})}\lesssim \varepsilon.
\end{align}

\item
It suffices to prove that one can extract a subsequence such that \eqref{eq 5.9} holds. We analyze three cases:

\textbf{Case 1: $\oo, \oo' \in \mathscr{F}_{h} $.}  
By taking arbitrarily close approximations, we may assume that $ \phi, \psi \in C_{0}^{\infty}(\mathbb{H}^{3}) $ and select a subsequence such that either
\begin{align}\label{eq 5.14}
\lim _{k \to \infty}\left|t_{k}-t_{k}'\right|=\infty
\end{align}
or
\begin{align}\label{eq 5.15}
\lim _{k \to \infty} t_{k}-t_{k}'=\overline{t} \in \mathbb{R} \quad\mbox{ and } \quad \lim _{k \to \infty} d(h_{k} \cdot \mathbf{0}, h_{k}' \cdot \mathbf{0})=\infty .
\end{align}
Using \eqref{eq Dispersive}, it follows that
\begin{align}
\left\|\Pi_{t, h} \phi\right\|_{L^{6}(\mathbb{H}^{3})} &\lesssim_{\phi}(1+|t|)^{-1} ,\\
 \left\|\Pi_{t, h} \psi\right\|_{L^{6}(\mathbb{H}^{3})} & \lesssim_{\psi}(1+|t|)^{-1},
\end{align}
for any $ t \in \mathbb{R} $ and $ h \in \mathbb{G} $. Thus
\begin{align}
\abs{\int_{\h^3} \widetilde{\phi}_{\oo_k} \cdot \overline{ \widetilde{\psi}_{\oo_k'} } \, d\mu} & = \abs{\int_{\h^3} \pi_{h_k'^{-1} h_k} e^{-i (t_k - t_k') \Delta_{\h^3}}  \phi \cdot \overline{\psi} \, d\mu} \\
& \lesssim\left\|\pi_{h_{k}'^{-1} h_{k}} e^{-i(t_{k}-t_{k}') \Delta_{\h^3}} \phi \right\|_{L^{6}(\mathbb{H}^{3})}\|\psi\|_{L^{6 / 5}(\mathbb{H}^{3})} \lesssim_{\phi, \psi}(1+\left|t_{k}-t_{k}'\right|)^{-1} .
\end{align}
The claim \eqref{eq 5.9} follows if the selected subsequence satisfies \eqref{eq 5.14}.

If the selected subsequence satisfies \eqref{eq 5.15} then, as before,
\begin{align}
\left| \int_{\mathbb{H}^{3}} \widetilde{\phi}_{\oo_{k}} \overline{\widetilde{\psi}_{\oo_{k}'}} \, d \mu \right|& =\left|\int_{\mathbb{H}^{3}} \pi_{h_{k}'^{-1} h_{k}} e^{-i(t_{k}-t_{k}') \Delta_{\h^3}} \phi \cdot \overline{ \psi} \, d \mu\right| \\
& \lesssim\left\| \psi\right\|_{L^{2}(\mathbb{H}^{3})} \cdot\left\|e^{-i \overline{t} \Delta_{\h^3}} \phi-e^{-i(t_{k}-t_{k}') \Delta_{\h^3}} \phi\right\|_{L^{2}(\mathbb{H}^{3})}+\int_{\mathbb{H}^{3}}\left|e^{-i \overline{t} \Delta_{\h^3}} \phi\right| \cdot\left|\pi_{h_{k}^{-1} h_{k}'}  \psi\right| \, d \mu .
\end{align}
The limit in \eqref{eq 5.9} follows. 

\textbf{Case 2: $ \oo \in \mathscr{F}_{h}, \oo' \in \mathscr{F}_{e} $.} 
We may assume that $ \phi \in C_{0}^{\infty}(\mathbb{H}^{3}) $ and $ \psi \in C_{0}^{\infty}(\mathbb{R}^{3}) $. We estimate
\begin{align}
\left|\int_{\mathbb{H}^{3}} \widetilde{\phi}_{\mathscr{O}_{k}} \overline{\widetilde{\psi}_{\mathscr{O}_{k}'}} \, d \mu\right| = \left|\int_{\mathbb{H}^{3}} \Pi_{t_{k}, h_{k}}\phi \cdot \overline{\Pi_{t_{k}', h_{k}'}(T_{N_{k}'} \psi)} \, d \mu\right| \lesssim_{\phi} (1+|t_k-t_k'|)^{-1}\left\|T_{N_{k}'} \psi\right\|_{L^{6/5}(\mathbb{H}^{3})} \lesssim_{\phi, \psi} N_{k}'^{-1} .
\end{align}
The limits in \eqref{eq 5.9} follow.

\textbf{Case 3: $ \oo, \oo' \in \mathscr{F}_{e} $.} 
We may assume that $ \phi, \psi \in C_{0}^{\infty}(\mathbb{R}^{3}) $ and select a subsequence such that either
\begin{align}\label{eq 5.17}
\lim _{k \to \infty} \frac{N_{k}}{N_{k}'} =0
\end{align}
or
\begin{align}\label{eq 5.18}
\lim _{k \to \infty} \frac{N_{k}}{N_{k}'} =\overline{N} \in(0, \infty), \quad \lim _{k \to \infty} N_{k}^{2}\left|t_{k}-t_{k}'\right|=\infty
\end{align}
or
\begin{align}\label{eq 5.19}
\lim _{k \to \infty} \frac{N_{k}}{N_{k}'}=\overline{N} \in(0, \infty), \quad \lim _{k \to \infty} N_{k}^{2}(t_{k}-t_{k}')=\overline{t} \in \mathbb{R}, \quad \lim _{k \to \infty} N_{k} d(h_{k} \cdot \mathbf{0}, h_{k}' \cdot \mathbf{0})=\infty .
\end{align}
Assuming \eqref{eq 5.17} we estimate, as in Case 2,
\begin{align}
\left|\int_{\mathbb{H}^{3}}  \widetilde{\phi}_{\oo_{k}} \overline{\widetilde{\psi}_{\oo_{k}'}} \, d \mu\right| & =\left|\int_{\mathbb{H}^{3}} \Pi_{t_{k}, h_{k}}(T_{N_{k}} \phi) \cdot \overline{\Pi_{t_{k}', h_{k}'}(T_{N_{k}'} \psi)} \, d \mu\right| \\
& \lesssim\left\|T_{N_{k}} \phi\right\|_{L^{6}(\mathbb{H}^{3})}\left\|T_{N_{k}'} \psi\right\|_{L^{6/5}(\mathbb{H}^{3})} \lesssim_{\phi, \psi} N_{k} N_{k}'^{-1} .
\end{align}
The limits in \eqref{eq 5.9} follow in this case.

To prove the limit \eqref{eq 5.9}  assuming \eqref{eq 5.18}, we estimate first, using \eqref{eq Dispersive},
\begin{align}\label{eq 5.20}
\left\|\Pi_{t, h}(T_{N} f)\right\|_{L^{6}(\mathbb{H}^{3})} \lesssim_{f}(1+N^{2}|t|)^{-1},
\end{align}
for any $ t \in \mathbb{R}, h \in \mathbb{G}, N \in[0, \infty) $, and $ f \in C_{0}^{\infty}(\mathbb{R}^{3}) $. Thus
\begin{align}
\abs{\int_{\h^3}  \widetilde{\phi}_{\oo_k} \cdot \overline{\widetilde{\psi}_{\oo_k'} } \, d\mu} & = \left|\int_{\mathbb{H}^{3}} \pi_{h_{k}'^{-1} h_{k}}  e^{-i(t_{k}-t_{k}') \Delta_{\h^3}}(T_{N_{k}} \phi) \cdot \overline{T_{N_{k}'} \psi} \, d \mu\right| \\
& \lesssim\left\|\pi_{h_{k}'^{-1} h_{k}}  e^{-i(t_{k}-t_{k}') \Delta_{\h^3}}(T_{N_{k}} \phi)\right\|_{L^{6}(\mathbb{H}^{3})}\left\|T_{N_{k}'} \psi \right\|_{L^{\frac{6}{5}}(\mathbb{H}^{3})} \\
& \lesssim_{\phi, \psi}(N'_k)^{-1}(1+N_{k}^{2}\left|t_{k}-t_{k}'\right|)^{-1} .
\end{align}
The claim \eqref{eq 5.9} follows if the selected subsequence verifies \eqref{eq 5.18}.

Finally, it remains to prove the limit \eqref{eq 5.9} if the selected subsequence verifies \eqref{eq 5.19}. For this, we will use the following claim:
\begin{claim}\label{claim 5.21}
If $ (g_{k}, M_{k})_{k \geq 1} \in \mathbb{G} \times[1, \infty), \lim _{k \to \infty} M_{k}=\infty, \lim _{k \to \infty} M_{k} d(g_{k} \cdot \mathbf{0}, \mathbf{0})=\infty $, and $ f, g \in L^2(\mathbb{R}^{3}) $ then
\begin{align}
\lim _{k \to \infty}\left|\int_{\mathbb{H}^{3}} \pi_{g_{k}}(T_{M_{k}} f) \cdot(T_{M_{k}} g) \, d \mu\right| =0 .
\end{align}
\end{claim}

Assuming Claim \ref{claim 5.21}, we can complete the proof of \eqref{eq 5.9}. It follows from \eqref{eq 5.12} that if $ f \in L^2(\mathbb{R}^{3}) $ and $ \left\{s_{k}\right\}_{k \geq 1} $ is a sequence with the property that $ \lim _{k \to \infty} N_{k}^{2} s_{k}=\overline{s} \in \mathbb{R} $ then
\begin{align}\label{eq 5.22}
\lim _{k \to \infty}\left\|e^{-i s_{k} \Delta_{\h^3}}(T_{N_{k}} f)-T_{N_{k}'} f'\right\|_{L^2 (\mathbb{H}^{3})}=0,
\end{align}
where $ f'(x)=\overline{N}^{\frac{3}{2}}(e^{-i \overline{s} \Delta_{\R^3}} f)(\overline{N} x)$. We estimate
\begin{align}
\left|\int_{\mathbb{H}^{3}} \widetilde{\phi}_{\oo_{k}} \overline{\widetilde{\psi}_{\oo_{k}'}} \, d \mu\right| & = \left|\int_{\mathbb{H}^{3}}\pi_{h_{k}'^{-1} h_{k}}  e^{-i(t_{k}-t_{k}') \Delta_{\h^3}}(T_{N_{k}} \phi) \cdot \overline{T_{N_{k}'} \psi} \, d \mu\right| \\
& \lesssim\left|\int_{\mathbb{H}^{3}} \pi_{h_{k}'^{-1}h_{k}} (T_{N_{k}'} \phi') \cdot \overline{T_{N_{k}'} \psi} \, d \mu\right| \\
& \quad +\|\psi\|_{L^2(\mathbb{R}^{3})} \cdot\left\|\pi_{h_{k}'^{-1} h_{k}} e^{-i(t_{k}-t_{k}') \Delta_{\h^3}}(T_{N_{k}} \phi)-\pi_{h_{k}'^{-1} h_{k}}(T_{N_{k}'} \phi')\right\|_{L^2 (\mathbb{H}^{3})} .
\end{align}
In view of Claim \ref{claim 5.21} and \eqref{eq 5.22}, both terms in the expression above converge to $0$ as $k \to \infty $, as desired.

It remains to prove Claim \ref{claim 5.21}. In view of the $L^2 (\R^3) \to L^2 (\h^3)$ boundedness of the operators $ T_{N} $, we may assume that $ f, g \in C_{0}^{\infty}(\mathbb{R}^{3}) $ and replace $ T_{M_{k}} f $ and $ T_{M_{k}} g $ by $ M_{k}^{\frac{3}{2}} f(M_{k} \Psi_{\mathcal{I}}^{-1}(x)) $ and $ M_{k}^{\frac{3}{2}} g(M_{k} \Psi_{\mathcal{I}}^{-1}(x)) $ respectively, up to small errors. Then we notice that the supports of these functions become disjoint for $ k $ sufficiently large (due to the assumption $ \lim _{k \to \infty} M_{k} d(g_{k} \cdot \mathbf{0}, \mathbf{0})=\infty $). The limit in Claim \ref{claim 5.21} follows.

\item  
By the boundedness of $ T_{N_{k}}$, it suffices to consider the case when $ \phi, \psi \in C_{0}^{\infty}(\mathbb{R}^{3}) $. In this case, we have
\begin{align}
\left\|T_{N_{k}} \phi-N_{k}^{3 / 2} \phi(N_{k} \Psi_{\mathcal{I}}^{-1} \cdot)\right\|_{L^{2}(\mathbb{H}^{3})} \to 0
\end{align}
as $ k \to \infty $. Hence, by the unitarity of $ \Pi_{t_{k}, h_{k}} $, it suffices to compute
\begin{align}
\lim _{k \to \infty} N^3_{k}\left\langle \phi(N_{k} \Psi_{\mathcal{I}}^{-1} \cdot), \psi(N_{k} \Psi_{\mathcal{I}}^{-1} \cdot) \right\rangle_{L^{2} \times L^{2}(\mathbb{H}^{3})}=\int_{\mathbb{R}^{3}}  \phi(x) \cdot \overline{\psi}(x) d x,
\end{align}
which follows after a change of variables and the use of the dominated convergence theorem.
\end{enumerate}
The proof of Lemma \ref{lem EquiOrth} is complete. 
\end{proof}

\subsection{Profile decomposition}

\begin{proposition}\label{prop ProfileDecomp}
Assume that $(f_k)_{k \geq 1}$ is a bounded sequence in $L^2 (\h^3)$. Then there are sequences of pairs $(\phi^{\mu} , \oo^{\mu} ) \in L^2 (\R^3) \times \F_e$ and $(\psi^{\nu} , \widetilde{\oo}^{\nu} ) \in L^2 (\h^3) \times \F_h$, $\mu, \nu \in \N^+$ such that, up to a subsequence, for any ${\Lambda} \geq 1$,
\begin{align}
f_k = \sum_{1 \leq \mu \leq {\Lambda}} \widetilde{\phi}_{\oo_k^{\mu}}^{\mu}  + \sum_{1 \leq \nu \leq {\Lambda}} \widetilde{\psi}_{\oo_k^{\nu}}^{\nu} + r_k^{\Lambda},
\end{align}
where $\widetilde{\phi}_{\oo_k^{\mu}}^{\mu} $ and $\widetilde{\psi}_{\oo_k^{\nu}}^{\nu} $ are the associated profiles in Definition \ref{defn Profiles} and 
\begin{align}\label{eq ErrSmall}
\lim_{{\Lambda} \to \infty} \limsup_{k \to \infty}  \norm{e^{it\Delta_{\h^3}} r_k^{\Lambda}}_{L_{t,x}^{\frac{10}{3}} (\R \times \h^3)} =0 .
\end{align}
Moreover the frames $\{ \oo^{\mu} \}_{\mu \geq 1}$ and $\{ \widetilde{\oo}^{\nu} \}_{\nu \geq 1}$ are pairwise orthogonal. Finally, the decomposition is asymptotically orthogonal in the sense that
\begin{align}\label{eq ProfileDecomp}
\lim_{{\Lambda} \to \infty} \limsup_{k \to \infty} \abs{\norm{f_k}_{L^2 (\h^3)}^2 - \sum_{1 \leq \mu \leq {\Lambda}}   \norm{\widetilde{\phi}_{\oo_k^{\mu}}^{\mu} }_{L^2 (\h^3)}^2 - \sum_{1 \leq \nu \leq {\Lambda}}   \norm{\widetilde{\psi}_{\oo_k^{\nu}}^{\nu} }_{L^2 (\h^3)}^2  - \norm{r_k^{\Lambda}}_{L^2 (\h^3)}^2 } =0 .
\end{align}
\end{proposition}

Proposition \ref{prop ProfileDecomp} is a consequence of the following finitary decomposition.

\begin{lemma}\label{lem ProfileDecomp}
Let $(f_k)_{k \geq 1}$ be a bounded sequence of functions in $L^2 (\h^3)$ and $\delta \in (0 ,\delta_0]$ be sufficiently small. Up to passing to a subsequence, the sequence $(f_k)_{k \geq 1}$ can be decomposed into $2{\Lambda} +1 = O(\delta^{-2})$ terms
\begin{align}
f_k = \sum_{1 \leq \mu \leq {\Lambda}} \widetilde{\phi}_{\oo_k^{\mu}}^{\mu}  + \sum_{1 \leq \nu \leq {\Lambda}} \widetilde{\psi}_{\oo_k^{\nu}}^{\nu} + r_k,
\end{align}
where $\widetilde{\phi}_{\oo_k^{\mu}}^{\mu} $ and $\widetilde{\psi}_{\oo_k^{\nu}}^{\nu} $ are Euclidean and hyperbolic profiles, respectively, associated to the sequences $(\phi^{\mu} , \oo^{\mu} ) \in L^2 (\R^3) \times \F_e$ and $(\psi^{\nu} , \widetilde{\oo}^{\nu} ) \in L^2 (\h^3) \times \F_h$, $\mu, \nu \in \N^+$ as in Definition \ref{defn Profiles}. 
Moreover the remainder $r_k$ is absent from all the frames $\oo^{\mu}$, $\widetilde{\oo}^{\nu}$, $1 \leq \mu, \nu \leq {\Lambda}$ and 
\begin{align}
\limsup_{k \to \infty}  \norm{e^{it\Delta_{\h^3}} r_k}_{L_{t,x}^{\frac{10}{3}} (\R \times \h^3)} \leq \delta.
 \end{align}
In addition, the frames $\oo^{\mu}$ and $\widetilde{\oo}^{\nu}$ are pairwise orthogonal, and the decomposition is asymptotically orthogonal in the sense that 
\begin{align}
\norm{f_k}_{L^2 (\h^3)}^2 = \sum_{1 \leq \mu \leq {\Lambda}}   \norm{\widetilde{\phi}_{\oo_k^{\mu}}^{\mu} }_{L^2 (\h^3)}^2 + \sum_{1 \leq \nu \leq {\Lambda}}   \norm{\widetilde{\psi}_{\oo_k^{\nu}}^{\nu} }_{L^2 (\h^3)}^2  + \norm{r_k^{\Lambda}}_{L^2 (\h^3)}^2  + o_k (1)
\end{align}
where $o_k(1) \to 0$ as $k \to \infty$.
\end{lemma}

Assuming Lemma \ref{lem ProfileDecomp}, we first prove Proposition \ref{prop ProfileDecomp}. 
\begin{proof}[Proof of Proposition \ref{prop ProfileDecomp}]
We apply Lemma \ref{lem ProfileDecomp} repeatedly for $\delta = 2^{-l}$, $l = 1,2 ,\cdots$ and we obtain Proposition \ref{prop ProfileDecomp}. 
\end{proof}

\begin{proof}[Proof of Lemma \ref{lem ProfileDecomp}]
For $(g_k)_k$, a bounded sequence in $L^2 (\h^3)$, we define
\begin{align}
\varepsilon ((g_k)_k)  := \limsup_{k \to \infty} \norm{e^{it\Delta_{\h^3}} g_k}_{L_{t,x}^{\frac{10}{3}} (\R \times \h^3)} .
\end{align}
If $\varepsilon ((g_k)_k)  \leq \delta$, then we let $\Lambda=0$ and $f_k = r_k$ and Lemma \ref{lem ProfileDecomp} follows. Otherwise, we use inductively the following important claim
\begin{claim}
Assume $(g_k)_k$ is a bounded sequence in $L^2 (\h^3)$ which is absent from a family of frames $(\oo^{\alpha})_{\alpha \leq A}$ and such that $\varepsilon((g_k)_k) \geq \delta$. Then, after passing to a subsequence, there exists a new frame $\oo$ which is orthogonal to $\oo^{\alpha}$ for all $\alpha \leq A$ and a profile $\widetilde{\phi}_{\oo_k'}$ of mass
\begin{align}
\lim_{k \to \infty} \norm{\widetilde{\phi}_{\oo_k}}_{L_x^2} \gtrsim \delta
\end{align}
such that $g_k - \widetilde{\phi}_{\oo_k}$ is absent from the frames $\oo$ and $\oo^{\alpha}$, $\alpha \leq A$. 
\end{claim}

By Proposition \ref{prop InvStrichartz}, there are two possibilities: (1) there exists $\phi \in L^2_x(\R^3)$, $\{ N_k\} \subset (0, \infty)$, $\{ t_k, h_k\} \in \R \times \mathbb{G}$ and
    \begin{align*}
        \phi_k (x) : = \Pi_{t_k, h_k} [T_{N_k}\phi](x)
    \end{align*}
or (2) there exists $\phi \in L^2_x(\h^3)$, $\{ t_k, h_k\} \in \R \times \mathbb{G}$ and
    \begin{align*}
        \phi_k (x) : = \Pi_{t_k, h_k} \phi(x) .
    \end{align*}
In either case, Proposition \ref{prop InvStrichartz} implies
    \begin{align*}
        \lim_{k \to \infty}|\langle g_k, \phi_k \rangle|=\lim_{k \to \infty}\|\phi_k\|_{L^2_x}=\|\phi\|_{L^2_x} \gtrsim B^{-9}\varepsilon((g_k)_k)^{10} \geq B^{-9}\delta^{10} ,
    \end{align*}
where $B:= \limsup_{k \to \infty} \|g_k\|_{L^2_x}$. Since $g_k$ is absent from $\oo^{\alpha}$ for all $\alpha \leq A$ and $B^{-9}\delta^{10}>0$,  Lemma \ref{lem EquiOrth} coupled with the previous inequality imply that $\{ N_k, t_k, h_k\}$ is orthogonal to $\oo^{\alpha}$ for all $\alpha \leq A$.  Therefore, $g_k - \widetilde{\phi}_{\oo_k}=g_k - \phi_k$  is also absent from $(\oo^{\alpha})_{\alpha \leq A}$.

Furthermore, Proposition \ref{prop InvStrichartz}, parts 1(b) and 2(b) imply
    \begin{align*}
        \langle g_k - \phi_k, \phi_k \rangle=\langle g_k - \phi_k, \widetilde{\phi}_{\oo_k} \rangle \to 0. 
    \end{align*}
Therefore, $g_k - \phi_k$ is also absent from $\oo$.

Now that the claim has been established, one proceeds with the standard induction argument for profile decomposition. The proof of Lemma \ref{lem ProfileDecomp} is complete.
\end{proof}


\section{Proof of Proposition \ref{prop key}}\label{sec Key}

\begin{proof}[Proof of Proposition \ref{prop key}]

 Using the time translation symmetry, we may assume that $t_k=0$ for all $k \geq 1$. We apply Proposition \ref{prop ProfileDecomp} to the sequence $(u_k(0))_k$ which is bounded in $L^2(\mathbb{H}^3)$ and we get sequences of pairs $(\phi^\mu, \oo^\mu) \in L^2 (\mathbb{R}^3) \times \mathscr{F}_e$ and $(\psi^{\nu}, \widetilde{\oo}^{\nu}) \in L^2 (\mathbb{H}^3 ) \times \mathscr{F}_h$, $\mu, v=1,2, \ldots$, such that the conclusion of Proposition \ref{prop ProfileDecomp} holds. Up to using Lemma \ref{lem EquiOrth}  (1), we may assume that for all $\mu$, either $t_k^\mu=0$ for all $k$ or $(N_k^\mu)^2\left|t_k^\mu\right| \to \infty$ and similarly, for all $v$, either $t_k^{\nu}=0$ for all $k$ or $\left|t_k^{\nu}\right| \to \infty$.

\textbf{Case 1: All profiles are trivial, $\phi^\mu=0, \psi^{\nu}=0$ for all $\mu, v$.}

In this case, we get from Proposition \ref{prop ProfileDecomp} that $u_k(0)=r_k^{\Lambda}$ satisfies
\begin{align}
\left\|e^{i t \Delta_{\h^3}}u_k(0)\right\|_{Z(\R)}  \to 0
\end{align}
as $k \to \infty$. Applying Lemma \ref{lem 6.1}, we see that
\begin{align}
\left\|u_k\right\|_{Z(\R)} \leq\left\|e^{i t \Delta_{\h^3}} u_k(0)\right\|_{L_{t, x}^{\frac{10}{3}}(\R \times \h^3 )}+ \left\|u_k-e^{i t \Delta_{\h^3}} u_k(0)\right\|_{S_{\h^3}^0(\R)} \to 0
\end{align}
as $k \to \infty$, which contradicts \eqref{eq Contradiction}.

Proceeding to the remaining cases: for every linear profile $\widetilde{\phi}_{\oo_k^\mu}^\mu$ (resp. $\widetilde{\psi}_{\widetilde{\oo}_k^{\nu}}^{\nu}$), define the associated nonlinear profile $U_{e, k}^\mu$ (resp. $U_{h, k}^{\nu}$) as the maximal solution of \eqref{NLS} with initial data $U_{e, k}^\mu(0)=\widetilde{\phi}_{\oo_k^\mu}^\mu$ (resp. $U_{h, k}^\nu(0)=\widetilde{\psi}_{\widetilde{\oo}_k^{\nu}}^{\nu}$). We may write $U_k^\gamma$ if we do not want to discriminate between Euclidean and hyperbolic profiles.

The nonlinear profiles are defined in the following way:
    \begin{enumerate}
        \item If $\oo^\mu \in \mathscr{F}_e$ is a Euclidean frame, this is given in Lemma \ref{lem 6.2}.
        \item  If $t_k^{\nu}=0$, letting $(I^{\nu}, W^{\nu})$ be the maximal solution of \eqref{NLS} with initial data $W^{\nu}(0)=\psi^{\nu}$ (maximal in the sense of Definition \ref{def order}), we see that for any interval $J \Subset I^{\nu}$,
    \begin{align}\label{eq 6.2}
        \left\|U_{h, k}^{\nu}(t)-\pi_{h_k}^{\nu} W^{\nu}(t-t_k^{\nu})\right\|_{S_{\h^3}^0(J)} \to 0
    \end{align}
as $k \to \infty$ (indeed, this is identically 0 in this case).
        \item If $t_k^{\nu} \to +\infty$, then we define  $(I^{\nu}, W^{\nu})$ to be the maximal solution of \eqref{NLS} satisfying 
\begin{align}
\left\|W^{\nu}(t)-e^{i t \Delta_{\h^3}} \psi^{\nu}\right\|_{L^2(\h^3)} \to 0
\end{align}
as $t \to -\infty$. Then, applying Proposition \ref{prop Stab}, we see that on any interval $J=(-\infty, T) \Subset I^{\nu}$, we have \eqref{eq 6.2}. Using the time reversal symmetry $u(t, x) \to \overline{u}(-t, x)$, we obtain a similar description when $t_k^{\nu} \to -\infty$.
\end{enumerate}

\textbf{Case 2a: There is only one Euclidean profile, i.e., there exists $\mu$ such that $u_k(0)=\widetilde{\phi}_{\oo_k^\mu}^\mu+o_k(1)$ in $L^2(\mathbb{H}^3)$.} 

Applying Lemma \ref{lem 6.2}, we see that $U_{e, k}^\mu$ is global with uniformly bounded $S_{\h^3}^0$-norm for $k$ large enough. Then, using the stability Proposition \ref{prop Stab} with $\widetilde{u}=U_{e, k}^\mu$, we see that for all $k$ large enough,
\begin{align}
\left\|u_k\right\|_{Z(I)} \lesssim_{M_{\max}} 1
\end{align}
which contradicts \eqref{eq Contradiction}.

\textbf{Case 2b: There is only one hyperbolic profile, i.e., there is $\nu$ such that $u_k(0)=\widetilde{\psi}_{\widetilde{\oo}_k^{\nu}}^{\nu}+o_k(1)$ in $L^2(\h^3)$.} 

If $t_k^{\nu} \to+\infty$, then, using Strichartz estimates, we see that
\begin{align*}
    \left\|e^{i t \Delta_{\h^3}} \Pi_{t_k^{\nu}, h_k^{\nu}} \psi^{\nu}\right\|_{Z((-\infty, 0)) }=\left\|e^{i t \Delta_{\h^3}}  \psi^{\nu}\right\|_{Z((-\infty, -t_k^{\nu}))} \to 0
\end{align*}
as $k \to \infty$, which implies that $\left\|e^{i t \Delta_{\h^3}} u_k(0)\right\|_{Z((-\infty, 0))} \to 0$ as $k \to \infty$. Using again Lemma \ref{lem 6.1}, we see that, for $k$ large enough, $u_k$ is defined on $(-\infty, 0)$ and $\left\|u_k\right\|_{Z((-\infty, 0))} \to 0$ as $k \to \infty$, which contradicts \eqref{eq Contradiction}. Similarly, $t_k^{\nu} \to-\infty$ yields a contradiction. Finally, if $t_k^{\nu}=0$, we get that
    \begin{align*}
        \Pi_{0, h_k^{-1}} u_k \to \psi^{\nu}
    \end{align*}
converges strongly in $L^2(\h^3)$, which is the desired conclusion of the proposition.

\textbf{Case 3: There exists $\mu$ or $v$ and $\eta>0$ such that
\begin{align}\label{eq 6.3}
2 \eta<\limsup _{k \to \infty} M (\widetilde{\phi}_{\oo_k^\mu}^\mu ), \quad \limsup _{k \to \infty} M(\widetilde{\psi}_{\oo_k^{\nu}}^{\nu} )<M_{\max}-2 \eta .
\end{align}}

Taking $k$ sufficiently large and maybe replacing $\eta$ by $\eta / 2$, we may assume that \eqref{eq 6.3} holds for all $k$. In this case, we claim that, for ${\Lambda}$ sufficiently large,
    \begin{align}\label{eq Uapp}
        U_k^{\mathrm{app}}=\sum_{1 \leq \mu \leq {\Lambda}} U_{e, k}^\mu+\sum_{1 \leq v \leq {\Lambda}} U_{h, k}^{\nu}+e^{i t \Delta_{\h^3}} r_k^{\Lambda}=U_{\mathrm{prof}, k}^{\Lambda}+e^{i t \Delta_k} r_k^{\Lambda}
    \end{align}
is a global approximate solution with bounded $Z$ norm for all $k$ sufficiently large.
First, by Lemma \ref{lem 6.2}, all the Euclidean profiles are global for $k$ large enough. Using \eqref{eq ProfileDecomp}, we see that for all $v$ and all $k$ sufficiently large, $M(U_{h, k}^{\nu})<M_{\max}-\eta$. By \eqref{eq 6.2}, this implies that $M (W^{\nu} )<M_{\max}-\eta$ so that by the definition of $M_{\max}, W^{\nu}$ is global and by Proposition \ref{prop Stab}, $U_{h, k}^{\nu}$ is global for $k$ large enough and
\begin{align}\label{eq 6.4}
\left\|U_{h, k}^{\nu}(t)-\pi_{h_k} W^{\nu} (t-t_k^{\nu} )\right\|_{S_{\h^3}^0(\R)} \to 0
\end{align}
as $k \to \infty$.
Now we claim that
\begin{align}\label{eq 6.5}
\underset{k \to \infty}{\limsup }\left\|  U_k^{\mathrm{app}}\right\|_{L_t^{\infty} L_x^2} \leq 2 M_{\max}^{\frac{1}{2}}
\end{align}
is bounded uniformly in ${\Lambda}$. Indeed, we first observe using \eqref{eq ProfileDecomp} that
\begin{align}
\left\|  U_k^{\mathrm{app}}\right\|_{L_t^{\infty} L_x^2} & \leq\left\|  U_{\text {prof}, k}^{\Lambda}\right\|_{L_t^{\infty} L_x^2}+\left\| r_k^{\Lambda}\right\|_{L_x^2}  \leq\left\| U_{\text {prof}, k}^{\Lambda}\right\|_{L_t^{\infty} L_x^2}+ M_{\max}^{\frac{1}{2}} .
\end{align}

Using Lemma \ref{lem 6.3}, we get that for fixed $t$ and ${\Lambda}$,
\begin{align}
\left\| U_{\mathrm{prof}, k}^{\Lambda}(t)\right\|_{L_x^2}^2 & \leq \sum_{1 \leq \gamma \leq 2 {\Lambda}}\left\| U_k^\gamma\right\|_{L_t^{\infty} L_x^2}^2+2 \sum_{\gamma \neq \gamma'}\left\langle U_k^\gamma(t),  U_k^{\gamma'}(t)\right\rangle_{L^2 \times L^2} \\
& \leq  \sum_{1 \leq \gamma \leq 2 {\Lambda}} M(U_k^\gamma)+o_k(1) \leq  M_{\max}+o_k(1),
\end{align}
where $o_k(1) \to 0$ as $k \to \infty$ for fixed ${\Lambda}$.

We also have
\begin{align}\label{eq 6.6}
\limsup _{k \to \infty}\left\| U_k^{\mathrm{app}}\right\|_{L_{t,x}^{\frac{10}{3}}} \lesssim_{M_{\max}, \eta} 1
\end{align}
is bounded uniformly in ${\Lambda}$ by Lemma \ref{lem nonlineardecoupling}. Indeed, from \eqref{eq 6.3} and \eqref{eq ProfileDecomp}, we see that for all $\gamma$ and all $k$ sufficiently large (depending maybe on ${\Lambda}$), $M(U_k^\gamma)<M_{\max}-\eta$ and from the definition of $M_{\max}$, we conclude that $U^{\gamma}_k$ exists globally and therefore,
using Proposition \ref{prop Stab}, we see that this implies that
\begin{align}\label{eq 6.7}
\sup _\gamma\left\|  U_k^\gamma\right\|_{L_{t, x}^{\frac{10}{3}}} \lesssim_{M_{\max} , \eta} 1 .
\end{align}

Now, using Lemma \ref{lem nonlineardecoupling} we know that 
\begin{align}
\left\| U_{\mathrm{prof}, k}^{\Lambda}\right\|_{L_{t, x}^{\frac{10}{3}}}^{\frac{10}{3}} \lesssim_{M_{\max}, \eta} 1.
\end{align}
Using triangle inequality and \eqref{eq ErrSmall}, we get \eqref{eq 6.6}.

Using \eqref{eq 6.5} and \eqref{eq 6.6} we can apply Proposition \ref{prop Stab} to $U_k^{\mathrm{app}}$ get $\varepsilon_1>0$ such that the conclusion of Proposition \ref{prop Stab} holds for $u_k$ and $U_k^{\mathrm{app}}$.
In particular,  for $F(x)=|x|^{\frac{4}{3}} x$, we have
\begin{align}
e= (i \partial_t+\Delta_g) U_k^{\mathrm{app}}-\left|U_k^{\mathrm{app}}\right|^{\frac{4}{3}} U_k^{\mathrm{app}}=\sum_{1 \leq \alpha \leq 2 {\Lambda}}((i \partial_t+\Delta_g) U_k^\alpha-F(U_k^\alpha))+\sum_{1 \leq \alpha \leq 2 {\Lambda}} F(U_k^\alpha)-F(U_k^{\mathrm{app}}) .
\end{align}
The first term is identically 0, while using Lemma \ref{lem 6.4}, we see that taking ${\Lambda}$ large enough, we can ensure that the $L_{t, x}^{10/7} $-norm of the second term is smaller than $\varepsilon_1$ for all $k$ large enough. Then, since $u_k(0)=U_k^{\text {app}}(0)$, the conclusion of Proposition \ref{prop Stab} implies that for all $k$ large, and any interval, $J$,
\begin{align}
\left\|u_k\right\|_{Z(J)} \lesssim_{M_{\max , \eta}} 1 ,
\end{align}
where we have used \eqref{eq 6.6}. Then, we see that $u_k$ is global for all $k$ large enough and that $u_k$ has uniformly bounded $Z$-norm, which contradicts \eqref{eq Contradiction}. This ends the proof of Proposition \ref{prop key}.
\end{proof}

\subsection{Criterion for linear evolution}

\begin{lemma}\label{lem 6.1}
For any $M > 0$, there exists $\delta >0$ such that for any interval $J \subset \R$, if
\begin{align}
\norm{\phi}_{L^2 (\h^3)} \leq M \quad \text{ and } \quad \norm{e^{it\Delta_{\h^3}} \phi}_{Z(J)} \leq \delta
\end{align}
then for any $t_0 \in J$, the maximal solution $(I ,u)$ of \eqref{NLS} with $u(t_0) = e^{it_0 \Delta_{\h^3}} \phi$ satisfies $J \subset I$ and
\begin{align}\label{eq 6.8}
\begin{aligned}
\norm{u - e^{it\Delta_{\h^3}} \phi}_{S_{\h^3}^0 (J)} \leq \delta , \\
\norm{u}_{S_{\h^3}^0 (J)} \leq C(M ,\delta) .
\end{aligned}
\end{align}
In addition, if $J = (-\infty ,T)$, then there exists a unique maximal solution $(I, u)$, $J \subset I$ of \eqref{NLS} such that
\begin{align}\label{eq 6.9}
\lim_{t \to - \infty} \norm{u(t) -e^{it\Delta_{\h^3}} \phi}_{L^2 (\h^3)} = 0
\end{align}
and \eqref{eq 6.8} holds in this case too. The same statement holds in the Euclidean case when $(\mathbb{H}^3, g)$ is replaced by $(\mathbb{R}^3, \delta_{i j})$.
\end{lemma}

\begin{proof}[Proof of Lemma \ref{lem 6.1}]
The first part is a direct consequence of Proposition \ref{prop Stab}. Indeed, let $v=e^{i t \Delta_{\h^3}} \phi$. Then using Strichartz estimates,
\begin{align}
\left\||v|^{\frac{4}{3}}v \right\|_{L_{t, x}^{\frac{10}{7}}(J \times \h^3)} = \left\| e^{i t \Delta_{\h^3}} \phi\right\|^{7/3}_{L_{t, x}^{\frac{10}{3}} (J \times \h^3)} \leq \delta^{\frac{7}{3}}.
\end{align}
Thus the assumptions in Proposition \ref{prop Stab} are satisfied. Then we can apply Proposition \ref{prop Stab} with $\rho=1$ to conclude. The second claim is classical and follows from a fixed point argument.
\end{proof}

\subsection{Description of a Euclidean nonlinear profile}

Let $\F_e$ denote the set of Euclidean frames,
\begin{align}
\widetilde{\F}_e = \{  (N_k , t_k ,h_k)_{k} \in \F_e : t_k =0 \text{ and all } k \text{ or } \lim_{k \to \infty}  N_k^2  \abs{t_k} = \infty  \} .
\end{align}
and let  $\F_h$ denote the set of hyperbolic frames,
\begin{align}
\widetilde{\F}_h = \{   (1 , t_k ,h_k)_{k}  \in \F_h :  t_k =0 \text{ and all } k \text{ or } \lim_{k \to \infty}   \abs{t_k} = \infty  \} .
\end{align}

\begin{lemma}\label{lem 6.2}
Assume $\phi \in L^2 (\mathbb{R}^3)$ and $(N_k, t_k, h_k)_k \in \widetilde{\mathscr{F}}_e$. Let $U_k$ be the solution of \eqref{NLS} such that $U_k(0)=\Pi_{t_k, h_k}(T_{N_k} \phi)$. Then there exists $\widetilde{C}=\widetilde{C}(M_{\R^3}(\phi))$ such that for $k$ large enough, $U_k \in C (\R; L^2(\h^3))$ is globally defined, and
\begin{align}\label{eq 6.10}
\left\|U_k\right\|_{Z(\R)} \leq  \widetilde{C}.
\end{align}

\end{lemma}

\begin{proof}[Proof of Lemma \ref{lem 6.2}]

We begin with the case $t_k=0$. We may assume that $h_k=\mathcal{I}$ for any $k$. Let $\delta>0$ and suppose $t_k=0$. Let $T= T(\phi, \delta)$ be large enough so that 
    \begin{align*}
\norm{e^{it\Delta_{\R^3}} \phi}_{L^{10/3}_{t, x}(\{|t| \geq T\} \times \R^3)} < \delta.
    \end{align*}
   For $k$ large enough (depending on $T$), Corollary \ref{cor ProfileAppx}, then implies
    \begin{align*}
\norm{e^{it\Delta_{\h^3}} T_{N_k}\phi}_{L^{10/3}_{t, x}(\{N^{-2}_k|t| \geq T\} \times \h^3)} \lesssim \delta.
    \end{align*}
If $\delta$ is chosen small enough, then Lemma \ref{lem 6.1} implies that 
    \begin{align}
        \|U_{k}\|_{S_{\h^3}^0(\{N^{-2}_k|t| \geq T\} \times \h^3)} \lesssim_{\phi} 1.
    \end{align}
Let $k=k(T)$ also be large enough so that $N_k$ is large enough to apply part (1) of Lemma \ref{lem ProfileAppx}. Lemma \ref{lem ProfileAppx} then implies that the remaining part of the solution norm is bounded. I.e. 
\begin{align}
\|U_{k}\|_{S_{\h^3}^0(\{N^{-2}_k|t| \leq T\} \times \h^3)} \lesssim_{\phi} 1 .
\end{align}

For the second case, we consider the Euclidean mass-critical NLS and a trajectory, $u \in C(\mathbb{R}; L^2 (\mathbb{R}^3))$, satisfying
\begin{align}\label{eq 6.11}
(i \partial_t+\Delta_{\R^3}) u= |u|^{\frac{4}{3}}u
\end{align}
with scattering data $u^{ \pm \infty}$ defined as in \eqref{eq Scattering}.

If $\lim _{k \to \infty} N_k^2\left|t_k\right|=\infty$, we may assume by symmetry that $N_k^2 t_k \to+\infty$. Then we let $u$ be the solution of \eqref{eq 6.11} such that
\begin{align}
\left\| u(t)-e^{i t \Delta_{\R^3}} \phi \right\|_{L^2(\R^3)} \to 0
\end{align}
as $t \to-\infty$ (thus $u^{-\infty}=\phi$). Therefore, 
\begin{align}\label{eq scat0}
 \limsup_k \left\| u(t_k)-e^{i t_k \Delta_{\R^3}} \phi \right\|_{L^2(\R^3)} = 0 .
\end{align}

We let $\widetilde{\phi}=u(0)$ and consider the frame $(N_k, 0, h_k)_k \in \mathscr{F}_e$ and $V_k(s)$, the solution of \eqref{NLS} with initial data $V_k(0)=\pi_{h_k} T_{N_k} \widetilde{\phi}$. From the previous case, we know that for $k$ large enough, $V_k$ is a global solution with 
\begin{align}
\left\|V_k\right\|_{Z(\R)} \lesssim_{\phi} 1 .
\end{align}
In particular, we see from the fact that $N_k^2 t_k \to+\infty$ and \eqref{eq scat0} that
\begin{align}
\left\|V_k (-t_k )-\Pi_{t_k, h_k} T_{N_k} \phi\right\|_{L^2(\h^3)} \to 0
\end{align}
as $k \to \infty$. Then, using Proposition \ref{prop Stab}, we see that
\begin{align}
\left\|U_k-V_k (\cdot-t_k)\right\|_{S_{\h^3}^0(\R)} \to 0
\end{align}
as $k \to \infty$ which completes the argument in Lemma \ref{lem 6.2}.
\end{proof}

\subsection{Noninteraction of nonlinear profiles}

\begin{lemma}\label{lem 6.3}
Let $\widetilde{\phi}_{\oo_k}$ and $\widetilde{\psi}_{\oo_k}$ be two profiles associated to orthogonal frames $\oo$ and $\oo'$ in $\widetilde{\F}_e \cup \widetilde{\F}_h$. Let $U_k$ and $U_k'$ be the solutions of the nonlinear equation \eqref{NLS} such that $U_k (0) = \widetilde{\phi}_{\oo_k}$ and $U_k' (0) = \widetilde{\psi}_{\oo_k'}$. Suppose also that $\norm{\widetilde{\phi}_{\oo_k}}_{L^2}^2 < M_{\max} - \eta$ (respectively, $\norm{\widetilde{\psi}_{\oo_k'}}_{L^2}^2 < M_{\max} - \eta$) if $\oo \in \F_h$ (respectively, $\oo' \in \F_h$). Then
\begin{align}\label{eq 6.15}
\sup_{T \in \R} \abs{ \inner{U_k (T) , U_k'(T)}_{L^2 \times L^2 (\h^3)}} + \norm{U_k U_k'}_{L_{t,x}^{\frac{5}{3}} (\R \times \h^3)}  \to 0
\end{align}  
as $k \to \infty$.
\end{lemma}

\begin{proof}[Proof of Lemma \ref{lem 6.3}]
It suffices to prove \eqref{eq 6.15} up to extracting a subsequence, and fix $\varepsilon>0$ sufficiently small.

We only provide the proof that the second norm in  \eqref{eq 6.15} decays; the first term follows similarly. Applying Lemma \ref{lem 6.2} if $U_k$ is a profile associated to a Euclidean frame (respectively \eqref{eq 6.4} if $U_k$ is a profile associated to a hyperbolic frame), we see that
\begin{align}
\left\|U_k\right\|_{S_{\h^3}^0} + \left\|U_k'\right\|_{S_{\h^3}^0} \leq c_0<+\infty
\end{align}
and that there exist $R$ and $\delta$ such that
\begin{align}\label{eq 6.16}
\begin{aligned}
\left\|U_k\right\|_{L_{t,x}^{10/3}\left(\left(\R \times \h^3 \right) \setminus \mathscr{S}_{N_k, t_k, h_k}^R\right)} +\left\|U'_k\right\|_{L_{t,x}^{10/3}\left(\left(\R \times \h^3 \right) \setminus \mathscr{S}_{N'_k, t'_k, h'_k}^R\right)}\leq \varepsilon, \\
\sup _{\tau, h}\left[\left\|U_k\right\|_{L_{t,x}^{10/3}\left(\mathscr{S}_{N_k, \tau, h}^{\delta}\right)}+\left\|U'_k\right\|_{L_{t,x}^{10/3}\left(\mathscr{S}_{N'_k, \tau, h}^{\delta}\right)}\right] \leq \varepsilon,
\end{aligned}
\end{align}
where
\begin{align}\label{eq 6.17}
\mathscr{S}_{N, T, h}^a:=\left\{(t, x) \in \mathbb{R} \times \mathbb{H}^3 : d (h^{-1} \cdot x, \mathbf{0}) \leq a N^{-1} \text { and }|t-T| \leq a^2 N^{-2}\right\} .
\end{align}
A similar claim holds for $U_k'$ with the same values of $R, \delta$.

If $N_k / N_k' \to \infty$, then for $k$ large enough we estimate

\begin{align}
\norm{U_k U_k'}_{L_{t,x}^{\frac{5}{3}} (\R \times \h^3)} & \lesssim \norm{U_k U_k'}_{L_{t,x}^{\frac{5}{3}} (\mathscr{S}_{N_k, t_k, h_k}^R)}  + \norm{U_k U_k'}_{L_{t,x}^{\frac{5}{3}} (\R \times \h^3 \setminus \mathscr{S}_{N_k, t_k, h_k}^R)}  \\
& \lesssim \norm{U_k}_{L_{t,x}^{\frac{10}{3}}} \norm{U_k'}_{L_{t,x}^{\frac{10}{3}} (\mathscr{S}_{N_k', t_k, h_k}^{\delta})} + \norm{U_k}_{L_{t,x}^{\frac{10}{3}} (\R \times \h^3 \setminus \mathscr{S}_{N_k, t_k, h_k}^R)} \norm{U_k'}_{L_{t,x}^{\frac{10}{3}} } \\
& \lesssim_{c_0} \varepsilon ,
\end{align}
where we used 
\begin{align}
\mathscr{S}_{N_k, t_k, h_k}^{R} \subset \mathscr{S}_{N_k', t_k, h_k}^{\delta} ,
\end{align}
since we can choose $k$ large enough so that $R N_k^{-1} < \delta N_k'^{-1}$.

The case when $N_k' / N_k \to \infty$ is similar.
Otherwise, we can assume that $C^{-1} \leq N_k / N_k' \leq C$ for all $k$, and then find $k$ sufficiently large that $\mathscr{S}_{N_k, t_k, h_k}^R \cap \mathscr{S}_{N_k', t_k', h_k'}^R=\varnothing$. Using  \eqref{eq 6.16} it follows as before that
\begin{align}
\norm{U_k U_k'}_{L_{t,x}^{\frac{5}{3}}}  \lesssim_{c_0} \varepsilon.
\end{align}

Hence, in all cases,
\begin{align}
\limsup_{k \to \infty}  \norm{U_k U_k'}_{L_{t,x}^{\frac{5}{3}} }  \lesssim_{c_0} \varepsilon .
\end{align}
The convergence to 0 of the second term in \eqref{eq 6.15} follows. 

The proof of Lemma \ref{lem 6.3} is complete now.
\end{proof}

\begin{lemma}\label{lem nonlineardecoupling}
With the notations in the proof of Proposition \ref{prop key} (See Eqn. \eqref{eq Uapp}), assuming that 
\begin{align}\label{eq 6.75}
\sup _\gamma\left\|  U_k^\gamma\right\|_{L_{t, x}^{\frac{10}{3}}} \lesssim_{M_{\max} , \eta} 1 ,
\end{align}
one can conclude that
\begin{align}
\left\| U_{\mathrm{prof}, k}^{\Lambda}\right\|_{L_{t, x}^{\frac{10}{3}}}^{\frac{10}{3}} \lesssim_{M_{\max}, \eta} 1.
\end{align}

\end{lemma}

\begin{proof}[Proof of Lemma \ref{lem nonlineardecoupling}]
    Using \eqref{eq 6.75} and Lemma \ref{lem 6.3}, we see that
\begin{align}
& \left|\left\|  U_{\mathrm{prof}, k}^{\Lambda}\right\|_{L_{t, x}^{\frac{10}{3}}}^{\frac{10}{3}}-\sum_{1 \leq \alpha \leq 2 {\Lambda}}\left\| U_k^\alpha\right\|_{L_{t, x}^{\frac{10}{3}}}^{\frac{10}{3}}\right| \leq \sum_{1 \leq \alpha \neq \beta \leq 2 {\Lambda}}\left\|( U_k^\alpha)^{\frac{7}{3}}  U_k^\beta\right\|_{L_{t, x}^1} \\
& \lesssim_{M_{\max}, \eta} \sum_{1 \leq \alpha \neq \beta \leq 2 {\Lambda}} \norm{U_k^{\alpha}}_{L_{t,x}^{\frac{10}{3}}}^{\frac{4}{3}} \left\|  U_k^\alpha  U_k^\beta\right\|_{L_{t, x}^{\frac{5}{3}}} \lesssim_{M_{\max}, \eta} o_k(1) . \label{eq decouple}
\end{align}

Additionally, using Lemma \ref{lem 6.1}, triangle inequality and linear Strichartz estimates for $\widetilde{\phi}_{\oo_k^\gamma}^\gamma$, we can see that for $\gamma$ large enough (depending on $M_{max}$) $M(U_k^\gamma) \leq \delta_0$ will be sufficiently small to conclude
\begin{align}\label{eq nonstrich}
\left\|  U_k^\gamma\right\|_{L_{t, x}^{\frac{10}{3}}}^2  \lesssim M(U_k^\gamma).
\end{align}
Combining inequalities \eqref{eq decouple} and \eqref{eq nonstrich}, along with the fact that $10/3\geq2$, we get
\begin{align}
\left\| U_{\mathrm{prof}, k}^{\Lambda}\right\|_{L_{t, x}^{\frac{10}{3}}}^{\frac{10}{3}} & \leq \sum_{1 \leq \alpha \leq 2 {\Lambda}}\left\| U_k^\alpha\right\|_{L_{t, x}^{\frac{10}{3}}}^{\frac{10}{3}}+o_k(1) \\
& \lesssim_{M_{\max , \eta}} C \sum_{1 \leq \alpha \leq 2 {\Lambda}} M(U_k^\alpha)+o_k(1) \lesssim_{M_{\max}, \eta} 1 .
\end{align}
Now we finish the proof of Lemma \ref{lem nonlineardecoupling}.
\end{proof}

\subsection{Control of the error term}

\begin{lemma}\label{lem 6.4}
With the notations in the proof of Proposition \ref{prop key} (See Eqn. \eqref{eq Uapp}),
\begin{align}\label{eq 6.18}
\lim_{{\Lambda} \to \infty} \limsup_{k \to \infty} \norm{F(U_k^{\mathrm{app}}) - \sum_{1 \leq \alpha \leq 2{\Lambda}} F(U_k^{\alpha})}_{L_{t,x}^{\frac{10}{7}}} =0 .
\end{align}
\end{lemma}

\begin{proof}[Proof of Lemma \ref{lem 6.4}]
Fix $\varepsilon_0>0$. For fixed ${\Lambda}$, we let
\begin{align}
U_{\mathrm{prof}, k}^{\Lambda}=\sum_{1 \leq \mu \leq {\Lambda}} U_{e, k}^\mu+\sum_{1 \leq \nu \leq {\Lambda}} U_{h, k}^{\nu}=\sum_{1 \leq \gamma \leq 2 {\Lambda}} U_k^\gamma  .
\end{align}
be the sum of the profiles. Then we separate
\begin{align}
\left\| F(U_k^{\mathrm{app}})-\sum_{1 \leq \alpha \leq 2 {\Lambda}} F(U_k^\alpha) \right\|_{L_{t,x}^{\frac{10}{7}}} & \leq\left\| F(U_k^{\mathrm{app}})-F(U_{\mathrm{prof}, k}^{\Lambda}) \right\|_{L_{t,x}^{\frac{10}{7}}}+\left\|  F(U_{\text {prof}, k}^{\Lambda})-\sum_{1 \leq \alpha \leq 2 {\Lambda}} F(U_k^\alpha) \right\|_{L_{t,x}^{\frac{10}{7}}} .
\end{align}

We first claim that for the second term and for fixed ${\Lambda}$,
\begin{align}\label{eq 6.19}
\limsup_{k \to \infty}\left\| F(U_{\mathrm{prof}, k}^{\Lambda})-\sum_{1 \leq \alpha \leq 2 {\Lambda}} F(U_k^\alpha) \right\|_{L_{t,x}^{\frac{10}{7}}} =0 .
\end{align}

Note that given $\{ f_i\}_{i=1}^{\Lambda} \subset\mathbb{C}$, we have the following inequality for $p >1$:
\begin{align}
\abs{\abs{\sum_{i=1}^{\Lambda} f_i}^p \sum_{i=1}^{\Lambda} f_i - \sum_{i=1}^{\Lambda} \abs{f_i}^p f_i}  \leq C(\Lambda, p) \abs{\sum_{i=1}^{\Lambda} (\sum_{j \neq i} \abs{f_j}^p) f_i}.
\end{align}
Then
\begin{align}
\norm{F (U_{\mathrm{prof}, k}^{\Lambda} )-\sum_{1 \leq \alpha \leq 2 {\Lambda}} F (U_k^\alpha )  }_{L_{t,x}^{\frac{10}{7}}} & = \norm{F (\sum_{1 \leq \alpha \leq 2 {\Lambda}} U_k^\alpha )-\sum_{1 \leq \alpha \leq 2 {\Lambda}} F (U_k^\alpha )}_{L_{t,x}^{\frac{10}{7}}} \\
& \lesssim_{\Lambda} \sum_{\alpha \neq \beta} \norm{ \abs{U_k^{\alpha}}^{\frac{4}{3}} U_k^{\beta} }_{L_{t,x}^{\frac{10}{7}}} \\
& \lesssim_{\Lambda} \sum_{\alpha \neq \beta} \norm{ U_k^{\alpha}}_{L_{t,x}^{\frac{10}{3}}}^{\frac{1}{3}} \norm{U_k^{\alpha} U_k^{\beta} }_{L_{t,x}^{\frac{5}{3}}}.
\end{align}

Therefore \eqref{eq 6.19} follows from \eqref{eq 6.15} since the sum is over a finite set and each profile is bounded in $L_{t, x}^{\frac{10}{3}}$ by \eqref{eq 6.7}.

We now complete the proof of \eqref{eq 6.18} by showing that for any given $\varepsilon_0>0$,
\begin{align}\label{eq 6.20}
\limsup _{{\Lambda} \to \infty} \limsup _{k \to \infty}\left\| F (U_k^{\text {app}} )-F (U_{\text {prof}, k}^{\Lambda} ) \right\|_{L_{t,x}^{\frac{10}{7}}} =0    .
\end{align}

We first remark that, from \eqref{eq 6.6}, $U_{\text {prof}, k}^{\Lambda}$ has bounded $L_{t, x}^{10/3}$ norm, uniformly in ${\Lambda}$ for $k$ sufficiently large. We also note that by the profile decomposition and Lemma \ref{lem 6.1} we have the following uniform bound 
\begin{align}\label{eq 6.21}
\sup _{\alpha } \limsup _{k \to \infty}\left\|U_k^\alpha\right\|_{L_{t, x}^{\frac{10}{3}}} \lesssim_{M_{max}} 1 .
\end{align}

A straightforward computation yields
\begin{align}
|F (U_{\mathrm{prof}, k}^{\Lambda}+e^{i t \Delta_{\h^3}} r_k^{\Lambda} )-&F (U_{\mathrm{prof}, k}^{\Lambda} )|
\\& =\abs{\abs{U_{\mathrm{prof}, k}^{\Lambda}+e^{i t \Delta_{\h^3}} r_k^{\Lambda}}^{\frac{4}{3}} ( U_{\mathrm{prof}, k}^{\Lambda}+e^{i t \Delta_{\h^3}} r_k^{\Lambda} ) - \abs{U_{\mathrm{prof}, k}^{\Lambda}}^{\frac{4}{3}} U_{\mathrm{prof}, k}^{\Lambda} }\\
& \lesssim  \abs{U_{\mathrm{prof}, k}^{\Lambda}}^{\frac{4}{3}}\abs{ e^{i t \Delta_{\h^3}} r_k^{\Lambda} }+ \abs{e^{i t \Delta_{\h^3}} r_k^{\Lambda}}^{\frac{4}{3}} \abs{ U_{\mathrm{prof}, k}^{\Lambda}} + \abs{e^{i t \Delta_{\h^3}} r_k^{\Lambda}}^{\frac{4}{3}}  \abs{e^{i t \Delta_{\h^3}} r_k^{\Lambda}} \\
& \lesssim \abs{e^{i t \Delta_{\h^3}} r_k^{\Lambda} }(\abs{U_{\mathrm{prof}, k}^{\Lambda}}^{\frac{4}{3}} + \abs{e^{i t \Delta_{\h^3}} r_k^{\Lambda}}^{\frac{4}{3}}) .
\end{align}

Now we compute
\begin{align}
\left\| F (U_{\mathrm{prof}, k}^{\Lambda}+e^{i t \Delta_{\h^3}} r_k^{\Lambda} )-F (U_{\mathrm{prof}, k}^{\Lambda} ) \right\|_{L_{t,x}^{\frac{10}{7}}} & \lesssim \norm{e^{i t \Delta_{\h^3}} r_k^{\Lambda}}_{L_{t,x}^{\frac{10}{3}}} \norm{\abs{U_{\mathrm{prof}, k}^{\Lambda}}^{\frac{4}{3}} + \abs{e^{i t \Delta_{\h^3}} r_k^{\Lambda}}^{\frac{4}{3}} }_{L_{t,x}^{\frac{10}{4}}} \\
& \lesssim \norm{e^{i t \Delta_{\h^3}} r_k^{\Lambda}}_{L_{t,x}^{\frac{10}{3}}} ( \norm{U_{\mathrm{prof}, k}^{\Lambda}}_{L_{t,x}^{\frac{10}{3}}}^{\frac{4}{3}}  + \norm{e^{i t \Delta_{\h^3}} r_k^{\Lambda}}_{L_{t,x}^{\frac{10}{3}}}^{\frac{4}{3}}  ) .
\end{align}
Since $\norm{e^{i t \Delta_{\h^3}} r_k^{\Lambda}}_{L_{t,x}^{\frac{10}{3}}} \to 0$ by Lemma \ref{prop ProfileDecomp}, it suffices to show $\norm{U_{\mathrm{prof}, k}^{\Lambda}}_{L_{t,x}^{\frac{10}{3}}} $ is bounded.  Lemma \ref{lem nonlineardecoupling} along with inequality \eqref{eq 6.21} imply that $\norm{U_{\mathrm{prof}, k}^{\Lambda}}_{L_{t,x}^{\frac{10}{3}}}$ is uniformly bounded.

Now we complete the proof of Lemma \ref{lem 6.4}.
\end{proof}


\section{Proof of Proposition \ref{prop MorDod}}\label{sec MorDod}

Let $u \in C(\mathbb{R}; L^2(\h^3))$ be an almost periodic (modulo $\mathbb{G}$) solution to \eqref{NLS} and let $T \in (0,\infty)$. For $\eta>0$, let $C_0:=C(\eta)$, where $C(\eta)$ is defined in Definition \ref{defn AP solution}.

 We first state an obvious consequence of our notion of almost periodicity: 
\begin{lemma}\label{lem etaSmall}
 Let $u \in C(\mathbb{R}; L^2(\h^3))$ be almost periodic modulo $\mathbb{G}$, and let $I \subset \R$. For any $\eta >0$,
\begin{align}
\lim_{N \to \infty}  \norm{P_{> N} u}_{L_t^{\infty} L_x^2(I \times \h^3)} +N^{-\frac{1}{2}} + \norm{P_{> \eta N^{\frac{1}{2}}}u}^{\frac{1}{3}}_{L_t^{\infty} L_x^2 (I \times \h^3)} = 0 .
\end{align}
\end{lemma}

The following is a  modification of Lemma 3.8 in \cite{Dod3d}
\begin{lemma}\label{claim 1}
Let $u \in C(\mathbb{R}; L^2(\h^3))$ be an almost periodic (modulo $\mathbb{G}$) solution to \eqref{NLS} and let $T \in (0,\infty)$. For $\eta>0$ and $0 \leq s \leq \frac{7}{3}$, there exists $c_{\eta}$ such that
\begin{align}
\norm{P_{\geq N} u}_{L_t^2 L_x^6([0, T] \times \h^3)} & \lesssim_s \norm{P_{\geq N} u(t_0)}_{L_x^2} + \sum_{M \leq \eta N} (\frac{M}{N})^s \norm{P_{\geq M} u}_{L_t^2 L_x^6([0, T] \times \h^3)} \\
&\quad + \eta^{\frac{4}{3}} \norm{P_{\geq \eta N} u}_{L_t^2 L_x^6([0, T] \times \h^3)}\\ 
& \quad + \frac{c_{\eta} T^{\frac{1}{2}}}{\eta^{\frac{1}{2}} N^{\frac{1}{2}}}  \left( \norm{P_{> \eta N} u }_{L_t^{\infty} L_x^2([0, T] \times \h^3)} + N^{-\frac{1}{2}} + \norm{P_{> \eta N^{\frac{1}{2}}} u }_{L_t^{\infty} L_x^2([0, T] \times \h^3)}^{\frac{1}{3}} \right) .
\end{align}
\end{lemma}

\begin{proof}[Proof of Lemma \ref{claim 1}]
Let  $I = [0, T]$.
Strichartz estimates and Duhamel's principle imply
\begin{align}\label{eq Duhamel}
\norm{P_{\geq N} u}_{L_t^2 L_x^6} \lesssim \norm{P_{\geq N} u(t_0))}_{L_x^2} + \norm{P_{\geq N} F(u)}_{L_t^2 L_x^{\frac{6}{5}}} ,
\end{align}
where $F(u) = \abs{u}^{\frac{4}{3}}u$.

To prove Lemma \ref{claim 1}, we only need to focus on the nonlinear estimate above. For each $t \in I$, define a cutoff $\chi_{h(t)}\in L^{\infty} (\h^3)$ in physical space,
\begin{align}
\chi_{h(t)}(x)= 
\begin{cases}
1,\,\,\, d(h(t) \cdot x, \0) < C(\eta)=C_0 ,\\
0,\,\,\, d(h(t) \cdot x, \0) \geq C(\eta)=C_0 .
\end{cases}
\end{align}
We note that the $L^2$ and $L^{\infty}$ invariance of $\pi_{h(t)}$ implies 
\begin{align*}
    \sup_{q \in [2, \infty]}  \norm{\chi_{h(t)} }_{L_t^{\infty} L_x^q (I \times \h^3)} =\sup_{q \in [2, \infty]}  \norm{\pi_{h(t)}\chi_{\mathcal{I}} }_{L_t^{\infty} L_x^q (I \times \h^3)}\lesssim C_{1}=C_1(\eta).
\end{align*}
 We split the nonlinear term into
\begin{align}
\begin{aligned}\label{eq 4}
&\norm{P_{\geq N} F(u)}_{L_t^2 L_x^{\frac{6}{5}} (I \times \h^3)} \\
&\quad \quad \quad \lesssim \norm{P_{\geq N} F(P_{\leq \eta N}u) }_{L_t^2 L_x^{\frac{6}{5}} (I \times \h^3)} + \norm{(P_{> \eta N} u) \abs{P_{>C_0} u}^{\frac{4}{3}}}_{L_t^2 L_x^{\frac{6}{5}} (I \times \h^3)} \\
& \quad \quad\quad\quad + \norm{(P_{> \eta N} u) (1-\chi_{h(t)}) \abs{u}^{\frac{4}{3}}}_{L_t^2 L_x^{\frac{6}{5}} (I \times \h^3)} + \norm{(P_{> \eta N} u) \chi_{h(t)} \abs{P_{\leq C_0 }u}^{\frac{4}{3}}}_{L_t^2 L_x^{\frac{6}{5}} (I \times \h^3)} .  
\end{aligned}
\end{align}

Now we compute it term by term. For the first term in \eqref{eq 4}, by Bernstein inequality, H\"older inequality and conservation of mass, we write
\begin{align}
 \norm{P_{\geq N} F(P_{\leq \eta N}u) }_{L_t^2 L_x^{\frac{6}{5}} (I \times \h^3)}  & \lesssim \frac{1}{N^s} \norm{\abs{\nabla}^{s} F(P_{\leq \eta N}u) }_{L_t^2 L_x^{\frac{6}{5}} (I \times \h^3)} \\
& \lesssim \frac{1}{N^s} \norm{\abs{\nabla}^{s} P_{\leq \eta N} u }_{L_t^2 L_x^{6} (I \times \h^3)} \norm{ P_{\leq \eta N} u }_{L_t^{\infty} L_x^{2} (I \times \h^3)}^{\frac{4}{3}}\\
& \lesssim \frac{1}{N^s} \norm{\abs{\nabla}^{s} P_{\leq \eta N} u }_{L_t^2 L_x^{6} (I \times \h^3)} \\
& \lesssim \sum_{M \leq \eta N} (\frac{M}{N})^s \norm{P_{M} u }_{L_t^2 L_x^{6} (I \times \h^3)} .
\end{align}

Then for the second term in \eqref{eq 4}
\begin{align}
\norm{(P_{> \eta N} u) \abs{P_{>C_0} u}^{\frac{4}{3}}}_{L_t^2 L_x^{\frac{6}{5}} (I \times \h^3)} & \lesssim \norm{P_{> \eta N} u}_{L_t^{2} L_x^6(I \times \h^3)} \norm{P_{>C_0} u}_{L_t^{\infty} L_x^{2} (I \times \h^3)}^{\frac{4}{3}} \\
& \lesssim \eta^{\frac{4}{3}} \norm{P_{> \eta N} u}_{L_t^{2} L_x^6(I \times \h^3)} ,
\end{align}
and the third term in \eqref{eq 4}
\begin{align}
\norm{(P_{> \eta N} u) (1-\chi_{h(t)}) \abs{u}^{\frac{4}{3}}}_{L_t^2 L_x^{\frac{6}{5}} (I \times \h^3)} & \lesssim  \norm{P_{> \eta N} u}_{L_t^{2} L_x^6(I \times \h^3)} \norm{(1- \chi_{h(t)}) u}_{L_t^{\infty} L_x^2 (I \times \h^3)}^{\frac{4}{3}}\\
& \lesssim \eta^{\frac{4}{3}} \norm{P_{> \eta N} u}_{L_t^{2} L_x^6(I \times \h^3)} .
\end{align}

Now take the last term in \eqref{eq 4}. Let  $I = [0, T] = \cup J_k$, where $J_k$ is an interval of local constancy as defined in Definition \ref{Defn LocCon}. Using H\"older inequality and  the nonlinear bilinear estimate in Lemma \ref{lem nonbilinear} (with $q=2$), we obtain
\begin{align}
&\norm{(P_{> \eta N} u) \chi_{h(t)} \abs{P_{\leq C_0 }u}^{\frac{4}{3}}}^2_{L_t^2 L_x^{\frac{6}{5}} (J_k \times \h^3)} \\
& \quad \quad\lesssim \norm{(P_{> \eta N} u) (P_{\leq C_0 } u) }^2_{L_{t,x}^2 (J_k \times \h^3)} \norm{\chi_{h(t)} }^2_{L_t^{\infty} L_x^6 (J_k \times \h^3)} \norm{u}_{L_t^{\infty} L_x^2(J \times \h^3)}^{\frac{2}{3}} \\
& \quad \quad \lesssim \frac{C_0}{\eta N}  C_1^2\norm{P_{> \eta N} u}^2_{S_*^0 (J_k \times \h^3)} \norm{u}^2_{S_*^0 (J_k \times \h^3)} , \label{eq 5}
\end{align}
where the $S_*^0$ norm is defined in Lemma \ref{lem nonbilinear}.

By \eqref{eq local}, we have
\begin{align}
\norm{u}_{S^0 (J_k)} \lesssim \norm{u_0}_{L^2 (\h^3)} + \norm{\abs{u}^{\frac{4}{3}} u}_{L_{t,x}^{\frac{10}{7}} (J_k \times \h^3)} = \norm{u_0}_{L^2 (\h^3)} + \norm{u}_{L_{t,x}^{\frac{10}{3}} (J_k \times \h^3)}^{\frac{7}{3}} \lesssim 1 ,
\end{align}
then
\begin{align}
\norm{u}_{S_*^0 (J_k \times \h^3)} = \norm{u_0}_{L^2 (\h^3)} + \norm{\abs{u}^{\frac{4}{3}} u}_{L_t^1 L_x^2 (J_k \times \h^3)} \lesssim \norm{u_0}_{L^2 (\h^3)} + \norm{u}_{L_t^{\frac{7}{3}} L_x^{\frac{14}{3}} (J_k \times \h^3)}^{\frac{7}{3}} \lesssim 1 .
\end{align}

Then
\begin{align}
\eqref{eq 5} & \lesssim\frac{C_0C_1^2}{\eta N} \norm{P_{> \eta N} u}^2_{S_*^0 (J_k \times \h^3)}  .
\end{align}

Lemma \ref{lem aprioriT} implies that there are at most a constant multiple of $T$ many subintervals $J_k$. Therefore, summing over the subintervals $J_k$, we obtain
\begin{align}
\norm{(P_{> \eta N} u) \chi_{h(t)} \abs{P_{\leq C_0 }u}^{\frac{4}{3}}}^2_{L_t^2 L_x^{\frac{6}{5}} (I \times \h^3)} & \lesssim \sum_{J_k \subset I} \norm{(P_{> \eta N} u) (P_{\leq C_0} u) }^2_{L_{t,x}^2 (J_k \times \h^3)}  \norm{\chi_{h(t)}}^2_{L_t^{\infty} L_x^6 (J_k \times \h^3)} \\
& \lesssim  \sum_{J_k \subset I}   \frac{C_0C_1^2}{\eta N} \norm{P_{> \eta N} u}_{S_*^0 (J_k \times \h^3)}  \\
& \lesssim \frac{C_0C_1^2}{\eta} \frac{T}{N} \left(\sup_{J_k} \norm{P_{>\eta N} u}^2_{S_*^0 (J_k \times \h^3)} \right) . \label{eq 3}
\end{align}

Next, we claim that for any $(q,r)$ admissible and any $k$
\begin{align}\label{eq S^0}
\norm{u}_{L_t^q L_x^r (J_k \times \h^3)} & \lesssim_q \norm{u}_{L_t^{\infty}L_x^2 (J_k \times \h^3)} \lesssim 1 .
\end{align}
In fact, we divide $J_k =  \cup J_k^\ell = \cup [a_\ell, b_\ell] $ with 
\begin{align}\label{eq 9}
\norm{u}_{L_{t,x}^{\frac{10}{3}}(J_k^\ell \times \h^3)} = \varepsilon .
\end{align}
Here we need $\varepsilon \ll 1$ and it will be determined later. Using the Duhamel principle, Strichartz estimates and \eqref{eq 9}, we write
\begin{align}
\norm{u}_{L_t^q L_x^r (J_k^\ell \times \h^3)} & \lesssim \norm{u(a_\ell)}_{L^2 (\h^3)} + \norm{ u}_{L_t^q L_x^r (J_k^\ell \times \h^3)} \norm{u}_{L_{t,x}^{\frac{10}{3}} (J_k^\ell \times \h^3)}^{\frac{4}{3}} \\
& \lesssim \norm{u(a_\ell)}_{L^2 (\h^3)} + \norm{ u}_{L_t^q L_x^r (J_k^\ell \times \h^3)} \varepsilon^{\frac{4}{3}} .
\end{align}
Then by choosing $\varepsilon$ small enough, a continuity argument gives
\begin{align}
\norm{u}_{L_t^q L_x^r (J_k^\ell \times \h^3)} & \lesssim \norm{u(a_\ell)}_{L^2 (\h^3)} \lesssim \norm{u}_{L_t^{\infty}L_x^2 (J_k^\ell \times \h^3)} .
\end{align}
Adding $J_k^\ell$, we obtain the claimed estimate in \eqref{eq S^0}.

To estimate the $S_*^0$ norm in \eqref{eq 3}, we only need to focus on the second term in the $S_*^0$ norm below
\begin{align}
\norm{P_{> \eta N} u}_{S_*^0 (J_k \times \h^3)} = \norm{P_{> \eta N}u_0}_{L^2 (\h^3)} + \norm{P_{> \eta N}  F(u)}_{L_t^1 L_x^2 (J_k \times \h^3)} ,
\end{align}
where $F(u) = \abs{u}^{\frac{4}{3}} u$.

First, we decompose 
\begin{align}\label{eq 7}
\begin{aligned}
\norm{P_{> \eta N}  F(u)}_{L_t^1 L_x^2 (J_k \times \h^3)} & \leq \norm{P_{> \eta N}  F(P_{< \eta N^{\frac{1}{2}}} u)}_{L_t^1 L_x^2 (J_k \times \h^3)} \\
& \quad + \norm{P_{> \eta N}  (F(u) - F(P_{< \eta N^{\frac{1}{2}}} u))}_{L_t^1 L_x^2 (J_k \times \h^3)} .
\end{aligned}
\end{align}
Then for the first term in \eqref{eq 7}, we write using Bernstein inequality, H\"older inequality and \eqref{eq S^0}
\begin{align}
\norm{P_{> \eta N}  F(P_{< \eta N^{\frac{1}{2}}} u)}_{L_t^1 L_x^2 (J_k \times \h^3)}  & \lesssim \frac{1}{\eta N} \norm{\nabla F(P_{< \eta N^{\frac{1}{2}}} u)}_{L_t^1 L_x^2 (J_k \times \h^3)} \\
& \lesssim \frac{1}{\eta N} \norm{\nabla P_{< \eta N^{\frac{1}{2}}} u }_{S^0 (J_k)} \norm{u}_{S^0 (J_k)}^{\frac{4}{3}} \lesssim \frac{1}{\eta N} \eta N^{\frac{1}{2}} =  N^{-\frac{1}{2}} \label{eq 10} ,
\end{align}
and for the second term in \eqref{eq 7}, using H\"older inequality and \eqref{eq S^0}, we have
\begin{align}
\norm{P_{> \eta N}  (F(u) - F(P_{< \eta N^{\frac{1}{2}}} u))}_{L_t^1 L_x^2 (J_k \times \h^3)} & \lesssim  \norm{P_{> \eta N^{\frac{1}{2}}} u }_{L_t^{\infty} L_x^2 (J_k \times \h^3)}^{\frac{1}{3}} \norm{u }_{L_t^2 L_x^6 (J_k \times \h^3)}^{2} \\
& \lesssim \norm{P_{> \eta N^{\frac{1}{2}}} u }_{L_t^{\infty} L_x^2 (J_k \times \h^3)}^{\frac{1}{3}}. \label{eq 11}
\end{align}

Now combining \eqref{eq 10} and \eqref{eq 11}, we get
\begin{align}
\eqref{eq 7} & = \norm{P_{> \eta N}  F(P_{< \eta N^{\frac{1}{2}}} u)}_{L_t^1 L_x^2 (J_k \times \h^3)} + \norm{P_{> \eta N}  (F(u) - F(P_{< \eta N^{\frac{1}{2}}} u))}_{L_t^1 L_x^2 (J_k \times \h^3)} \\
& \lesssim N^{-\frac{1}{2}} + \norm{P_{> \eta N^{\frac{1}{2}}} u }_{L_t^{\infty} L_x^2 (J_k \times \h^3)}^{\frac{1}{3}} .
\end{align}
Therefore, we obtain the following control of the $S_*^0$ norm in \eqref{eq 3}, that is,
\begin{align}
\norm{P_{> \eta N} u}_{S_*^0 (J_k \times \h^3)} & = \norm{P_{> \eta N}u_0}_{L_x^2 (\h^3)} + \norm{P_{> \eta N}  (\abs{u}^{\frac{4}{3}} u)}_{L_t^1 L_x^2 (J_k \times \h^3)} \\
& \lesssim \norm{P_{> \eta N} u }_{L_t^{\infty} L_x^2 (J_k \times \h^3)} + N^{-\frac{1}{2}} + \norm{P_{> \eta N^{\frac{1}{2}}} u }_{L_t^{\infty} L_x^2(J_k \times \h^3)}^{\frac{1}{3}} .
\end{align}

Finally, let $C_0C_1^2 =: c^2_{\eta}$, then
\begin{align}
\eqref{eq 3} & = \frac{C_0C_1^2}{\eta} \frac{T}{N} \sup_{J_k \subset I} \norm{P_{>\eta N} u}^2_{S_*^0 (J_k \times \h^3)}  \\
& \lesssim  \frac{c^2_{\eta}T}{\eta N} \left( \norm{P_{> \eta N} u }_{L_t^{\infty} L_x^2 (I \times \h^3)} +  N^{-\frac{1}{2}} + \norm{P_{> \eta N^{\frac{1}{2}}} u }_{L_t^{\infty} L_x^2(I \times \h^3)}^{\frac{1}{3}}\right)^2 .
\end{align}
which completes the estimate for the last term in \eqref{eq 4}. 

Now putting all the terms above, we obtain Lemma \ref{claim 1}.
\end{proof}

\begin{lemma}\label{lem LTS}
    Let $u \in C(\mathbb{R}; L^2(\h^3))$ be an almost periodic (modulo $\mathbb{G}$) solution to \eqref{NLS} and let $T \in (0,\infty)$. 
    \begin{enumerate}
        \item 
        We have the following so-called long-time Strichartz estimates
        \begin{align}\label{eq LTS1}
            \left\|P_{\geq N}u\right\|_{L_t^2 L_x^6([0, T] \times \h^3)} \lesssim 1+ \frac{T^{\frac{1}{2}}}{N^{\frac{1}{2}}}.
        \end{align}

        \item 
        For any $\varepsilon >0$, there exists $N_0(\varepsilon) > 0$ so that for $N \geq N_0$,
\begin{align}\label{eq LTS2}
\left\|P_{\geq N}u\right\|_{L_t^2 L_x^6([0, T] \times \h^3)} \lesssim \varepsilon(1+ \frac{T^{\frac{1}{2}}}{N^{\frac{1}{2}}} ).
\end{align}
        \item 
        The $L_t^2 L_x^6$ norm in \eqref{eq LTS1} and \eqref{eq LTS2} can be replaced by $L_t^2 L_x^q$ norm with $q \in (2, 6]$.
    \end{enumerate}
\end{lemma}

\begin{proof}[Proof of Lemma \ref{lem LTS}]
The proof of this lemma follows the main idea in \cite{Dod3d}.
We start from the proof of (1). 

Fix $T$ and let 
\begin{align}
f(N) : = \norm{P_{\geq N} u}_{L_t^2 L_x^6([0, T] \times \h^3)}.
\end{align}
By Lemma \ref{claim 1}, 
\begin{align}
f(N) \lesssim  1 +  \sum_{M \leq \eta N} (\frac{M}{N}) f(M) +  \frac{c_{\eta}}{\eta^{\frac{1}{2}}} (\frac{T}{N})^{\frac{1}{2}} . 
\end{align}
Let 
\begin{align}
c = \sup_N \frac{f(N)}{1 + (\frac{T}{N})^{\frac{1}{2}} } 
\end{align}
taking the supremum over all dyadic integers $N$. Then
\begin{align}
f(N) & \lesssim  1 + c \sum_{M \leq \eta N} (\frac{M}{N}) \parenthese{ 1 + (\frac{T}{M})^{\frac{1}{2}} } + \frac{c_{\eta}}{\eta^{\frac{1}{2}}} (\frac{T}{N})^{\frac{1}{2}} \\
& \lesssim  1 + c \eta^{\frac{1}{2}}(\frac{T}{N})^{\frac{1}{2}} + c \eta + \frac{c_{\eta}}{\eta^{\frac{1}{2}}} (\frac{T}{N})^{\frac{1}{2}} .
\end{align}
Fixing $\eta_0(u) >0$ sufficiently small, we can take
\begin{align}
c \lesssim  \frac{c_{\eta_0}}{\eta_0^{\frac{1}{2}}} .
\end{align}

Now (2) can be shown with the following argument: Consider 
\begin{align}
f(N) & \lesssim \inf_{t_0 \in I} \norm{P_{\geq N} u(t_0)}_{L_x^2} + \sum_{M \leq \eta N} (\frac{M}{N})^s f(M)  + \eta^{\frac{4}{3}} \norm{P_{\geq \eta N} u}_{L_t^2 L_x^6} \\
& \quad + \frac{c_{\eta} T^{\frac{1}{2}}}{\eta^{\frac{1}{2}} N^{\frac{1}{2}}}  \left( \norm{P_{> \eta N} u }_{L_t^{\infty} L_x^2 } +  N^{-\frac{1}{2}} + \norm{P_{> \eta N^{\frac{1}{2}}} u }_{L_t^{\infty} L_x^2}^{\frac{1}{3}}\right).
\end{align}
Using the following equation (Lemma \ref{lem etaSmall})
\begin{align}
\lim_{N \to \infty}  \norm{P_{> \eta N} u }_{L_t^{\infty} L_x^2(I \times \h^3)} +  N^{-\frac{1}{2}} + \norm{P_{> \eta N^{\frac{1}{2}}} u }_{L_t^{\infty} L_x^2(I \times \h^3)}^{\frac{1}{3}} = 0
\end{align}
and part (1), we can write
\begin{align}
f(N) & \lesssim \inf_{t_0 \in I} \norm{P_{\geq N} u(t_0)}_{L_x^2} + \sum_{M \leq \eta N} (\frac{M}{N})^s f(M)  + \eta^{\frac{4}{3}} \norm{P_{\geq \eta N} u}_{L_t^2 L_x^6} \\
& \quad + \frac{c_{\eta} T^{\frac{1}{2}}}{\eta^{\frac{1}{2}} N^{\frac{1}{2}}}  \left( \norm{P_{> \eta N} u }_{L_t^{\infty} L_x^2 } +  N^{-\frac{1}{2}} + \norm{P_{> \eta N^{\frac{1}{2}}} u }_{L_t^{\infty} L_x^2}^{\frac{1}{3}}\right)  \\
& \lesssim \inf_{t_0 \in I} \norm{P_{\geq N} u(t_0)}_{L_x^2} + \sum_{M \leq \eta N} (\frac{M}{N})^s \parenthese{1 + \frac{T^{\frac{1}{2}}}{M^{\frac{1}{2}}} } + \eta^{\frac{4}{3}} \parenthese{1 + \frac{T^{\frac{1}{2}}}{\eta^{\frac{1}{2}} N^{\frac{1}{2}}} } \\
& \quad + \frac{c_{\eta} T^{\frac{1}{2}}}{\eta^{\frac{1}{2}} N^{\frac{1}{2}}}  \left( \norm{P_{> \eta N} u }_{L_t^{\infty} L_x^2 } +  N^{-\frac{1}{2}} + \norm{P_{> \eta N^{\frac{1}{2}}} u }_{L_t^{\infty} L_x^2}^{\frac{1}{3}}\right) .
\end{align}

Using Lemma \ref{lem etaSmall} and choosing  $\eta$ small enough, then $N$ large enough we conclude
\begin{align}
\left\|P_{\geq N}u\right\|_{L_t^2 L_x^6([0, T] \times \h^3)} \lesssim \varepsilon(1+ \frac{ T^{1 / 2}}{ N^{1 / 2}} ).
\end{align}

Now we turn to part (3). 
In fact, the $L_t^2 L_x^6$ norm on the left-hand-side of inequalities \eqref{eq LTS1} and \eqref{eq LTS2} is not essential, since we only used the $L^2$ integrability in time. Thus such norm can be replaced by any admissible $L_t^2 L_x^q$ norm ($q \in (2,6]$), thanks to a larger range of admissible pairs (compared to the corresponding Euclidean case) as shown in Figure \ref{figure}. That is,
\begin{align}
\norm{P_{\geq N}u}_{L_t^2 L_x^q (J\times \h^3)}^2 \lesssim 1 + \frac{T}{N} .
\end{align}
Combining with part (2), we have for any $\varepsilon >0$, there exists $N_0(\varepsilon) > 0$ so that for $N \geq N_0$,
\begin{align}
\left\|P_{\geq N}u\right\|_{L_t^2 L_x^q([0, T] \times \h^3)}^2 \lesssim \varepsilon(1+ \frac{T}{ N} ).
\end{align}
which completes part (3).

Then we finish the proof of Lemma \ref{lem LTS}.
\end{proof}

\begin{lemma}\label{lem gradlow}
    Let $u \in C(\mathbb{R}; L^2(\h^3))$ be an almost periodic (modulo $\mathbb{G}$) solution to \eqref{NLS}, let $s\in (\frac{1}{2}, 1]$, and let $T \in (0,\infty)$. Then for any $q \in (2,6]$
        \begin{align*}
            \norm{\abs{\nabla}^s P_{\leq N}u}_{L_t^2 L_x^q([0, T] \times \h^3)} \lesssim  N^{s-\frac{1}{2}}T^{\frac{1}{2}} .
        \end{align*}
\end{lemma}

\begin{proof}[Proof of Lemma \ref{lem gradlow}]
For $\frac{1}{2} < s \leq 1$ and $2 < q \leq 6$, using Lemma \ref{lem LTS},
\begin{align}
\norm{\abs{\nabla}^s P_{\leq N} u}_{L_t^2 L_x^q ([0, T] \times \h^3)} & \lesssim_s \sum_{M \leq N} M^s \norm{ P_{M} u}_{L_t^2 L_x^q ([0, T] \times \h^3)} \\
& \lesssim_s \sum_{M \leq N} M^s \left((\frac{T}{M})^{\frac{1}{2}} +1 \right) \\
&\lesssim_s N^{(s-\frac{1}{2})}T^{\frac{1}{2}}+ N^s = N^s \left( (\frac{T}{N})^{\frac{1}{2}} +1 \right) .
\end{align}
In particular, we have
\begin{align}
\norm{\nabla P_{\leq T} u}_{L_t^2 L_x^q ([0, T] \times \h^3)} \lesssim T.
\end{align}
\end{proof}

We are now prepared to prove Proposition \ref{prop MorDod}.

\begin{proof}[Proof of Proposition \ref{prop MorDod}]
    Let $0<\eta\ll \norm{u}_{L^2(\h^3)}$ be small enough so that $\mbox{Vol}_{\h^3}(B_{\0}(C(\eta))\geq 1$ and 
        \begin{align*}
            \tfrac{1}{2}\norm{u}^{6/5}_{L^2(\h^3)} \leq \left(\int_{d(x, \0) \geq C(\eta)}|(\pi_{h(t)}u)(t, x)|^2 d \mu(x) \right)^{3/5}
        \end{align*}
    for all $t \in \R$. Then by H\"older inequality, 
        \begin{align*}
            C_{\h^3}\norm{u}^{6/5}_{L^2(\h^3)}  \leq \norm{u(t)}_{L_x^{10/3}(\h^3)}^{10/3} .
        \end{align*}
    Therefore,
        \begin{align*}
          C T  \leq \norm{u}_{L_{t, x}^{10/3}([0, T]\times\h^3)}^{10/3} 
        \end{align*}
    for some $C= C(\h^3, \|u\|_{L^2(\h^3)}).$ Moreover, due to almost periodicity, for $T$ large enough (depending on $\eta$).
        \begin{align*}
          C T  \leq \norm{P_{\leq T}u}_{L_{t, x}^{10/3}([0, T]\times\h^3)}^{10/3} .
        \end{align*}

    Now $P_{\leq N}u$ solves \eqref{eq modNLS} with $\NN$ defined by
        \begin{align*}
            \NN&= P_{\leq T}(|u|^{\frac{4}{3}}u)-\abs{P_{\leq T} u}^{\frac{4}{3}}(P_{\leq T}u) .
        \end{align*}
    Then by Proposition \ref{prop Morawetz},
        \begin{align}
            \norm{P_{\leq T} u}_{L_{t, x}^{10/3}([0, T]\times\h^3)}^{10/3} &\lesssim  \norm{P_{\leq T} u}_{L_t^{\infty} L_x^2 ([0,T] \times \h^3)} \norm{P_{\leq T} u}_{L_t^{\infty} H_x^1 ([0,T] \times \h^3)} \label{eq Mor1}\\
            &\hspace{1cm}+ \norm{\NN\overline{P_{\leq T} u}}_{L_{t,x}^1 ([0,T] \times \h^3)} + \norm{\NN \nabla \overline{ P_{\leq T} u}}_{L_{t,x}^1 ([0,T] \times \h^3)} . \label{eq Mor2}
        \end{align}
    Let $T\gtrsim \|u\|_{L^2(\h^3)}\eta^{-1}C(\eta)$, then $\|P_{\geq T}u\|_{L^2(\h^3)}\leq \|P_{\geq C(\eta)}u\|_{L^2(\h^3)}\leq \eta $ and for all $t \in [0, T]$,
        \begin{align}\label{eq gradest}
            \norm{\nabla P_{\leq T} u(t) }_{L^2(\h^3)} &\leq \norm{\nabla P_{\geq C(\eta)} P_{\leq T} u(t) }_{L^2(\h^3)}+\norm{\nabla P_{\leq C(\eta)} P_{\leq T} u(t) }_{L^2(\h^3)}\\
            &\lesssim T\norm{ P_{\geq C(\eta)}  u(t) }_{L^2(\h^3)}+\norm{\nabla P_{\leq C(\eta)} u(t) }_{L^2(\h^3)}\\
            &\lesssim \eta T+C(\eta) \|u\|_{L^2(\h^3)}\\
            &\lesssim \eta T
        \end{align}
which implies 
\begin{align}
\eqref{eq Mor1} \lesssim \eta T .
\end{align}

To estimate the terms in \eqref{eq Mor2}, we write
\begin{align}
\NN = P_{\leq T} F(u) - F(P_{\leq T} u) & = [F(u) - F(P_{\leq T} u)] - P_{>T} F(u) \\
& = [F(u) - F(P_{\leq T} u)] - P_{>T} F(P_{< \varepsilon T} u) - P_{> T} (F(u) -F(P_{< \varepsilon T} u) ) .
\end{align}

For $T > N_0 (\varepsilon)$, using \eqref{eq LTS2}
\begin{align}
\norm{F(u) - F(P_{\leq T} u)}_{L_t^2 L_x^{\frac{6}{5}}} \lesssim \norm{P_{> T} u}_{L_t^2 L_x^6} \norm{u}_{L_t^{\infty} L_x^2}^{\frac{4}{3}} \lesssim \varepsilon .
\end{align}
Using Lemma \ref{lem gradlow}
\begin{align}
\norm{P_{>T} F(P_{< \varepsilon T} u)}_{L_t^2 L_x^{\frac{6}{5}}} & \lesssim \frac{1}{T}\norm{\nabla P_{>T} F(P_{< \varepsilon T} u)}_{L_t^2 L_x^{\frac{6}{5}}} \lesssim \frac{1}{T} \norm{\nabla P_{< \varepsilon T} u}_{L_t^2 L_x^6} \norm{u}_{L_t^{\infty} L_x^2}^{\frac{4}{3}} \\
& \lesssim \frac{1}{T} (\varepsilon T)^{1/2} T^{1/2} = \varepsilon^{1/2} .
\end{align}

For $\varepsilon T > N_0$, using \eqref{eq LTS2}
\begin{align}
\norm{P_{> T} (F(u) -F(P_{< \varepsilon T} u) ) }_{L_t^2 L_x^{\frac{6}{5}}} \lesssim \norm{P_{> \varepsilon T} u}_{L_t^2 L_x^6} \norm{u}_{L_t^{\infty} L_x^2}^{\frac{4}{3}} \lesssim \varepsilon (1 + \frac{T^{1/2}}{\varepsilon^{1/2} T^{1/2}}) \lesssim \varepsilon^{1/2} .
\end{align}

Then we conclude
\begin{align}
\norm{P_{\leq T} F(u) - F(P_{\leq T} u) }_{L_t^2 L_x^{\frac{6}{5}}} \lesssim \varepsilon^{1/2} .
\end{align}
In fact, we also have
\begin{align}\label{eq 8}
\norm{P_{\leq T} F(u) - F(P_{\leq T} u) }_{L_t^2 L_x^{\frac{6}{5}-}} \lesssim \varepsilon^{1/2} .
\end{align}
where we just need to replace all the $L_t^2 L_x^6$ norm above by $L_t^2 L_x^{6-}$ and combine with part (3) in Lemma \ref{lem LTS}.

By Lemma \ref{lem gradlow}, we have
\begin{align}
\norm{\nabla P_{\leq T} u}_{L_t^2 L_x^6} \lesssim T .
\end{align}
Hence the second error term has the following bound
\begin{align}
\norm{[P_{\leq T} F(u) - F(P_{\leq T} u)] \overline{\nabla P_{\leq T} u}}_{L_{t,x}^1}  \lesssim \norm{P_{\leq T} F(u) - F(P_{\leq T} u) }_{L_t^2 L_x^{\frac{6}{5}}} \norm{\nabla P_{\leq T} u}_{L_t^2 L_x^6} \lesssim \varepsilon^{1/2} T .
\end{align}

For the other error term, using Sobolev embedding, Lemma \ref{lem gradlow} and \eqref{eq 8}
\begin{align}
\norm{[P_{\leq T} F(u) - F(P_{\leq T} u)] \overline{P_{\leq T} u}}_{L_{t,x}^1} & \lesssim \norm{P_{\leq T} F(u) - F(P_{\leq T} u) }_{L_t^2 L_x^{\frac{6}{5}-}} \norm{ P_{\leq T} u}_{L_t^2 L_x^{6+}} \\
& \lesssim \norm{P_{\leq T} F(u) - F(P_{\leq T} u) }_{L_t^2 L_x^{\frac{6}{5}-}} \norm{\nabla P_{\leq T} u}_{L_t^2 L_x^{2+}} \\
& \lesssim \varepsilon^{1/2} T^{1/2} T^{1/2} = \varepsilon^{1/2} T .
\end{align}

Therefore, we obtain
\begin{align}
\eqref{eq Mor2} \lesssim \varepsilon^{1/2} T .
\end{align}
and thus
\begin{align*}
\norm{P_{\leq T} u}_{L_{t, x}^{10/3}([0, T]\times\h^3)}^{10/3}\lesssim (\eta + \varepsilon^{1/2}) T.
\end{align*}

Assuming $\varepsilon< \eta^2$ completes the argument in Proposition \ref{prop MorDod}.
\end{proof}

\bibliographystyle{plain}
\bibliography{bibfile}

\end{document}